\documentclass{louloupart1}
\usepackage{pinlabel}
\usepackage{graphicx}
\usepackage[all]{xy}
\usepackage{tikz}
\usetikzlibrary{decorations.markings,arrows,arrows.meta,positioning}
\tikzset{->-/.style={decoration={
  markings,
  mark=at position #1 with {\arrow{Computer Modern Rightarrow[length=5pt,width=5pt]}}},postaction={decorate}}}
\tikzset{->-rev/.style={decoration={
  markings,
  mark=at position #1 with {\arrow{Computer Modern Rightarrow[length=5pt,width=5pt,reversed]}}},postaction={decorate}}}

\title[Trickle groups]{Trickle groups}
\author[P Bellingeri]{Paolo Bellingeri}
\givenname{Paolo}
\surname{Bellingeri}
\address{Paolo Bellingeri, Laboratoire Nicolas Oresme, UMR 6139, CNRS, Université de Caen-Normandie, 14032 Caen Cedex, France
\vskip 3 pt
Eddy Godelle, Laboratoire Nicolas Oresme, UMR 6139, CNRS, Université de Caen-Normandie, 14032 Caen Cedex, France
\vskip 3 pt
Luis Paris, Institut de Mathématiques de Bourgogne, UMR 5584, CNRS, Université de Bourgogne, 21000 Dijon, France}
\email{paolo.bellingeri@unicaen.fr}
\author[E Godelle]{Eddy Godelle}
\givenname{Eddy}
\surname{Godelle}
\email{eddy.godelle@unicaen.fr}
\author[L Paris]{Luis Paris}
\givenname{Luis}
\surname{Paris}
\email{lparis@u-bourgogne.fr}

\newtheorem{thm}{Theorem}[section]
\newtheorem{lem}[thm]{Lemma}
\newtheorem{prop}[thm]{Proposition}
\newtheorem{corl}[thm]{Corollary}

\theoremstyle{definition}
\newtheorem*{defin}{Definition}
\newtheorem*{rem}{Remark}
\newtheorem*{expl}{Example}
\newtheorem*{acknow}{Acknowledgments}
\newtheorem*{expl1}{Example 1}
\newtheorem*{expl2}{Example 2}
\newtheorem*{expl3}{Example 3}

\numberwithin{equation}{section}

\makeatletter
\renewcommand{\thefigure}{\ifnum \c@section>\z@ \thesection.\fi
 \@arabic\c@figure}
\@addtoreset{figure}{section}
\makeatother

\begin{document}

\def\N{\mathbb N} \def\R{\mathbb R} \def\SSS{\mathfrak S}
\def\SS{\mathcal S} \def\id{{\rm id}} \def\VJ{{\rm VJ}}
\def\KJ{{\rm KJ}} \def\link{{\rm link}} \def\starE{{\rm star}}
\def\Tr{{\rm Tr}} \def\Z{\mathbb Z} \def\supp{{\rm supp}}
\def\st{{\rm st}} \def\pil{{\rm pil}} \def\RR{\mathcal R}
\def\LL{\mathcal L} \def\nf{{\rm nf}} \def\syl{{\rm syl}}
\def\FF{\mathcal F} \def\AC{{\rm AC}} \def\Div{{\rm Div}}
\def\Ker{{\rm Ker}} \def\KVJ{{\rm KVJ}} \def\PVJ{{\rm PVJ}}
\def\SF{{\rm SF}} \def\PP{\mathcal P} \def\fin{{\rm fi}}
\def\linkE{{\rm link}} \def\Aut{{\rm Aut}} \def\AA{\mathcal A}


\begin{abstract}
A new family of groups, called trickle groups, is presented.
These groups generalize right-angled Artin and Coxeter groups, as well as cactus groups.
A trickle group is defined by a presentation with relations of the form $xy = zx$ and $x^\mu = 1$, that are governed by a simplicial graph, called a trickle graph, endowed with a partial ordering on the vertices, a vertex labeling, and an automorphism of the star of each vertex.
We show several examples of trickle groups, including extended cactus groups, certain finite-index subgroups of virtual cactus groups, Thompson group F, and ordered quandle groups.
A terminating and confluent rewriting system is established for trickle groups, enabling the definition of normal forms and a solution to the word problem.
An alternative solution to the word problem is also presented, offering a simpler formulation akin to Tits' approach for Coxeter groups and Green's for graph products of cyclic groups.
A natural notion of a parabolic subgraph of a trickle graph is introduced.
The subgroup generated by the vertices of such a subgraph is called a standard parabolic subgroup and it is shown to be the trickle group associated with the subgraph itself.
The intersection of two standard parabolic subgroups is also proven to be a standard parabolic subgroup.
If only relations of the form $xy = zx$ are retained in the definition of a trickle group, then the resulting group is called a preGarside trickle group.
Such a group is proved to be a preGarside group, a torsion-free group, and a Garside group if and only if its associated trickle graph is finite and complete.

\smallskip\noindent
{\bf AMS Subject Classification\ \ } 
Primary: 20F10, Secondary: 20F05, 20F36, 20F55, 20F65.

\smallskip\noindent
{\bf Keywords\ \ } 
Trickle groups, right-angled Coxeter groups, right-angled Artin groups, cactus groups, virtual cactus groups, Thompson group F, word problem, rewriting systems, preGarside groups, Garside groups.
\end{abstract}

\maketitle


\section{Introduction}\label{sec1}

There are numerous groups in the literature defined by relations of the form $xy = zx$, often with additional constraints on the orders of generators. 
Prominent examples include right-angled Artin groups, right-angled Coxeter groups, and more generally, graph products of cyclic groups. 
The aim of the present paper is to study a specific family of such groups that we call \emph{trickle groups}.

Cactus groups are emblematic examples of trickle groups. 
These groups first appeared as quasi-braid groups in the study of the mosaic operad \cite{Devad1,EHKR1,KhWi1} and they were subsequently generalized to all Coxeter groups \cite{DaJaSc1}. 
Their significance was further highlighted in their connection to coboundary categories \cite{HenKam1}, mirroring the role of braid groups in braided categories. 
Note that the term ``cactus groups'' was coined in \cite{HenKam1}. 
Additionally, cactus groups and their generalizations to Coxeter groups have found applications in representation theory under various guises \cite{KnTaWo1,Bonna1,Losev1,ChGlPy1,RouWhi1}.

For $n \in \N_{\ge 2}$, the \emph{cactus group} $J_n$ is defined by the presentation with generators $x_{p,q}$, $1 \le p < q \le n$, and relations:
\begin{itemize}
\item[(j1)]
$x_{p,q}^2 = 1$, for $1 \le p <q \le n$,
\item[(j2)]
$x_{p,q} x_{m,r} = x_{m,r} x_{p,q}$, for $[p,q] \cap [m,r] = \emptyset$,
\item[(j3)]
$x_{p,q} x_{m,r} = x_{p+q-r,p+q-m} x_{p,q}$, for $[m,r]\subset[p,q]$.
\end{itemize}

The elements of $J_n$ are often depicted using planar diagrams.
More precisely, an element $g \in J_n$ is represented by an $n$-tuple of smooth paths in the plane, $b=(b_1,\dots,b_n)$, $b_i : [0,1] \to \R^2$, satisfying the following conditions.
\begin{itemize}
\item
There exists a permutation $\sigma \in \SSS_n$ such that $b_i(0)=(0,i)$ and $b_i(1)=(1,\sigma(i))$, for all $i\in \{1,\dots, n\}$.
\item
For all $t\in [0,1]$ and all $i \in \{1, \dots, n\}$ we have $\pi_1(b_i(t))=t$, where $\pi_1 : \R^2 \to \R$ denotes the projection onto the first coordinate.
\item
Crossings between the $b_i$'s may be multiple but are always transversal.
\end{itemize}
The generator $x_{p,q}$ is represented in Figure \ref{fig1_1} and relation (j3) is illustrated in Figure \ref{fig1_2}.

\begin{figure}[ht!]
\begin{center}
\includegraphics[width=2.8cm]{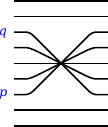}
\caption{Generator of $J_n$}\label{fig1_1}
\end{center}
\end{figure}

\begin{figure}[ht!]
\begin{center}
\includegraphics[width=8.4cm]{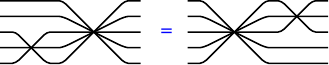}
\caption{Relation (j3) in the presentation of $J_n$}\label{fig1_2}
\end{center}
\end{figure}

The extension of this definition to Coxeter groups is straightforward.
Let $(W,S)$ be a Coxeter system associated with a Coxeter graph $\Upsilon$.
For $X \subseteq S$, the subgroup of $W$ generated by $X$ is called a \emph{standard parabolic subgroup} and is denoted by $W_X$, and the full Coxeter subgraph of $\Upsilon$ spanned by $X$ is denoted by $\Upsilon_X$.
We say that $X \subseteq S$ is \emph{irreducible} if $\Upsilon_X$ is connected, and we say that $X$ is of \emph{spherical type} if $W_X$ is finite.
In the latter case $W_X$ contains a unique element of maximal length (with respect to $S$), denoted by $w_X$, and this element satisfies $w_X X w_X^{-1}=X$ and $w_X^2=\id$ (see \cite{Bourb1}).
We denote the set of non-empty irreducible and spherical subsets of $S$ by $\SS_{\fin}$.

The \emph{cactus group} $C (W, S)$ associated with $(W, S)$ is defined by the presentation with generators $x_X$, $X \in \SS_{\fin}$, and relations:
\begin{itemize}
\item[(j1)]
$x_X^2=1$, for $X \in \SS_{\fin}$,
\item[(j2)]
$x_X x_Y = x_Y x_X$, if $\Upsilon_{X\cup Y}$ is the disjoint union of $\Upsilon_X$ and $\Upsilon_Y$, meaning that $X \cap Y = \emptyset$ and $st = ts$ for all $s \in X$ and $t \in Y$, 
\item[(j3)] 
$x_X x_Y = x_{w_X(Y)} x_X$, for $Y \subset X$.
\end{itemize}

Our aim is to investigate the combinatorial properties of these groups within a broader framework that encompasses various more or less natural generalizations of cactus groups.

Let $\Gamma$ be a simplicial graph.
We denote the vertex set of $\Gamma$ by $V (\Gamma)$ and the edge set by $E (\Gamma)$.
The \emph{link} of a vertex $x \in V(\Gamma)$, denoted by $\linkE_x (\Gamma)$, is the full subgraph of $\Gamma$ spanned by $\{y \in V (\Gamma) \mid \{x, y\} \in E(\Gamma)\}$.
The \emph{star} of $x$, denoted by $\starE_x (\Gamma)$, is the full subgraph of $\Gamma$ spanned by $\{y\in V (\Gamma) \mid \{x, y\} \in E(\Gamma)\}\cup \{x\}$.

A \emph{trickle graph} is a quadruple $(\Gamma, \le, \mu, (\varphi_x)_{x\in V(\Gamma)})$, where:
\begin{itemize}
\item
$\Gamma$ is a simplicial graph,
\item
$\le$ is a (partial) order on $V (\Gamma)$,
\item
$\mu: V(\Gamma) \to \N_{\ge 2} \cup \{\infty\}$ is a vertex labeling,
\item
$\varphi_x: \starE_x (\Gamma) \to \starE_x (\Gamma)$ is an automorphism of $\starE_x (\Gamma)$ for all $x \in V(\Gamma)$,
\end{itemize}
that must satisfy certain conditions defined in Section \ref{sec2}.

The \emph{trickle group} $\Tr (\Gamma)$ associated with a trickle graph $\Gamma = (\Gamma, \le, \mu,(\varphi_x)_{x\in V(\Gamma)})$ is the group defined by the following presentation.
\begin{gather*}
\Tr (\Gamma) = \langle V(\Gamma) \mid x^{\mu(x)} = 1 \text{ for all } x \in V (\Gamma) \text{ such that } \mu (x) \neq \infty\,,\ \varphi_x(y)\,x = \varphi_y(x)\,y \\
\text{ for all } \{x, y\} \in E(\Gamma) \rangle\,.
\end{gather*}

One of the conditions in the definition of a trickle graph in Section \ref{sec2} entails that, for any edge $\{x,y\} \in E (\Gamma)$, either $\varphi_x (y) = y$ or $\varphi_y (x) = x$.
Consequently, the above presentation is indeed a presentation with relations of the form $xy=zx$ and $x^\mu=1$.
Another immediate consequence of these conditions is that the trivial order is admissible, and therefore right-angled Coxeter groups, right-angled Artin groups and, more generally, graph products of cyclic groups are trickle groups.

As mentioned above, cactus groups are trickle groups, where here $\mu(x) = 2$ for all $x \in V (\Gamma)$.
Trickle groups include other groups naturally related to cactus groups, such as the ``Artin'' versions of cactus groups (associated with Coxeter groups), where we keep relations (j2) and (j3) and ignore relations (j1).
In this case, such a group is also a preGarside group in the sense of \cite{GodPar2} (see Section \ref{sec7} and Subsection \ref{subsec2_5}).
More generally, the example of cactus groups associated with Coxeter systems can be naturally extended to families of subgroups of a given group $G$ satisfying certain properties that are presented in Subsection \ref{subsec3_1}.

Our study is also an opportunity to investigate \emph{virtual cactus groups} since, as we will see in Subsection \ref{subsec3_2}, they contain finite index subgroups that are trickle groups.
Let $S_1, \dots, S_\ell$ be a collection of $\ell$ circles immersed in the plane having only double transverse crossings.
We assign to each crossing a ``positive'', ``negative'' or ``virtual'' value that we indicate on the graphical representation of $S_1 \cup \cdots \cup S_\ell$ as in Figure \ref{fig1_3}.
Such a figure is called a \emph{virtual link diagram}.
We consider the equivalence relation on the set of virtual link diagrams generated by the isotopy and the virtual Reidemeister moves as defined by Kauffman \cite{Kauff1,Kauff2}.
An equivalence class of virtual link diagrams is a \emph{virtual link}.

\begin{figure}[ht!]
\begin{center}
\includegraphics[width=4.4cm]{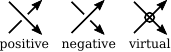}
\caption{Crossings in a virtual link diagram}\label{fig1_3}
\end{center}
\end{figure}

Since the publication of Kauffman's seminal paper \cite{Kauff1}, the theory of virtual knots and links has grown significantly, and this notion has been extended to other combinatorial and/or topological objects represented by planar diagrams. 
The leitmotif underlying these theories is that two arcs connecting the same points and passing only through virtual crossings are equivalent.

Since the elements of the cactus group $J_n$ are represented by planar diagrams, it is natural to extend $J_n$ by adding virtual crossings to the cactus crossings while keeping the principle that two arcs connecting the same points and passing only through virtual crossings are equivalent.
Then we obtain the \emph{virtual cactus group}, $\VJ_n$, which will be studied in detail in Subsection \ref{subsec3_2}.
Note that these groups are not new, having been introduced in \cite{IKLPR1}, where it is shown that $\VJ_n$ is the $\SSS_n$-equivariant fundamental group of the real form of the ``cactus flower moduli space'' $\bar F_n$.

The connection between virtual cactus groups and trickle groups mirrors that between virtual braid groups and Artin groups \cite{GodPar1, BeCiPa1, BelPar1, BePaTh1}.
We prove that $\VJ_n$ can be decomposed as a semi-direct product $\VJ_n = \KJ_n \rtimes \SSS_n$, where $\SSS_n$ is the symmetric group on $\{1,\dots, n\}$, and $\KJ_n$ is a trickle group (see Proposition \ref{prop3_7}).
This decomposition enables us to address the word problem in $\VJ_n$, to define normal forms for the elements of $\VJ_n$, and to show that $J_n$ embeds into $\VJ_n$.

It is probable that numerous other trickle groups exist within the literature.
Among these, we have pinpointed two specific examples: one originating from dynamical systems and the other from knot theory.

Thompson group $F$, introduced by Richard Thompson in 1965 in an unpublished manuscript, is a group of homeomorphisms of the real line with many unusual properties.
For instance, its derived group $F'$ is a simple group, the quotient $F/F'$ is a free abelian group of rank $2$, and it contains no subgroup isomorphic to the rank $2$ free group.
We refer to \cite{CaFlPa1} for a general overview on this group.
A well-known presentation for $F$ is as follows:
\[
\langle x_n\,,\ n\in \N \mid x_k x_n = x_{n+1} x_k \text{ for } k < n \rangle\,,
\]
which, while featuring relations of the form $xy=zx$, does not fulfill all the requirements of a trickle group presentation.
However, in Subsection \ref{subsec3_3} we show another presentation for $F$ that is indeed a trickle presentation (see Theorem \ref{thm3_15}).
So, Thompson group $F$ is a trickle group.
This structure is particularly noteworthy as any ``natural standard parabolic subgroup'' of $F$ is a copy of $F$ itself (see Proposition \ref{prop3_16}).
Moreover, Theorem \ref{thm2_14} will imply that $F$ is a preGarside group.

Quandles are algebraic structures whose axioms reproduce the Reidemeister moves in knot theory.
Independently introduced by Joyce \cite{Joyce1} and Matveev \cite{Matve1}, they have been frequently used to construct knot or link invariants.
As noted by Joyce \cite{Joyce1} and Matveev \cite{Matve1}, a classification of quandles would de facto entail a classification of knots.
This explains the difficulty of studying the set of all quandles, leading researchers to focus on specific families of quandles.
Ordered quandles were introduced recently in this perspective \cite{BaPaSi1,DDHPV1}, and it turns out that they are a disguised form of trickle graphs.
The details of this construction are given in Subsection \ref{subsec3_4}.

As previously mentioned, the objective of this paper is a combinatorial study of trickle groups.
In Section \ref{sec4} we determine a confluent and terminating rewriting system for these groups (see Theorem \ref{thm2_4}).
According to Newman \cite{Newma1}, this  enables the definition of normal forms for their elements and, consequently,  a solution to the word problem (see Corollary \ref{corl2_5}).
Moreover, it yields other immediate results, such as a characterization of finite trickle groups (see Corollary \ref{corl2_7}).
Notice that our rewriting system and its associated normal forms are not novel for graph products of cyclic groups \cite{Wyk1,HerMei1,CrGoWi1}, and a similar rewriting system for classical cactus groups was considered in \cite{Genev1}.
Note also that the term ``trickle'', used to designate trickle groups, originates from this algorithm because it metaphorically involves pushing down as many syllables as possible using only relations of the form $xy = zx$.

The remaining sections of the paper explore how trickle groups behave in a manner analogous to groups studied in the theory of Coxeter, Artin, and Garside groups.
Building upon the algorithms and methods introduced in Section \ref{sec4}, we present in Section \ref{sec5} an alternative algorithm for solving the word problem in a trickle group. 
This new algorithm offers a simpler formulation compared to the one presented in Section \ref{sec4}. 
Furthermore, it aligns more closely with the algorithms described in \cite{Tits1} for Coxeter groups, in \cite{Green1} for graph products of cyclic groups, and more generally, in \cite{ParSoe1} for Dyer groups.

A natural notion of a \emph{parabolic subgraph} emerges in the context of trickle graphs.
This leads to the natural question of studying the trickle groups defined by such subgraphs.
Drawing a parallel with the theory of Coxeter and Artin groups, we refer to these as \emph{standard parabolic subgroups}.
As the name suggests, we establish in Section \ref{sec6} that a standard parabolic subgroup is indeed a subgroup of the trickle group associated with the original graph (see Theorem \ref{thm2_10}). 
Moreover, we prove that the intersection of two standard parabolic subgroups is a standard parabolic subgroup (see Corollary \ref{corl2_12}).

Section \ref{sec7} establishes a connection between trickle groups and Garside theory.
A \emph{preGarside trickle graph} is defined as a trickle graph $\Gamma = (\Gamma, \le, \mu, (\varphi_x)_{x \in V (\Gamma)})$ for which $\mu (x) = \infty$ for all $x \in V (\Gamma)$.
A \emph{preGarside trickle group} is a trickle group associated with a preGarside trickle graph.
Additionally, one can define an associated \emph{preGarside trickle monoid} using the same presentation, but interpreted as a monoid presentation.
The term ``preGarside'' originates from \cite{GodPar2}, where the authors investigate monoids and groups that they call preGarside monoids and preGarside groups.
Notable examples of such monoids and groups include all Artin monoids and all Artin groups.
Garside monoids and Garside groups are also preGarside monoids and preGarside groups.
In Section \ref{sec7}, we prove that a preGarside trickle monoid is indeed a preGarside monoid and that a preGarside trickle group is a preGarside group (see Theorem \ref{thm2_14}).
Furthermore, the monoid and the group are respectively a Garside monoid and a Garside group if and only if $\Gamma$ is finite and complete (see Theorem \ref{thm2_15}).
This introduces new examples of Garside groups to which we can associate Coxeter-style quotients.

Another objective of Section \ref{sec7} is to address certain questions that arise for preGarside monoids and  groups within the context of preGarside trickle monoids and groups.
First, we already know from Section \ref{sec4} that a preGarisde trickle group has a solution to the word problem.
Then, we prove in Section \ref{sec7} that, if the vertex set of the graph is finite, then the preGarside trickle group is torsion-free (see Theorem \ref{thm2_16}).
The remaining questions concern the relationship between monoids and groups.
These inquiries, posed in \cite{GodPar2}, are specifically focused on preGarside monoids and groups.
We prove that a preGarside trickle monoid embeds into its enveloping group (see Theorem \ref{thm2_17}) and we prove several results concerning its parabolic submonoids and subgroups (see Theorem \ref{thm2_18}).

Trickle groups appear to be a reasonable generalization of graph products of cyclic groups, and consequently, of right-angled Artin groups and right-angled Coxeter groups.
Therefore, it is natural to explore whether results known for some or all graph products of cyclic groups can be extended to some or all trickle groups.
Relevant questions in this direction include:
\begin{itemize}
\item[(1)]
Do the normal forms described in Section \ref{sec4} form a regular language? Are trickle groups automatic or bi-automatic?
\item[(2)]
Which trickle groups admit geometric actions on CAT(0) cube complexes?
\item[(3)]
Are trickle groups residually finite? Are preGarside trickle groups residually nilpotent without torsion? Can we determine the Lie algebra associated with their lower central series, as done for right-angled Artin groups (see \cite{DucKro1})?
\item[(4)]
Which preGarside trickle groups are orderable, bi-orderable, or admit isolated orders?
\end{itemize}
Additionally, it would be valuable to discover new examples of trickle groups, particularly those arising from areas of mathematics beyond group theory.

The paper is organized as follows.
Section \ref{sec2} presents the fundamental definitions and precise statements of our main results.
It is divided into five subsections.
In the first subsection we introduce the concepts of trickle graphs and trickle groups, along with illustrative examples like graph products of cyclic groups.
In Subsection \ref{subsec2_2} we state the main results of Section \ref{sec4}, which focuses on the trickle algorithm.
In Subsection \ref{subsec2_3} we state the main results of Section \ref{sec5}, which focuses on the Tits-style algorithm.
In Subsection \ref{subsec2_4} we state the main results of Section \ref{sec6}, which focuses on parabolic subgroups.
In Subsection \ref{subsec2_5} we state the main results of Section \ref{sec7}, which focuses on preGarside trickle groups.
Section \ref{sec3} is devoted to examples and it is divided into four subsections: Subsection \ref{subsec3_1} for cactus groups in their generalized version, Subsection \ref{subsec3_2} for virtual cactus groups, Subsection \ref{subsec3_3} for Thompson group F, and Subsection \ref{subsec3_4} for ordered quandle groups.
As mentioned earlier, Sections \ref{sec4}, \ref{sec5}, \ref{sec6} and \ref{sec7} contain the proofs: those concerning the trickle algorithm in Section \ref{sec4}, those concerning the Tits-style algorithm in Section \ref{sec5}, those concerning parabolic subgroups in Section \ref{sec6}, and those concerning preGarside trickle groups in Section \ref{sec7}.

\begin{acknow}
This work originated during a research residency program titled ``Cactus and Posets'' at CIRM (Luminy, Marseille, France) from November 7th to 11th, 2022.
The three authors extend their sincere gratitude to CIRM for the generous support and resources provided (funding, dedicated workspaces, library access, etc.), without which this project would not have been possible.
\end{acknow}


\section{Definitions and statements}\label{sec2}

\subsection{Definitions and first examples}\label{subsec2_1}

The set of vertices of a simplicial graph $\Gamma$ is denoted by $V(\Gamma)$ and the set of its edges is denoted by $E(\Gamma)$.
The \emph{link} of a vertex $x \in V(\Gamma)$, denoted by $\link_x(\Gamma)$, is the full subgraph of $\Gamma$ spanned by $\{y \in V (\Gamma) \mid \{x, y\} \in E (\Gamma) \}$, and the \emph{star} of $x$, denoted by $\starE_x (\Gamma)$, is the full subgraph of $\Gamma$ spanned by $\{y \in V(\Gamma) \mid \{x, y\} \in E (\Gamma)\} \cup \{x\}$.

\begin{defin}
A \emph{trickle graph} is a quadruple $(\Gamma, \le, \mu, (\varphi_x)_{x \in V (\Gamma)})$, where
\begin{itemize}
\item
$\Gamma$ is a simplicial graph,
\item
$\le$ is a (partial) order on $V (\Gamma)$,
\item
$\mu$ is a labeling $\mu: V(\Gamma) \to \N_{\ge 2} \cup \{\infty\}$ of the vertices,
\item
$\varphi_x: \starE_x (\Gamma) \to \starE_x (\Gamma)$ is an automorphism of $\starE_x (\Gamma)$ for all $x \in V (\Gamma)$.
\end{itemize}
For $x, y \in V (\Gamma)$ the notation $x||y$ means that $x$ and $y$ are not comparable in the sense that $x \not \le y$ and $y \not \le x$.
We set $E_{||} (\Gamma) = \{\{x, y\} \in E(\Gamma) \mid x||y \}$.
The quadruple $(\Gamma, \le, \mu, (\varphi_x)_{x \in V (\Gamma)})$ must satisfy the following conditions.
\begin{itemize}
\item[(a)]
For all $x, y \in V (\Gamma)$, if $x < y$, then $\{x, y\} \in E(\Gamma)$.
\item[(b)]
For all $x, y, z \in V (\Gamma)$, if $\{x, y\} \in E_{||} (\Gamma)$ and $z \le y$, then $\{ x, z\} \in E_{||} (\Gamma)$.
\item[(c)]
For all $x \in V(\Gamma)$ and all $y, z \in \starE_x (\Gamma)$, we have $z\le y$ if and only if $\varphi_x(z) \le \varphi_x (y)$.
\item[(d)]
For all $x \in V(\Gamma)$ and all $ y \in \starE_x (\Gamma)$, if $\varphi_x (y) \neq y$, then $y < x$.
\item[(e)]
For all $x \in V (\Gamma)$, if $\mu (x)$ is finite, then $\varphi_x$ has finite order and its order divides $\mu (x)$.
\item[(f)]
For all $x \in V(\Gamma)$ and all $y \in \starE_x (\Gamma)$, $\mu (\varphi_x (y)) = \mu (y)$.
\item[(g)]
For all $x, y, z \in V (\Gamma)$, if $z < y < x$, then $(\varphi_x \circ \varphi_y) (z) = (\varphi_{y'} \circ \varphi_x) (z)$, where $y' = \varphi_x (y)$.
\end{itemize}
We will often say that $\Gamma$ is a trickle graph meaning that, implicitly, $\le$, $\mu$ and $(\varphi_x)_{x \in V (\Gamma)}$ are also given.
\end{defin}

\begin{rem}
Let $x, y, z \in V(\Gamma)$ be such that $z \le y \le x$.
Then it is easily seen that $(\varphi_x \circ \varphi_y) (z) = (\varphi_{y'} \circ \varphi_x) (z)$, where $ y' = \varphi_x (y)$, if either $z=y$ or $y=x$.
So, Condition (g) in the definition of a trickle graph also holds if at least two of the three vertices are equal.
\end{rem}

\begin{defin}
The \emph{trickle group} $\Tr (\Gamma)$ associated with a trickle graph $\Gamma = (\Gamma, \le, \mu, (\varphi_x)_{x \in V(\Gamma) })$ is the group defined by the following presentation.
\[
\Tr (\Gamma) = \langle V(\Gamma) \mid x^{\mu (x)} = 1 \text{ for } x \in V (\Gamma) \text{ such that } \mu ( x) \neq \infty \,, \ \varphi_x (y) \, x = \varphi_y (x) \,y
\text{ for } \{x, y\} \in E (\Gamma) \rangle\,.
\]
\end{defin}

\begin{rem}
By Condition (d) in the definition of a trickle graph, if $\{x, y\} \in E_{||} (\Gamma)$, then the relation $\varphi_x (y)\, x = \varphi_y(x) \,y$ becomes $y x = x y$.
If $x < y$, then this relation becomes $y x = \varphi_y (x) \, y$.
So, $\Tr (\Gamma)$ has a presentation with relations of the form $x y = z x$ and $x^\mu=1$.
\end{rem}

\begin{expl1}
Let $\Gamma$ be a simplicial graph and let $\mu : V(\Gamma) \to \N_{\ge 2} \cup \{ \infty\}$ be a labeling of the vertices.
To the pair $(\Gamma, \mu)$ we associate the \emph{graph product of cyclic groups} $G(\Gamma, \mu)$ defined by the following presentation.
\[
G (\Gamma, \mu) = \langle V (\Gamma) \mid x^{\mu (x)} = 1 \text{ for } x \in V (\Gamma) \text{ such that } \mu (x) \neq \infty\,,\ x y = y x \text{ for } \{x, y\} \in E(\Gamma) \rangle\,.
\]
If $\mu (x) = \infty$ for all $x \in V(\Gamma)$, then $G (\Gamma, \mu)$ is a \emph{right-angled Artin group}, and if $\mu (x) = 2$ for all $x \in V(\Gamma)$, then $G (\Gamma, \mu)$ is a \emph{right-angled Coxeter group}.
It is easily seen that $G (\Gamma, \mu)$ is a trickle group, where $\le$ is the trivial order defined by $x \le y$ if and only if $x = y$, and $ \varphi_x$ is the identity of $\starE_x (\Gamma)$ for all $x \in V(\Gamma)$.
\end{expl1}

\begin{expl2}
Le $\Upsilon$ be a Coxeter graph and let $(W,S)$ be its associated Coxeter system.
As in the introduction, for $X \subseteq S$, we denote by $W_X$ the standard parabolic subgroup generated by $X$ and by $\Upsilon_X$ the full Coxeter subgraph of $\Upsilon$ spanned by $X$.
Recall that $X \subseteq S$ is called \emph{irreducible} and of \emph{spherical type} if $\Upsilon_X$ is connected and $W_X$ is finite.
Recall also that, in this case, $W_X$ contains a unique element of maximal length (with respect to $S$), denoted by $w_X$, and this element satisfies $w_X^2 = 1$ and $w_X X w_X = X$ (see \cite{Bourb1}).
We denote by $\SS_{\fin}$ the set of non-empty subsets $X \subseteq S$ that are irreducible and of spherical type.

We define $\Gamma = (\Gamma, \le, \mu, (\varphi_x)_{x \in V (\Gamma)})$ as follows.
The set of vertices of $\Gamma$ is a set $V(\Gamma) = \{ x_X \mid X \in \SS_\fin \}$ in one-to-one correspondence with $\SS_\fin$.
Two vertices $x_X$ and $x_Y$ are connected by an edge if either $X \subset Y$, or $Y \subset X$, or $\Upsilon_{X \cup Y}$ is the disjoint union of $\Upsilon_X$ and $\Upsilon_Y$ in the sense that $X \cap Y = \emptyset$ and $\{s, t\} \not \in E (\Upsilon)$ for all $s \in X$ and $t \in Y$.
We set $x_X \le x_Y$ if $X \subseteq Y$.
We set $\mu(x_X) = 2$ for all $X \in \SS_{\fin}$.
Let $x_X \in V (\Gamma)$ and $x_Y \in V (\starE_{x_X} (\Gamma))$.
We set $\varphi_{x_X} (x_Y) = x_{w_X(Y)}$ if $Y \subset X$ and $\varphi_{x_X} (x_Y) = x_Y$ otherwise.
It is easily verified that $\le$ is a (partial) order on $V (\Gamma)$ and that, for all $x_X \in V (\Gamma)$, $\varphi_{x_X}$ is an automorphism of $\starE_{x_X} (\Gamma)$.

Now, we prove that $\Gamma = (\Gamma, \le, \mu, (\varphi_x)_{x \in V(\Gamma)})$ satisfies Conditions (a) to (g) of the definition of a trickle graph.

\begin{lem}\label{lem2_1}
The quadruple $\Gamma = (\Gamma, \le, \mu, (\varphi_x)_{x \in V(\Gamma)})$ above defined is a trickle graph.
\end{lem}

\begin{proof}
Conditions (a), (d) and (f) are satisfied by definition of $\Gamma$.

We show that $\Gamma$ satisfies Condition (b).
Let $X,Y,Z \in \SS_{\fin}$ be such that $\{x_X, x_Y\} \in E_{||} (\Gamma)$ and $x_Z \le x_Y$.
Then $\Upsilon_{X \cup Y}$ is the disjoint union of $\Upsilon_X$ and $\Upsilon_Y$ and $Z \subset Y$.
Thus, $\Upsilon_{X \cup Z}$ is the disjoint union of $\Upsilon_X$ and $\Upsilon_Z$, hence $\{x_X, x_Z\} \in E_{||} (\Gamma)$.

We show that $\Gamma$ satisfies Condition (e).
Let $X \in \SS_{\fin}$.
Since $w_X$ has order $2$, we have $\varphi_{x_X}^2 (x_Y) = x_{w_X^2(Y)} = x_Y$ if $Y \subset X$.
We have $\varphi_{x_X}^2 (x_Y) = \varphi_{x_X} (x_Y) = x_Y$ for every other vertex of $\starE_{x_X} (\Gamma)$, hence the order of $\varphi_{x_X}$ divides $\mu(x_X) =2$.

We show that $\Gamma$ satisfies Condition (c).
Let $X, Y, Z \in \SS_{\fin}$ be such that $x_Z, x_Y \in \starE_{x_X} (\Gamma)$.
Since,  by Condition (e) already proved, $\varphi_{x_X}^2 = \id$, to show the equivalence $x_Z \le x_Y\ \Leftrightarrow\ \varphi_{x_X} (x_Z) \le \varphi_{x_X} (x_Y)$, it suffices to show the implication $x_Z \le x_Y\ \Rightarrow\ \varphi_{x_X} (x_Z) \le \varphi_{x_X} (x_Y)$.
Suppose $x_Z \le x_Y$, that is, $Z \subseteq Y$.
If $\Upsilon_{X \cup Y}$ is the disjoint union of $\Upsilon_X$ and $\Upsilon_Y$, then $\Upsilon_{X \cup Z}$ is the disjoint union of $\Upsilon_X$ and $\Upsilon_Z $, hence $\varphi_{x_X} (x_Z) = x_Z \le x_Y = \varphi_{x_X} (x_Y)$.
If $x_X \le x_Z \le x_Y$, then $X \subseteq Z \subseteq Y$, hence $\varphi_{x_X} (x_Z) = x_Z \le x_Y = \varphi_{x_X} (x_Y)$.
If $x_Z \le x_X \le x_Y$, then $Z \subseteq X \subseteq Y$ hence $w_X (Z) \subseteq X \subseteq Y$, and therefore $\varphi_{x_X} (x_Z) = x_{w_X(Z)} \le x_X \le x_Y = \varphi_{x_X} (x_Y)$.
If $x_X \le x_Y$ and $x_Z || x_X$, then $X \subseteq Y$, $Z \subseteq Y$ and $\Upsilon_{X \cup Z}$ is the disjoint union of $\Upsilon_X$ and $\Upsilon_Z$, hence $\varphi_{x_X} (x_Z) = x_Z \le \varphi_{x_X} (x_Y) = x_Y$.
If $x_Z \le x_Y \le x_X$, then $Z \subseteq Y \subseteq X$, hence $w_X(Z) \subseteq w_X(Y) \subseteq X$, and therefore $\varphi_{x_X} (x_Z) = x_{w_X(Z)} \le x_{w_X(Y)} = \varphi_{x_X} (x_Y)$.

Finally we show that $\Gamma$ satisfies Condition (g).
Let $X, Y, Z \in \SS_{\fin}$ be such that $Z \subseteq Y \subseteq X$.
Let $Y' = w_X (Y)$.
We have
\[
(w_X w_Y) (Z) = (w_X w_Y w_X^{-1} w_X) (Z) = (w_{w_X(Y)} w_X) (Z) = (w_{Y'} w_X) (Z)\,,
\]
hence $(\varphi_{x_X} \circ \varphi_{x_Y}) (x_Z) = (\varphi_{x_{Y'}} \circ \varphi_{x_X}) (x_Z)$.
\end{proof}

It is obvious that the cactus group $C(W,S)$ is equal to the trickle group $\Tr (\Gamma)$.
\end{expl2}

\begin{expl3}
Let $\Gamma = (\Gamma, \le, \mu, (\varphi_x)_{x \in V(\Gamma)})$ be a trickle graph.
Notice that $\widetilde{\Gamma} = (\Gamma, \le, \mu, (\varphi_x^{-1})_{x \in V(\Gamma)})$ is also a trickle graph which we call the \emph{dual trickle graph} of $\Gamma$.
On the other hand, we can define the \emph{dual trickle group} $\widetilde{\Tr} (\Gamma)$ by the following presentation.
\[
\widetilde{\Tr} (\Gamma) = \langle V (\Gamma) \mid x^{\mu(x)} = 1 \text{ for } x \in V (\Gamma) \text{ such that } \mu (x) \neq \infty\,,\ x\, \varphi_x (y) = y\, \varphi_y (x)
\text{ for } \{ x, y \} \in E (\Gamma) \rangle\,.
\]
Trickle groups and dual trickle groups are related by the following.

\begin{prop}\label{prop2_2}
Let $\Gamma = (\Gamma, \le, \mu, (\varphi_x)_{x \in V(\Gamma)})$ be a trickle graph.
Then $\widetilde{\Tr} (\widetilde{\Gamma})=\Tr (\Gamma)$.
\end{prop}

\begin{proof}
The group $\widetilde{\Tr} (\widetilde{\Gamma})$ has the following presentation.
\begin{gather*}
\widetilde{\Tr} (\widetilde{\Gamma}) = \langle V (\Gamma) \mid x^{\mu (x)} = 1 \text{ for } x \in V (\Gamma) \text{ such that } \mu (x) \neq \infty\,,\ x\, \varphi_x^{-1} (y) = y \, \varphi_y^{-1} (x)\\
\text{ for } \{x, y\} \in E(\Gamma) \rangle\,.
\end{gather*}
Let $f : V(\Gamma) \to \widetilde{\Tr} (\widetilde{\Gamma})$ be the map defined by $f(x) = x$ for all $x \in V(\Gamma)$.
Let $x \in V(\Gamma)$ be such that $\mu (x) \neq \infty$.
Then $f(x)^{\mu (x)} = x^{\mu (x)} = 1$.
Let $e = \{x, y\} \in E (\Gamma)$.
If $x || y$, then with $e$ we associate the relation $xy = yx$ in the presentation of $\Tr (\Gamma)$ as well as in that of $\widetilde{\Tr} (\widetilde{\Gamma})$.
So, $f(x) \, f(y) = f(y) \, f(x)$.
Suppose $x < y$.
The case $y < x$ is treated in the same way.
With $e$ we associate the relation $y x = \varphi_y(x)\,y$ in the presentation of $\Tr (\Gamma)$.
Let $x' = \varphi_y (x)$.
By definition of a trickle graph we have $e' =\{ x' ,y \} \in E(\Gamma)$ and $x' < y$.
With $e'$ we associate the relation $y \, \varphi_y^{-1} (x') = x' y$ in the presentation of $\widetilde{\Tr} (\widetilde{\Gamma})$.
But, since $x' = \varphi_y (x)$, this relation is also read $y x = \varphi_y(x)\, y$.
Thus, $f(y) \, f(x) = f(\varphi_y(x)) \, f(y)$.
This shows that $f$ induces a homomorphism $f : \Tr (\Gamma) \to \widetilde{\Tr} (\widetilde{\Gamma})$.
We show in the same way that we have a homomorphism $f': \widetilde{\Tr} (\widetilde{\Gamma}) \to \Tr (\Gamma)$ which sends $x$ to $x$ for all $x \in V(\Gamma)$.
It is clear that $f'$ is the inverse of $f$, hence $f$ is an isomorphism.
\end{proof}

\begin{rem}
The proof of Proposition \ref{prop2_2} proves more than what is stated: it actually proves that the presentation of $\widetilde{\Tr} (\widetilde{\Gamma})$ is equal to that of $\Tr (\Gamma)$.
Nevertheless, even if the two presentations coincide, it will be useful subsequently to call the presentation
\[
\langle V (\Gamma) \mid x^{\mu (x)} = 1 \text{ for } x \in V (\Gamma) \text{ such that } \mu (x) \neq \infty\,,\ x\,\varphi_x^{-1} (y) = y \, \varphi_y^{-1} (x) \text{ for } \{x, y\} \in E(\Gamma) \rangle
\]
the \emph{dual presentation} of $\Tr (\Gamma)$.
\end{rem}
\end{expl3}


\subsection{The trickle algorithm -- Results of Section \ref{sec4}}\label{subsec2_2}

Let $A$ be a set, which we call an \emph{alphabet}, and let $A^*$ be the free monoid on $A$.
The elements of $A^*$ are called \emph{words} and they are written as finite sequences.
The empty word is denoted by $\epsilon$ and the concatenation of two words $w_1, w_2 \in A^*$ is denoted by $w_1 \cdot w_2$.
A \emph{rewriting system} on $A^*$ is defined to be a subset $R \subseteq A^* \times A^*$.
Let $w, w' \in A^*$.
We set $w \stackrel{R}{\to} w'$ or simply $w \to w'$ if there exist $w_1, w_2 \in A^*$ and $(u,v) \in R$ such that $w = w_1 \cdot u \cdot w_2$ and $w' = w_1 \cdot v \cdot w_2$.
More generally, we set $w \stackrel{R\ *}{\to} w'$ or simply $w \to^* w'$ if either $w = w'$ or there exists a finite sequence $w = w_0, w_1, \dots, w_p = w'$ in $A^*$ such that $w_{i-1} \to w_i$ for all $i \in \{1, \dots, p\}$.
A word $w \in A^*$ is said to be \emph{$R$-reducible} if there exists $w' \in A^*$ such that $w \to w'$.
Otherwise we say that $w$ is \emph{$R$-irreducible}.
The pair $(A, R)$ is a \emph{rewriting system for a monoid} $M$ if $\langle A \mid u = v \text{ for } (u,v) \in R \rangle^+$ is a monoid presentation for $M$.
A \emph{rewriting system for a group} $G$ is a rewriting system for $G$ viewed as a monoid.
In particular, in this case $A$ generates $G$ as a monoid.
If $(A, R)$ is a rewriting system for a monoid $M$ and $w=(\alpha_1, \alpha_2, \dots,\alpha_\ell) \in A^*$, then we denote by $\overline{w} = \alpha_1 \alpha_2 \dots \alpha_\ell$ the element of $M$ represented by $w$.

Let $R$ be a rewriting system on $A^*$.
We say that $R$ is \emph{terminating} if there is no infinite sequence $\{w_k\}_{k=0}^\infty$ in $A^*$ such that $w_{k-1} \to w_k$ for all $k \in \N_{\ge 1}$.
We say that $R$ is \emph{confluent} if, for all $u, v_1, v_2 \in A^*$ such that $u \to^* v_1$ and $u \to^* v_2$, there exists $w \in A^*$ such that $v_1 \to^* w$ and $v_2 \to^* w$.
The importance of terminating and confluent rewriting systems comes from the following.

\begin{thm}[Newman \cite{Newma1}]\label{thm2_3}
Let $(A,R)$ be a terminating and confluent rewriting system for a monoid $M$.
\begin{itemize}
\item[(1)]
For all $w'\in A^*$ there exists a unique $R$-irreducible word $w \in A^*$ such that $w' \to^* w$.
\item[(2)]
For all $g \in M$ there exists a unique $R$-irreducible word $w \in A^*$ such that $g = \overline{w}$. \end{itemize}
\end{thm}

Now, we fix a trickle graph $\Gamma = (\Gamma, \le, \mu, (\varphi_x)_{x \in V (\Gamma)})$, and we turn to describe a rewriting system for $\Tr (\Gamma)$.

The alphabet of our rewriting system is not $V(\Gamma) \sqcup V (\Gamma)^{-1}$, as one may expect, but a more complicated set $\Omega = \Omega (\Gamma)$, which is generally infinite, and which is described as follows.

Throughout the paper we use the following notations.
For $\mu \in \N_{\ge 2} \cup \{\infty\}$ we set $\Z_\mu = \Z / \mu \Z$ if $\mu \neq \infty$ and $\Z_\mu = \Z$ if $\mu = \infty$.

The set of \emph{syllables} of $\Gamma$ is the abstract set
\[
S (\Gamma) = \{x^a \mid x \in V (\Gamma) \text{ and } a \in \Z_{\mu (x)} \setminus \{0\} \}\,.
\]
A \emph{stratum} of $\Gamma$ is a finite subset $U = \{x_1^{a_1}, x_2^{a_2}, \dots, x_p^{a_p} \} \subseteq S (\Gamma)$ such that $x_i \neq x_j$ and $\{x_i, x_j\} \in E (\Gamma)$ for all $i, j \in \{1, \dots, p\}$ such that $i \neq j$.
The \emph{support} of $U$ is $\supp (U) = \{x_1, x_2, \dots, x_p\} \subseteq V (\Gamma)$ and its \emph{length} is the integer $p$, which is denoted by $\lg_\st (U)$.
The empty set $\emptyset$ is assumed to be a stratum whose support is $\emptyset$.
The set of strata is denoted by $\Omega = \Omega (\Gamma)$.

The set $\Omega$ is the alphabet of our rewriting system.
The elements of $\Omega^*$ are called \emph{pilings}.
If $u = (U_1, \dots, U_p)$ is a piling, then $p$ is the \emph{length} of $u$, which is denoted by $\lg_\pil (u)$.

Now, we define three operations on the strata which will be used to define our rewriting system.

The first operation consists in removing an element from a stratum.
If $U = \{x_1^{a_1}, x_2^{a_2}, \dots, x_p^{a_p}\}$ is a non-empty stratum and $x_i^{a_i} \in U$, then we set
\[
L (U, x_i^{a_i}) = U \setminus \{ x_i^{a_i} \} = \{ x_1^{a_1}, \dots, x_{i-1}^{a_{i-1}} , x_{i+1}^{a_{i+1}}, \dots, x_p^{a_p}\}\,.
\]
Note that $L (U, x_i^{a_i})$ is a stratum.

The second operation consists in ``extracting'' a syllable from a stratum.
Let $U = \{x_1^{a_1}, x_2^{a_2}, \dots, x_p^{a_p}\}$ be a non-empty stratum and let $x_i^{a_i} \in U$.
We number the elements of $U$ such that, if $x_j > x_k$, then $j < k$.
Then we set 
\[
\gamma (U, x_i^{a_i}) = \big( (\varphi_{x_1}^{a_1} \circ \varphi_{x_2}^{a_2} \circ \cdots \circ \varphi_{x_{i-1}}^{a_{i-1}}) (x_i)\big)^{a_i}\,.
\]
It is easily seen that $\gamma(U, x_i^{a_i})$ is well-defined and belongs to $S(\Gamma)$.
Moreover, it will be proved in Section \ref{sec4} (see Lemma \ref{lem4_2}) that the definition of $\gamma (U, x_i^{a_i})$ does not depend on the choice of the numbering of the elements of $U$.

The third operation consists in ``adding'' a syllable to a stratum.
Let $U = \{x_1^{a_1}, x_2^{a_2}, \dots, x_p^{a_p} \} \in \Omega$ and let $y^b \in S(\Gamma)$.
We say that $y^b$ can be \emph{added} to $U$ if either $y \in \supp (U)$ or $\{y, x_i\} \in E(\Gamma)$ for all $i \in \{1, \dots, p\}$.
Suppose $y^b$ can be added to $U$.
If $y \not \in \supp (U)$, then we set
\[
R (U, y^b) = \{ \varphi_y^{-b} (x_1)^{a_1}, \dots, \varphi_y^{-b} (x_p)^{a_p}, y^b\}\,.
\]
If $y = x_i \in \supp (U)$ and $b + a_i = 0$ (in $\Z_{\mu (y)}$), then we set
\[
R(U, y^b) = \{\varphi_y^{-b} (x_1)^{a_1}, \dots, \varphi_y^{-b} (x_{i-1})^{a_{i-1}}, \varphi_y^{-b} (x_{i+1})^{a_{i+1}}, \dots, \varphi_y^{-b} (x_p)^{a_p}\}\,.
\]
If $y = x_i \in \supp (U)$ and $b + a_i \neq 0$ (in $\Z_{\mu (y)}$), then we set 
\[
R (U, y^b) = \{ \varphi_y^{-b} (x_1)^{a_1}, \dots, \varphi_y^{-b} (x_{i-1})^{a_{i-1}}, y^{a_i+b}, \varphi_y^{-b} (x_{i+1})^{a_{i+1}}, \dots, \varphi_y^{-b} (x_p) ^{a_p}\}\,.
\]
Note that, in the third case, since $y = x_i = \varphi_y^{-b}(x_i)$, $y^{a_i+b}$ can be replaced by $\varphi_y^{-b} (x_i)^{a_i+b}$.
Note also that $R (U, y^b)$ is always a stratum.

\begin{defin}
Let $(U, V)$ be a pair of strata with $V \neq \emptyset$ and let $x^a \in V$.
We set $V' = L(V, x^a)$ and $y^a = \gamma (V, x^a)$.
We assume that $y^a$ can be added to $U$ and we set $U' = R(U, y^a)$.
Then we say that $r = ((U, V), (U', V')) \in \Omega^* \times \Omega^*$ is a \emph{T-transformation}.
In this case we write $T (U, V, x^a) = (U', V')$.
We denote by $\RR_1$ the set of T-transformations.
On the other hand, we set $\RR_0 = \{ ((\emptyset), \epsilon)\} \subset \Omega^* \times \Omega^*$, where $(\emptyset)$ is the piling of length $1$ whose only entry is $\emptyset$ and $\epsilon$ is the empty piling of length $0$.
Finally, we set $\RR = \RR (\Gamma)= \RR_0 \cup \RR_1$.
\end{defin}

The following will be proved in Section \ref{sec4}.

\begin{thm}\label{thm2_4}
Let $\Gamma = (\Gamma, \le, \mu, (\varphi_x)_{x \in V(\Gamma)})$ be a trickle graph, let $\Omega = \Omega (\Gamma)$ be the set of strata of $\Gamma$, and let $\RR = \RR (\Gamma) \subseteq \Omega^* \times \Omega^*$ be as defined above.
Then $\RR$ is a rewriting system for $\Tr (\Gamma)$, and it is terminating and confluent.
\end{thm}

As a consequence of Theorem \ref{thm2_4} we get that $\langle \Omega \mid u = v \text{ for } (u,v) \in \RR \rangle^+$ is a monoid presentation for $\Tr (\Gamma)$.
We explain how to go from this presentation to the standard one and vice versa.

Set $M (\Gamma) = \langle \Omega \mid u = v \text{ for } (u,v) \in \RR \rangle^+$.
Let $U \in \Omega$.
We write $U = \{ x_1^{a_1}, x_2^{a_2}, \dots, x_p^{a_p} \}$ so that, if $x_i > x_j$, then $i < j$, and we set $\omega (U) = x_1^{a_1} x_2^{a_2} \dots x_p^{a_p} \in \Tr (\Gamma)$.
We will show in Lemma \ref{lem4_13} that the definition of $\omega (U)$ does not depend on the choice of the numbering of the elements of $U$.
Then the isomorphism $\Phi: M (\Gamma) \to \Tr (\Gamma)$ sends $U$ to $\omega (U)$ for all $U \in \Omega$.
The reverse isomorphism $\Psi: \Tr (\Gamma) \to M (\Gamma)$ sends $x$ (resp. $x^{-1}$) to the piling $(\{x\})$ (resp. $(\{ x^{-1} \})$) of length $1$ for all $x \in V (\Gamma)$.
Details and proofs about these isomorphisms will be given in Section \ref{sec4}.

Recall that, if $G$ is a group and $V$ is a generating set for $G$, then a \emph{set of normal forms} for $G$ based on $V \sqcup V^{-1}$ is a language $\LL \subseteq (V \sqcup V^{-1})^*$ such that, for all $g \in G$, there exists a unique $w \in \LL$ such that $\overline{w} = g$.
The above constructions provide normal forms for $\Tr (\Gamma)$ based on $V (\Gamma) \sqcup V(\Gamma)^{-1}$ as well as an algorithm to calculate them when $V (\Gamma)$ is finite, as follows.

Let $\mu \in \N_{\ge 2} \cup \{ \infty \}$ and $a \in \Z_\mu \setminus \{ 0\}$.
If $\mu = \infty$, then we set $\rho (a) = a$.
If $\mu \neq \infty$, then we denote by $\rho (a)$ the unique representative of $a$ sitting inside $\{1, \dots, \mu-1\}$.
We fix a total order $\preceq$ on $V (\Gamma)$ which extends the partial order $\le$ in the sense that, if $x \le y$, then $x \preceq y$.
Such a total order always exists but it is not unique in general.
Let $U$ be a non-empty stratum that we write $U = \{ x_1^{a_1}, x_2^{a_2}, \dots, x_p^{a_p} \}$ with $x_1 \succ x_2 \succ \cdots \succ x_p$.
Then we set $\hat \omega (U) = x_1^{\rho (a_1)} \cdot x_2^{\rho (a_2)} \cdots x_p^{\rho (a_p)} \in (V (\Gamma) \sqcup V (\Gamma)^{-1})^*$.
Note that $\hat \omega (U)$ is a representative of $\omega (U)$.

Let $g \in \Tr (\Gamma)$.
By Theorems \ref{thm2_3} and \ref{thm2_4} there exists a unique piling $w = (U_1, U_2, \dots, U_p) \in \Omega^*$ such that $w$ is $\RR$-irreducible and $\overline{w} = \Psi (g)$.
Then we set 
\[
\nf (g) = \hat \omega (U_1) \cdot \hat \omega (U_2) \cdots \hat \omega (U_p) \in (V (\Gamma) \sqcup V (\Gamma)^{-1})^*\,.
\]
The following is a direct consequence of the previous constructions.

\begin{corl}\label{corl2_5}
Let $\Gamma = (\Gamma, \le, \mu, (\varphi_x)_{x \in V(\Gamma)})$ be a trickle graph.
\begin{itemize}
\item[(1)]
The set $\LL = \LL (\Gamma) = \{ \nf (g) \mid g \in \Tr (\Gamma) \}$ is a set of normal forms for $\Tr (\Gamma)$ based on $V (\Gamma) \sqcup V (\Gamma)^{-1}$.
\item[(2)]
Suppose that $V( \Gamma)$ is finite.
Then there exists an algorithm which, given a word $w \in (V(\Gamma) \sqcup V(\Gamma)^{-1})^*$, calculates $\nf (\overline{w})$.
In particular, this algorithm is a solution to the word problem in $\Tr (\Gamma)$.
\end{itemize}
\end{corl}

\begin{rem}
Corollary \ref{corl2_5}\,(2) can be extended to a trickle graph with infinite but countable vertex set $V (\Gamma)$ provided that there exist an algorithm to manipulate the elements of $V (\Gamma)$, an algorithm which, given $x, y \in V (\Gamma)$, decides whether $\{x, y \} \in E (\Gamma)$ or not, an algorithm which, given $x,y \in V (\Gamma)$, decides whether $x \le y$ or not, and an algorithm which, given $x \in V (\Gamma)$ and $y \in \starE_x (\Gamma)$, determines $\varphi_x (y)$.
\end{rem}

Another straightforward consequence of Theorems \ref{thm2_3} and \ref{thm2_4} is the following.

\begin{corl}\label{corl2_6}
Let $\Gamma = (\Gamma, \le, \mu, (\varphi_x)_{x \in V(\Gamma)})$ be a trickle graph.
Then the map $S (\Gamma) \to \Tr (\Gamma)$, $x^a \mapsto x^a$, is injective.
In particular, $V(\Gamma)$ is a subset of $\Tr (\Gamma)$.
\end{corl}

\begin{proof}
The result follows from the fact that, for $x^a \in S(\Gamma)$, $x^{\rho (a)} \in (V(\Gamma) \sqcup V (\Gamma)^{-1})^*$ is the normal form of $x^a \in \Tr (\Gamma)$.
\end{proof}

Another more unexpected consequence is the following.

\begin{corl}\label{corl2_7}
Let $\Gamma = (\Gamma, \le, \mu, (\varphi_x)_{x \in V(\Gamma)})$ be a trickle graph.
Then $\Tr (\Gamma)$ is finite if and only if $V(\Gamma)$ is finite, $\Gamma$ is complete, and $\mu (x) \neq \infty$ for all $x \in V(\Gamma)$.
\end{corl}

\begin{proof}
Suppose $V (\Gamma)$ is infinite.
Then $\Tr (\Gamma)$ is infinite because $V(\Gamma) \subseteq \Tr (\Gamma)$.
Suppose there exists $x \in V(\Gamma)$ such that $\mu (x) = \infty$.
Then, by Corollary \ref{corl2_6}, $\{ x^a \mid a \in \Z \setminus \{0\} \} \subseteq S (\Gamma) \subseteq \Tr (\Gamma)$, hence $\Tr (\Gamma)$ is infinite.
Suppose $\Gamma$ is not complete.
Let $x,y \in V(\Gamma)$ be such that $\{x, y\} \not \in E (\Gamma)$.
Let $u=(\{x\}, \{y\}) \in \Omega^*$.
Then, for all $n \in \N$, $u^n$ is an $\RR$-irreducible piling of length $2n$, hence, by Theorems \ref{thm2_3} and \ref{thm2_4}, the set $\{\overline{u ^n} \mid n \in \N\}$ is an infinite subset of $\Tr (\Gamma)$, and therefore $\Tr (\Gamma)$ is infinite.

Suppose $V(\Gamma)$ is finite, $\Gamma$ is complete, and $\mu (x) \neq \infty$ for all $x \in V(\Gamma)$.
We fix a total order $\preceq$ on $V (\Gamma)$ which extends the partial order $\le$, and we write $V (\Gamma) = \{ x_1, x_2, \dots, x_p \} $ with $x_1 \succ x_2 \succ \cdots \succ x_p$.
Then the set of normal forms for $\Tr (\Gamma)$ is
\[
\{ x_{i_1}^{a_1} \cdot x_{i_2}^{a_2} \cdots x_{i_q}^{a_q} \mid 1 \le i_1 < i_2 < \cdots < i_q \le p\,,\ 1 \le a_j \le \mu (x_{i_j})-1 \text{ for } 1 \le j \le q\}\,.
\]
This set is finite and it is in one-to-one correspondence with $\Tr (\Gamma)$, hence $\Tr (\Gamma)$ is finite.
\end{proof}


\subsection{The Tits-style algorithm -- Results of Section \ref{sec5}}\label{subsec2_3}

Let $\Gamma = (\Gamma, \le, \mu, (\varphi_x)_{x\in V(\Gamma)})$ be a trickle graph.
Recall that the set of \emph{syllables} of $\Tr (\Gamma)$ is $S (\Gamma) = \{x^a \mid x \in V(\Gamma) \text{ and } a \in \Z_{\mu (x)} \setminus\{0\} \}$.
Recall also that, by Corollary \ref{corl2_6}, $S (\Gamma)$ is a subset of $\Tr (\Gamma)$.
The elements of $S(\Gamma)^*$ are called \emph{syllabic words}.
For $w=(x_1^{a_1}, x_2^{a_2}, \dots, x_p^{a_p}) \in S(\Gamma)^*$ we denote by $\overline{w} = x_1^{a_1} x_2^{a_2} \dots x_p^{a_p}$ the element of $\Tr (\Gamma)$ represented by $w$. 
The integer $p$ is called the \emph{length} of $w$ and it is denoted by $\lg (w)$.
The smallest length of a syllabic word representing an element $g \in \Tr (\Gamma)$ is called the \emph{syllabic length} of $g$ and it is denoted by $\lg_{\syl} (g)$.
A syllabic word $w = (x_1^{a_1}, x_2^{a_2}, \dots, x_p^{a_p})$ is said to be \emph{syllabically reduced} if $\lg (w) = \lg_\syl (\overline{w})$.

Our Tits-style algorithm uses a rewriting system on $S (\Gamma)^*$, but it does not use Newman's results \cite{Newma1} stated in Theorem \ref{thm2_3}.
This rewriting system is defined as follows.

We set 
\begin{gather*}
\RR_{I} = \{ ((x^a, x^{-a}), \epsilon) \mid x \in V (\Gamma)\,,\ a \in \Z_{\mu (x)} \setminus \{0\} \}\ \cup\\
\{ ((x^a, x^b),(x^{a+b})) \mid x \in V (\Gamma)\,, \ a, b \in \Z_{\mu (x)} \setminus \{0\} \text{ and } a + b \neq 0\}\,,\\
\RR_{II} = \{ ((x^a, y^b), (\varphi_x^a (y)^b, \varphi_y^{-b}(x)^a)) \mid x^a, y^b \in S (\Gamma) \text{ and } \{x,y\} \in E (\Gamma) \}\,,\\
\RR_M = \RR_I \cup \RR_{II}\,.
\end{gather*}
Let $x^a, y^b \in S (\Gamma)$ be such that $\{x,y\} \in E (\Gamma)$. 
Note that, if $x || y$, then $(\varphi_x^a (y)^b, \varphi_y^{-b}(x)^a) = (y^b, x^a)$, if $x > y$, then $ (\varphi_x^a (y)^b, \varphi_y^{-b}(x)^a) = (\varphi_x^a (y)^b, x^a)$, and, if $x < y$, then $(\varphi_x^a (y)^b, \varphi_y^{-b}(x)^a) = (y^b, \varphi_y^{-b}(x)^a)$.
Note also that in all three cases we have the equality $x^a y^b = \varphi_x^a (y)^b\, \varphi_y^{-b}(x)^a$ in $\Tr (\Gamma)$.
For $X \in \{I, II, M\}$ and $w, w' \in S (\Gamma)^*$, we write $w \stackrel{X}{\to} w'$ in place of $w \stackrel{\RR_X}{\to} w'$ and $w \stackrel{X\ *}{\to} w'$ in place of $w \stackrel{\RR_X\, *}{\longrightarrow} w'$. 

Let $w, w' \in S (\Gamma)^*$.
Notice that, if $w \stackrel{I}{\to} w'$, then $\lg (w) > \lg (w')$.
In particular, the operation $\stackrel{I}{\to}$ is not reversible.
On the other hand the operation $\stackrel{II}{\to}$ is reversible in the sense that, if $w \stackrel{II}{\to} w'$, then $\lg (w) = \lg (w')$ and $w' \stackrel{II}{\to} w$.
We say that a syllabic word $w \in S (\Gamma)^*$ is \emph{syllabically $M$-reduced} if there is no word $w' \in S (\Gamma)^*$ such that $w \stackrel{M\ *}{\to} w'$ and $\lg (w') < \lg (w)$.

The following theorem is similar to the classical solution to the word problem for Coxeter groups \cite{Tits1}, for graph products of cyclic groups \cite{Green1}, and, more generally, for Dyer groups \cite{ ParSoe1}.
It will be proved in Section \ref{sec5}.

\begin{thm}\label{thm2_8}
Let $\Gamma = (\Gamma, \le, \mu, (\varphi_x)_{x\in V(\Gamma)})$ be a trickle graph.
\begin{itemize}
\item[(1)]
For every $w \in S (\Gamma)^*$, $w$ is syllabically reduced if and only if $w$ is syllabically $M$-reduced.
\item[(2)]
For every $w, v \in S (\Gamma)^*$, if $w$ and $v$ are syllabically reduced and $\overline{w} = \overline{v}$, then $w \stackrel{II\, *}{\to} v$.
\end{itemize}
\end{thm}


\subsection{Parabolic subgroups -- Results of Section \ref{sec6}}\label{subsec2_4}

\begin{defin}
Let $\Gamma = (\Gamma, \le, \mu, (\varphi_x)_{x\in V(\Gamma)})$ be a trickle graph.
A full subgraph $\Gamma_1$ of $\Gamma$ is called \emph{parabolic} if, for every $x \in V(\Gamma_1)$, the map $\varphi_x$ stabilizes $\starE_x (\Gamma_1)$.
\end{defin}

\begin{expl}
Let $\Gamma = (\Gamma, \le, \mu, (\varphi_x)_{x\in V(\Gamma)})$ be a trickle graph.
\begin{itemize}
\item[(1)]
If $X$ is a subset of $V(\Gamma)$ whose elements are pairwise incomparable, then the full subgraph of $\Gamma$ spanned by $X$ is parabolic.
\item[(2)]
If $X$ is a subset of $V(\Gamma)$ and $X_\downarrow = \{y \in V(\Gamma) \mid \exists x \in X, y \leq x\}$, then the full subgraph of $\Gamma$ spanned by $X_\downarrow$ is parabolic.
\end{itemize}
\end{expl}

The following can be easily checked.

\begin{lem}\label{lem2_9}
Let $\Gamma = (\Gamma, \le, \mu, (\varphi_x)_{x\in V(\Gamma)})$ be a trickle graph and let $\Gamma_1$ be a parabolic subgraph of $\Gamma$.
We denote by $\le_1$ and by $\mu_1$ the restrictions of $\le$ and $\mu$ to $V(\Gamma_1)$, respectively, and, for $x \in V (\Gamma_1)$, we denote by $\varphi_{1,x}$ the restriction of $\varphi_x$ to $\starE_x (\Gamma_1)$.
\begin{itemize}
\item[(1)]
The quadruple $\Gamma_1 = (\Gamma_1, \le_1, \mu_1, (\varphi_{1,x})_{x\in V(\Gamma_1)})$ is a trickle graph.
\item[(2)]
The embedding of $V(\Gamma_1)$ into $V(\Gamma)$ induces a group homomorphism $\iota_1: \Tr (\Gamma_1) \to \Tr (\Gamma)$.
\end{itemize}
\end{lem}

In our study in Section \ref{sec6} we will use the normal forms for $\Tr (\Gamma)$ as they are defined in Subsection \ref{subsec2_2}.
We recall them.
We fix a total order $\preceq$ on $V(\Gamma)$ which extends the partial order $\le$ in the sense that, if $x \le y$, then $x \preceq y$.
If $U$ is a non-empty stratum that we write $U = \{ x_1^{a_1}, x_2^{a_2}, \dots, x_p^{a_p} \}$ with $x_1 \succ x_2 \succ \cdots \succ x_p$, then we set $\hat \omega (U) = x_1^{\rho (a_1)} \cdot x_2^{\rho (a_2)} \cdots x_p^{\rho (a_p)} \in (V (\Gamma) \sqcup V (\Gamma)^{-1})^*$.
Let $g \in \Tr (\Gamma)$.
By Theorems \ref{thm2_3} and \ref{thm2_4} there exists a unique piling $w = (U_1, U_2, \dots, U_p) \in \Omega^*$ such that $w$ is $\RR$-irreducible and $\overline{w} = g$.
Then the \emph{normal form} of $g$ is
\[
\nf_\Gamma (g) = \nf (g) = \hat \omega (U_1) \cdot \hat \omega (U_2) \cdots \hat \omega (U_p) \in (V (\Gamma) \sqcup V (\Gamma)^{-1})^*\,.
\]

Let $\Gamma_1$ be a parabolic subgraph of $\Gamma$.
As before, we assume that $V (\Gamma_1)$ is endowed with the restriction $\le_1$ of $\le$ to $V (\Gamma_1)$ and with the restriction $\mu_1: V (\Gamma_1) \to \N_{\ge 2} \cup \{ \infty\}$ of $\mu$ to $V (\Gamma_1)$.
Also, for each $x \in V (\Gamma_1)$, we denote by $\varphi_{1,x}$ the restriction of $\varphi_x$ to $\starE_x (\Gamma_1)$.
Thus, as pointed out in Lemma \ref{lem2_9}, $\Gamma_1 = (\Gamma_1, \le_1, \mu_1, (\varphi_{1,x})_{x \in V (\Gamma_1)})$ is a trickle graph.
Now, we denote by $\preceq_1$ the restriction of $\preceq$ to $V (\Gamma_1)$ and we observe that $\preceq_1$ extends the partial order $\le_1$, hence it can be used to define normal forms for $\Tr (\Gamma_1)$.
Unless otherwise stated, we will always use this total order to define normal forms for $\Tr (\Gamma_1)$.

The main result of Section \ref{sec6} is the following.

\begin{thm}\label{thm2_10}
Let $\Gamma = (\Gamma, \le, \mu, (\varphi_x)_{x\in V(\Gamma)})$ be a trickle graph and let $\Gamma_1$ be a parabolic subgraph of $\Gamma$.
\begin{itemize}
\item[(1)]
The homomorphism $\iota_1: \Tr (\Gamma_1) \to \Tr (\Gamma)$ induced by the embedding of $V (\Gamma_1)$ into $V (\Gamma)$ is injective.
In particular, we can identify $\Tr (\Gamma_1)$ with its image in $\Tr (\Gamma)$ under $\iota_1$.
\item[(2)]
Let $g \in \Tr (\Gamma)$.
We have $g \in \Tr (\Gamma_1)$ if and only if $\nf (g) \in (V (\Gamma_1) \sqcup V(\Gamma_1)^{-1})^*$.
Moreover, in this case, $g$ has the same normal form in $\Tr (\Gamma)$ as in $\Tr (\Gamma_1)$, that is, $\nf_\Gamma (g) = \nf_{\Gamma_1} (g)$.
\end{itemize}
\end{thm}

\begin{defin}
Let $\Gamma = (\Gamma, \le, \mu, (\varphi_x)_{x \in V (\Gamma)})$ be a trickle graph and let $\Gamma_1$ be a parabolic subgraph of $\Gamma$.
Then $\Tr (\Gamma_1)$, viewed as a subgroup of $\Tr (\Gamma)$, is called a \emph{standard parabolic subgroup} of $\Tr (\Gamma)$.
\end{defin}

\begin{expl}
Let $(W,S)$ be a Coxeter system.
Let $\Gamma = (\Gamma, \le, \mu, (\varphi_x)_{x \in V (\Gamma)})$ be the trickle graph such that $\Tr (\Gamma) = C(W, S)$, as defined in Subsection \ref{subsec2_1}.
Recall that $\SS_{\fin}$ denotes the set of non-empty subsets $X$ of $S$ that are irreducible and of spherical type, and that the set of vertices of $\Gamma$ is $V (\Gamma) = \{ x_X \mid X \in \SS_{\fin} \}$.
Let $Y \subseteq S$.
We set $\SS_{Y,\fin} = \{X \in \SS_{\fin} \mid X \subseteq Y \}$ and we denote by $\Gamma_Y$ the full subgraph of $\Gamma$ spanned by $\{ x_X \mid X \in \SS_{Y,\fin}\}$
(Caution: do not confuse $\Gamma_Y$ with the full Coxeter subgraph spanned by $Y$). 
The proof of the following is left to the reader.

\begin{lem}\label{lem2_11}
Under the above assumptions, $\Gamma_Y$ is a parabolic subgraph of $\Gamma$ and $\Tr (\Gamma_Y)$ is naturally isomorphic to the cactus group $C (W_Y, Y)$.
In particular, $C(W_Y, Y)$ is a standard parabolic subgroup of $C(W, S)$.
\end{lem}
\end{expl}

If $\Gamma_1$ and $\Gamma_2$ are two full subgraphs of a trickle graph $\Gamma$, then we denote by $\Gamma_1 \cap \Gamma_2$ the full subgraph of $\Gamma$ spanned by $V (\Gamma_1) \cap V (\Gamma_2)$.
It is easily seen that, if $\Gamma_1$ and $\Gamma_2$ are both parabolic subgraphs, then $\Gamma_1 \cap \Gamma_2$ is also a parabolic subgraph.
Now, from Theorem \ref{thm2_10} we deduce the following.

\begin{corl}\label{corl2_12}
Let $\Gamma = (\Gamma, \le, \mu, (\varphi_x)_{x\in V(\Gamma)})$ be a trickle graph and let $\Gamma_1, \Gamma_2$ be two parabolic subgraphs of $\Gamma$.
Then $\Tr (\Gamma_1) \cap \Tr (\Gamma_2) = \Tr (\Gamma_1 \cap \Gamma_2)$.
In particular, the intersection of two standard parabolic subgroups of $\Tr (\Gamma)$ is a standard parabolic subgroup.
\end{corl}

\begin{proof}
The inclusion $\Tr (\Gamma_1 \cap \Gamma_2) \subseteq \Tr (\Gamma_1) \cap \Tr (\Gamma_2)$ is obvious.
We prove the reverse inclusion.
Let $g \in \Tr (\Gamma_1) \cap \Tr (\Gamma_2)$.
By Theorem \ref{thm2_10}, $\nf (g) \in (V (\Gamma_1) \sqcup V (\Gamma_1)^{-1})^*$ and $\nf (g) \in (V (\Gamma_2 ) \sqcup V (\Gamma_2)^{-1})^*$.
But
\[
(V (\Gamma_1) \sqcup V (\Gamma_1)^{-1})^* \cap (V (\Gamma_2) \sqcup V (\Gamma_2)^{-1})^* = (V (\Gamma_1 \cap \Gamma_2) \sqcup V (\Gamma_1 \cap \Gamma_2)^{-1})^*\,,
\]
hence $\nf (g) \in (V (\Gamma_1 \cap \Gamma_2) \sqcup V (\Gamma_1 \cap \Gamma_2)^{-1})^*$, and therefore $g \in \Tr (\Gamma_1 \cap \Gamma_2)$.
\end{proof}

Another straightforward consequence of Theorem \ref{thm2_10} is the following.

\begin{corl}\label{corl2_13}
Let $\Gamma = (\Gamma, \le, \mu, (\varphi_x)_{x \in V(\Gamma)})$ be a trickle graph with $V (\Gamma)$ finite, and let $\Gamma_1$ be a parabolic subgraph of $\Gamma$.
There exists an algorithm which, given $g \in \Tr (\Gamma)$, decides whether $g$ belongs to $\Tr (\Gamma_1)$ or not.
\end{corl}

\begin{proof}
Let $g \in \Tr (\Gamma)$.
Then, by Theorem \ref{thm2_10}, $g \in \Tr (\Gamma_1)$ if and only if $\nf (g) \in (V (\Gamma_1) \sqcup V (\Gamma_1)^{-1})^*$.
\end{proof}

\begin{rem}
As with Corollary \ref{corl2_5}\,(2), Corollary \ref{corl2_13} can be extended to the case where $V (\Gamma)$ is infinite but countable, provided that one has an algorithm for manipulating the elements of $V (\Gamma)$, an algorithm which, given $x,y \in V (\Gamma)$, decides whether $\{ x,y\}$ belongs to $E (\Gamma)$ or not, an algorithm which, given $x,y \in V (\Gamma)$, decides whether $x < y$ or not, an algorithm which, given $x \in V (\Gamma)$ and $y \in \starE_x (\Gamma)$, calculates $\varphi_x (y)$, and an algorithm which, given $x \in V (\Gamma)$, decides whether $x \in V (\Gamma_1)$ or not.
\end{rem}


\subsection{PreGarside trickle groups -- Results of Section \ref{sec7}}\label{subsec2_5}

We start by recalling some definitions and notations on monoids.
We say that a monoid $M$ is \emph{cancellative} if, for all $a, b, c, d \in M$, if $c a d = c b d$, then $a = b$.
We say that $M$ is \emph{atomic} if there exists a map $\nu : M \to \N$, called a \emph{norm}, such that $\nu (a) = 0$ if and only if $a=1$, and $\nu (a b) \ge \nu (a) + \nu (b)$ for all $a, b \in M$.
If $M$ is atomic, then we can define two partial orders $\le_L$ and $\le_R$ on $M$ as follows.
Let $a, b \in M$.
We set $a \le_L b$ if there exists $c \in M$ such that $a c = b$, and we set $a \le_R b$ if there exists $c' \in M$ such that $c' a = b$.
For $b \in M$ we set $\Div_L (b) = \{a \in M \mid a \le_L b \}$ and $\Div_R (b) = \{ a \in M \mid a \le_R b\}$.
We say that $b$ is \emph{balanced} if $\Div_L (b) = \Div_R (b)$, and, in this case, we set $\Div (b) = \Div_L (b) = \Div_R (b)$.
The enveloping group of a monoid $M$ is denoted by $G(M)$ and the natural homomorphism from $M$ to $G (M)$ is denoted by $\iota_M : M \to G(M)$.

\begin{defin}
A monoid $M$ is called a \emph{preGarside monoid} if it satisfies the following three properties.
\begin{itemize}
\item[(a)]
$M$ is cancellative and atomic.
\item[(b)]
For all $a,b \in M$, if the set $\{ c \in M \mid a \le_L c \text{ and } b \le_L c \}$ is non-empty, then it contains a least element that we denote by $a \vee_L b$.
\item[(c)]
For all $a,b \in M$, if the set $\{ c \in M \mid a \le_R c \text{ and } b \le_R c \}$ is non-empty, then it contains a least element that we denote by $a \vee_R b$.
\end{itemize}
A \emph{preGarside group} is the enveloping group of a preGarside monoid.
\end{defin}

\begin{defin}
Let $M$ be a preGarside monoid.
An element $\Delta \in M$ is called a \emph{Garside element} if $\Delta$ is balanced and $\Div (\Delta)$ generates $M$.
If $M$ contains a Garside element, then $M$ is called a \emph{Garside monoid} and $G (M)$ is called a \emph{Garside group}.
\end{defin}

The notions of Garside monoids and Garside groups were introduced by Dehornoy and the third author in \cite{DehPar1,Dehor1}, where they show that these monoids and groups share many properties with spherical type Artin monoids and groups such as solutions to the word problem and to the conjugacy problem.
The theory of Garside groups is experiencing a significant growth and many examples and applications have appeared since the publication of \cite{DehPar1,Dehor1}.
On the other hand, our definitions of preGarside monoid and group are taken from \cite{GodPar2}.
The motivation for their study is that, by \cite{BriSai1}, every Artin monoid is a preGarside monoid and every Artin group is a preGarside group.
So, preGarside groups are to Garside groups what Artin groups are to spherical type Artin groups.

\begin{defin}
We say that a trickle graph $\Gamma = (\Gamma, \le, \mu, (\varphi_x)_{x \in V (\Gamma)})$ is a \emph{preGarside trickle graph} if $\mu (x) = \infty$ for all $x \in V (\Gamma)$.
In this case, in addition to the trickle group $\Tr (\Gamma)$, we associate with $\Gamma$ a \emph{preGarside trickle monoid}, $\Tr^+ (\Gamma)$, which is defined by the following monoid presentation.
\[
\Tr^+ (\Gamma) = \langle V (\Gamma) \mid \varphi_x (y)\, x = \varphi_y (x)\, y \text{ for all } \{ x, y \} \in E (\Gamma) \rangle^+ \,.
\]
It is clear that $\Tr (\Gamma)$ is the enveloping group of $\Tr^+ (\Gamma)$.
\end{defin}

The above definition is motivated by the following theorem which will be proved in Section \ref{sec7}.

\begin{thm}\label{thm2_14}
Let $\Gamma$ be a preGarside trickle graph.
Then $\Tr^+ (\Gamma)$ is a preGarside monoid.
\end{thm}

The next question is: when is $\Tr^+ (\Gamma)$ a Garside monoid?
The answer is given by the following theorem which will be also proved in Section \ref{sec7}.

\begin{thm}\label{thm2_15}
Let $\Gamma$ be a preGarside trickle graph.
Then $\Tr^+ (\Gamma)$ is a Garside monoid if and only if $V(\Gamma)$ is finite and $\Gamma$ is a complete graph.
\end{thm}

\begin{rem}
Any spherical type Artin group $A$ has a natural finite quotient $W$, which is its associated Coxeter group, and the projection $A \to W$ is an important tool in the study of the group.
Such a quotient does not exist in general for a given Garside group, and determining which Garside groups have such quotients is a standard question in the theory (see for instance \cite{Gobet1}). 
Now, notice that Theorem \ref{thm2_15} provides new examples of such Garside groups.
Indeed, if $\Gamma = (\Gamma, \le, \mu, (\varphi_x)_{x \in V (\Gamma)})$ is a preGarisde trickle graph with finite $V (\Gamma)$ and $\Gamma$ complete, then, by Theorem \ref{thm2_15}, $\Tr (\Gamma)$ is a Garside group and, by Corollary \ref{corl2_7}, the quotient of $\Tr (\Gamma)$ by the relations $x^{\bar \mu(x)} = 1$, $x \in V (\Gamma)$, where $\bar \mu (x)$ is the order of $\varphi_x$ if $\varphi_x$ is different from the identity and $\bar \mu (x) = 2$ if $\varphi_x$ is the identity, is a finite trickle group.
\end{rem}

In \cite{GodPar3} four elementary questions are asked on Artin groups, questions which are unsolved to this date, and which also arise naturally for preGarside groups.
Two of them are of particular interest in this paper:
\begin{itemize}
\item[(1)]
Is any preGarside group torsion-free?
\item[(2)]
Does any preGarside group has a solution to the word problem?
\end{itemize}
Concerning preGarside trickle groups, a positive answer to Question (2) when $V (\Gamma)$ is finite is given by Corollary \ref{corl2_5}.
A positive answer to Question (1) in the case where $V (\Gamma)$ is finite is given by the following theorem which will be proved in Section \ref{sec7}.

\begin{thm}\label{thm2_16}
Let $\Gamma$ be a preGarside trickle graph with $V (\Gamma)$ finite.
Then $\Tr (\Gamma)$ is torsion-free.
\end{thm}

We do not know whether the conclusion of Theorem \ref{thm2_16} holds if $V (\Gamma)$ is infinite.

Another question more specific to preGarside groups is the following.
\begin{itemize}
\item[(3)]
Let $M$ be a preGarside monoid.
Is the natural homomorphism $\iota_M : M \to G(M)$ injective?
\end{itemize}
We know from \cite{Paris1} that the answer is yes if the monoid is an Artin monoid.
The answer is also yes if the monoid is a preGarside trickle monoid as stated in the following theorem which will be proved in Section \ref{sec7}.

\begin{thm}\label{thm2_17}
Let $\Gamma$ be a preGarside trickle graph.
Then the natural homomorphism $\iota_{\Tr^+ (\Gamma)}: \Tr^+ (\Gamma) \to \Tr (\Gamma)$ is injective.
\end{thm}

\begin{defin}
Let $M$ be a preGarside monoid and let $N$ be a submonoid of $M$.
We say that $N$ is \emph{special} if, for all $a,b \in M$, if $ab \in N$, then $a,b \in N$.
We say that $N$ is a \emph{parabolic submonoid} of $M$ if it satisfies the following three properties.
\begin{itemize}
\item[(a)]
$N$ is special.
\item[(b)]
For all $a,b \in N$, if $a \vee_L b$ exists, then $a\vee_L b \in N$.
\item[(c)]
For all $a,b \in N$, if $a \vee_R b$ exists, then $a\vee_R b \in N$.
\end{itemize}
\end{defin}

In the case of PreGarside trickle monoids, this definition, taken from \cite{GodPar2}, coincides with that given in Subsection \ref{subsec2_4} (see Proposition \ref{prop7_8}).
On the other hand, in the case of Garside monoids, it coincides with the definition of \cite{Godel1,Godel2}.
In particular, a parabolic submonoid of a Garside monoid is itself a Garside monoid.

Now, a consequence of Proposition \ref{prop7_8} combined with the results of Subsection \ref{subsec2_4} is a positive answer in the particular case of preGarside trickle monoids to three questions asked in \cite{GodPar2} for parabolic submonoids of preGarside monoids.

\begin{thm}\label{thm2_18}
Let $M$ be a preGarside trickle monoid.
\begin{itemize}
\item[(1)]
Let $N$ be a parabolic submonoid of $M$.
Then the embedding $N \hookrightarrow M$ induces an embedding $G (N) \hookrightarrow G (M)$.
So, we can consider $G(N)$ as a subgroup of $G(M)$.
\item[(2)]
Let $N$ be a parabolic submonoid of $M$.
Then $M \cap G(N) = N$.
\item[(3)]
Let $N_1$ and $N_2$ be two parabolic submonoids of $M$.
Then $N_1 \cap N_2$ is a parabolic submonoid of $M$ and $G(N_1) \cap G (N_2) = G (N_1 \cap N_2)$.
\end{itemize}
\end{thm}

Again, a detailed proof of this theorem will be given in Section \ref{sec7}.


\section{Examples}\label{sec3}

\subsection{Generalized cactus groups}\label{subsec3_1}

Let $G$ be a group.
A \emph{cactus basis} based on $G$ is a non-empty family $\FF$ of pairs $X = (G_X, \Delta_X)$, where $G_X$ is a subgroup of $G$ and $\Delta_X$ is an element of $G_X$, such that,
\begin{itemize}
\item[(a)]
$G_X \neq G_Y$ for $X,Y \in \FF$ distinct;
\item[(b)]
for $X, Y \in \FF$, if $G_Y \subset G_X$, then $(\Delta_X G_Y \Delta_X^{-1}, \Delta_X \Delta_Y \Delta_X^{-1}), (\Delta_X^{-1} G_Y \Delta_X, \Delta_X^{-1} \Delta_Y \Delta_X) \in \FF$.
\end{itemize}

Recall that, if $H_1, H_2$ are two subgroups of a group $G$, then $[H_1, H_2]$ denotes the subgroup of $G$ generated by $\{h_1 h_2 h_1^{-1} h_2^{-1} \mid h_1 \in H_1\,,\ h_2 \in H_2\}$.
In particular, the equality $[H_1, H_2]=\{ 1 \}$ means that $h_1 h_2 = h_2 h_1$ for all $h_1 \in H_1$ and $h_2 \in H_2$.
Let $\FF$ be a cactus basis based on $G$.
Let $\mu : \FF \to \N \cup \{\infty\}$ be a map which satisfies the following conditions:
\begin{itemize}
\item
for all $X \in \FF$, if $\mu (X)$ is finite, then the order of $\Delta_X$ is finite and divides $\mu (X)$;
\item
for all $X = (G_X, \Delta_X)$ and $Y = (G_Y, \Delta_Y)$ in $\FF$, if $G_Y \subset G_X$, then $\mu (Y) = \mu (Y')$, where $Y' = (\Delta_X G_Y \Delta_X^{-1}, \Delta_X \Delta_Y \Delta_X^{-1})$.
\end{itemize}
Then we define a quadruple $\Gamma_{\FF,\mu} = \Gamma = (\Gamma, \le, \mu, (\varphi_x)_{x \in V (\Gamma)})$ as follows.
The set of vertices of $\Gamma$ is a set $V(\Gamma) = \{ x_X \mid X \in \FF\}$ in one-to-one correspondence with $\FF$.
Two vertices $x_X$ and $x_Y$ are connected by an edge if either $G_X \subset G_Y$, or $G_Y \subset G_X$, or $G_X \cap G_Y = \{1\}$ and $[G_X, G_Y]= \{1\}$.
We set $x_X \le x_Y$ if $G_X \subseteq G_Y$.
The map $\mu : V (\Gamma) \to \N_{\ge 2} \cup \{ \infty\}$ is defined by $\mu (x_X) = \mu (X)$ for all $x_X \in V (\Gamma)$.
Let $x_X \in V (\Gamma)$ and $x_Y \in V (\starE_{x_X} (\Gamma))$.
If $x_Y \le x_X$, then we set $\varphi_{x_X} (x_Y) = x_{Y'}$, where $X=(G_X,\Delta_X)$, $Y=(G_Y, \Delta_Y)$ and $Y'= (\Delta_X G_Y \Delta_X^{-1}, \Delta_X \Delta_Y \Delta_X^{-1})$.
If $x_Y \not \le x_X$, then we set $\varphi_{x_X} (x_Y) = x_Y$.

The proof of the following is identical to that of Lemma \ref{lem2_1}, hence it is left to the reader.

\begin{lem}\label{lem3_1}
Let $\FF$ be a cactus basis based on a group $G$ and let $\mu : \FF \to \N_{\ge 2} \cup \{ \infty \}$ be a map as above. 
Then $\Gamma_{\FF,\mu} = (\Gamma, \le, \mu, (\varphi_x)_{x \in V(\Gamma)})$ is a trickle graph.
\end{lem}

\begin{rem}
Let $\FF$ be a cactus basis based on a group $G$ and let $\mu : \FF \to \N_{\ge 2} \cup \{ \infty \}$ be a map as above.
Then there is a homomorphism $\pi: \Tr (\Gamma_{\FF,\mu}) \to G$ which sends $x_X$ to $\Delta_X$ for all $X \in \FF$.
This homomorphism plays an important role in the study of cactus groups associated to Coxeter systems.
In that case it is known to be surjective and its kernel is known to embed into a right-angled Coxeter group (see \cite{Mosto1, Yu1}).
\end{rem}

It is not difficult to produce examples of cactus basis.
Below we give three examples which seem particularly interesting to us because they highlight the link between trickle groups and the theory of Coxeter, Artin and Garside groups.

\begin{expl1}
Let $(W,S)$ be a Coxeter system.
Recall that $\SS_{\fin}$ denotes the set of non-empty subsets $X \subseteq S$ that are irreducible and of spherical type, and that, for $X \in \SS_{\fin}$, $w_X$ denotes the longest element in $W_X$.
Then $\FF = \{(W_X,w_X) \mid X \in \SS_{\fin} \}$ is a cactus basis based on $W$.
If $\mu((W_X,w_X))=2$ for all $X \in \SS_{\fin}$, then $\Tr(\Gamma_{\FF,\mu})$ is the cactus group $C(W, S)$.
If $\mu((W_X,w_X)) = \infty$ for all $X \in \SS_{\fin}$, then $\Tr (\Gamma_{\FF,\mu})$ is the ``Artin'' version of $C(W,S)$.
We call this group the \emph{Artin-cactus group} associated with $(W,S)$ and we denote it by $\AC (W,S)$.
Like Artin groups, these groups are preGarside groups (see Theorem \ref{thm2_14}).
We also know that, when $S$ is finite, they are torsion-free and they have a solution to the word problem (see Theorem \ref{thm2_16} and Corollary \ref{corl2_5}).
\end{expl1}

\begin{expl2}
Dual structures for Artin groups were introduced by Birman--Ko--Lee \cite{BiKoLe1} for the braid groups and by Bessis \cite{Bessi1} for all Artin groups.
They are quite mysterious but they can be extremely useful for understanding some Artin groups.
For example, they are the main tool in the solutions to the word problem and to the $K (\pi, 1)$ conjecture for Artin groups of affine type \cite{McCSul1,PaoSal1}.

Let $(W,S)$ be a Coxeter system of spherical type (i.e. such that $W$ is finite).
Then another cactus basis based on $W$ can be defined as follows.
We denote by $T = \{ w s w^{-1} \mid w \in W\,,\ s \in S\}$ the set of reflections of $W$ and by $\lg_T$ the word length on $W$ with respect to $T$.
Fix a Coxeter element $c$.
Then the triple $(W,T,c)$ is called a \emph{Coxeter dual system} (see \cite[Definition 1.3.2]{Bessi1}).
Let $\delta$ be a \emph{standard parabolic element} (relative to $c$), that is, an element of $W$ such that $\lg_T (\delta) + \lg_T (\delta^{-1} c) = \lg_T (c)$.
Let $t_1 t_2\dots t_k = \delta$ be a decomposition of $\delta$ on $T$ with $k = \lg_T (\delta)$.
Then the \emph{standard (dual) parabolic subgroup} $W_\delta$ associated with $\delta$ is the subgroup generated by $\{t_1, t_2, \dots, t_k\}$.
We know from \cite[Corollary 1.6.2, Definition 1.6.3]{Bessi1} that this subgroup depends only on $\delta$ and not on the chosen decomposition and that $(W_\delta, T \cap W_\delta, \delta)$ is a dual Coxeter system.
Moreover, by \cite[Lemma 2.5]{BraWat1}, if $\delta, \delta'$ are two standard parabolic elements, then $\delta = \delta'$ if and only if $W_\delta = W_{\delta'}$.
We say that a standard parabolic subgroup $W_\delta$ (or $\delta$) is \emph{irreducible} if we cannot decompose $T \cap W_{\delta}$ into a non-trivial partition $ T_1\sqcup T_2$ such that $xy = yx$ for all $x \in T_1$ and $y \in T_2$.

Let $(W,T,c)$ be a dual Coxeter system with $W$ finite.
Then the family $\FF$ consisting of the pairs $X_\delta = (W_\delta, \delta)$, where $\delta$ is a standard (non-trivial) irreducible parabolic element, is a cactus basis based on $W$.
If we denote by $\mu (X_\delta)$ the order of $\delta$ for all $X_\delta \in \FF$, then we call $\Tr(\Gamma_{\FF, \mu})$ the \emph{dual cactus group} associated with $(W,T,c)$ and we denote it by $C^* (W,T,c)$.
If we set $\mu (X_\delta) = \infty$ for all $X_\delta \in \FF$, then we call $\Tr (\Gamma_{\FF,\mu})$ the \emph{dual Artin-cactus group} associated with $(W,T,c)$ and we denote it by $\AC^*(W,T,c)$.
We do not know when $C^* (W, T, \delta)$ is isomorphic to $C (W, S)$ and when $\AC^* (W, T, \delta)$ is isomorphic to $\AC (W, S)$.

\begin{expl}
Let $W$ be the symmetric group on three letters endowed with its dual presentation
\[
W = \langle x ,y ,z \mid x^2 = y^2 = z^2 = 1\,,\ xy = yz = zx \rangle\,.
\]
For the Coxeter element $c = xy$, there are four standard non-trivial parabolic elements, $x, y, z, u=xy$, and the dual cactus group and dual Artin-cactus group have the following presentations:
\begin{gather*}
C^* = \langle x, y, z, u \mid x^2 = y^2 = z^2 = u^3 = 1\,,\ x u = u z\,,\ y u = u x\,, z u = u y \rangle\,,\\
\AC^* = \langle x, y, z, u \mid x u = u z\,,\ y u = u x\,, z u = u y \rangle\,.
\end{gather*}
\end{expl}
\end{expl2}

\begin{expl3}
Examples 1 and 2 extend naturally to preGarside monoids (or groups) as follows.
Let $M$ be a preGarside monoid.
An \emph{atom} of $M$ is an element $a$ in $M \setminus \{1\}$ such that, for all $a_1, a_2 \in M$, if $a = a_1 a_2$, then $a_1 = 1$ or $a_2 = 1$.
We denote by $\AA$ the set of atoms of $M$.
It is easily shown that $\AA$ generates $M$ and, if $N$ is a parabolic submonoid, then $\AA \cap N$ is the set of atoms of $N$ and it generates $N$.
Let $N$ be a Garside parabolic submonoid of $M$.
By \cite{Dehor1} the submonoid $N$ admits a least Garside element, which we denote by $\Delta_N$, in the sense that, if $\Delta'_N$ is another Garside element of $N$, then $\Delta_N \in \Div (\Delta'_N)$.
We know that $N$ embeds into $G(N)$ and that the conjugation by $\Delta_N$ leaves $N$ invariant, that is, the map $N \to N$, $ g \mapsto \Delta_N g \Delta_N^{-1}$, is well-defined and is an automorphism.
Moreover, if $N_1$ is a parabolic submonoid of $N$, then $N_1$ and $\Delta_N N_1 \Delta_N^{-1}$ are both Garside parabolic submonoids of $N$ and $\Delta_N \Delta_{N_1} \Delta_N = \Delta_{\Delta_N N_1 \Delta_N^{-1}}$.

Let $M$ be a preGarside monoid and let $\AA$ be the set of its atoms.
We say that a parabolic submonoid $N$ of $M$ is \emph{irreducible} if we cannot decompose $\AA \cap N$ into a non-trivial partition $\AA_1 \sqcup \AA_2$ such that $ab = ba$ for all $a \in \AA_1$ and $b \in \AA_2$.
Then the family $\FF$ consisting of the pairs $X_N = (G(N), \Delta_N)$, with $N$ an irreducible Garside parabolic submonoid of $M$, is a cactus basis.
Note that we do not specify that $\FF$ is a cactus basis based on a given group, and actually we do not know if $\FF$ is a cactus basis based on a group ($G(M)$, for example), but the reader will easily extend the definition in this context.

To complete the data we also need to define the labeling $\mu: \FF \to \N_{\ge 2} \cup \{\infty\}$.
There are two natural choices.
The first consists in taking $\mu(X_N)$ the order of the conjugation by $\Delta_N$ in $\Aut (G(N))$ if different from $1$ and $2$ if the conjugation by $\Delta_N$ is the identity.
Note that $\Delta_N$ has always infinite order.
The second consists in setting $\mu (X_N) = \infty$ for all $X_N \in \FF$.
\end{expl3}


\subsection{Virtual cactus groups}\label{subsec3_2}

\begin{defin}
Let $n \ge 2$.
The \emph{virtual cactus group} on $n$ strands, denoted by $\VJ_n$, is the group defined by the presentation with generators
\[
x_{p,q}\,,\ 1\le p < q \le n \quad \text{and} \quad \rho_i\,,\ 1 \le i \le n-1\,,
\]
and relations:
\begin{itemize}
\item[(j1)]
$x_{p,q}^2 = 1$ for $1 \le p < q \le n$,
\item[(j2)]
$x_{p,q} x_{m,r} = x_{m,r} x_{p,q}$ for $[p,q] \cap [m,r] = \emptyset$,
\item[(j3)]
$x_{p,q} x_{m,r} = x_{p+q-r,p+q-m} x_{p,q}$ for $[m,r] \subset [p,q]$,
\item[(s1)]
$\rho_i^2 = 1$ for $1 \le i \le n-1$,
\item[(s2)]
$\rho_i \rho_j \rho_i = \rho_j \rho_i \rho_j$ for $ |i-j| = 1$,
\item[(s3)]
$\rho_i \rho_j = \rho_j \rho_i$ for $|i-j| > 1$,
\item[(m1)]
$\rho_i x_{p,q} = x_{p,q} \rho_i$ for $i < p-1$ and for $i \ge q+1$,
\item[(m2)]
$x_{p,q} \rho_{q} \rho_{q-1} \dots \rho_p = \rho_{q} \rho_{q-1} \dots \rho_p x_{p+1,q+1} $ for $1 \le p < q <n$.
\end{itemize}
\end{defin}

As pointed out in the introduction, the elements of the virtual cactus group are represented by planar diagrams.
To the usual diagrams which represent the elements of $J_n$ we add virtual crossings keeping the principle that two arcs connecting the same points and passing only through virtual crossings are equivalent.
For instance, the generators $x_{p,q}$ and $\rho_i$ are depicted in Figure \ref{fig3_1} and the relation (m2) is illustrated in Figure \ref{fig3_2}.

\begin{figure}[ht!]
\begin{center}
\includegraphics[width=6cm]{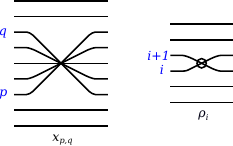}
\caption{Generators of $\VJ_n$}\label{fig3_1}
\end{center}
\end{figure}

\begin{figure}[ht!]
\begin{center}
\includegraphics[width=10.2cm]{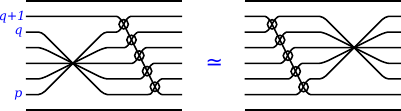}
\caption{Relation (m2) in the presentation of $\VJ_n$}\label{fig3_2}
\end{center}
\end{figure}

Relations (j1--j3) are the relations which define the cactus group $J_n$, and they induce a homomorphism $\iota_J : J_n \to \VJ_n$ which sends $x_{p,q}$ to $x_{p,q}$ for all $1 \le p < q \le n$.
Relations (s1--s3) are the relations in the standard presentation of the symmetric group $\SSS_n$, and they induce a homomorphism $\iota_S: \SSS_n \to \VJ_n$ which sends $s_i=(i,i+1)$ to $\rho_i$ for all $1 \le i \le n-1$.
Relations (m1--m2), called the \emph{mixed relations}, make that two arcs connecting the same points and passing through only virtual crossings are equivalent.

There is a homomorphism $\pi_K: \VJ_n \to \SSS_n$ defined by
\[
\pi_K (x_{p,q}) = 1 \text{ for } 1 \le p < q \le n \text{ and } \pi_K (\rho_i) = s_i \text{ for } 1 \le i \le n-1\,.
\]
Note that $\pi_K \circ \iota_S = \id_{\SSS_n}$, hence $\pi_K$ is surjective, $\iota_S$ is injective, and there is a semi-direct product decomposition $\VJ_n = \KVJ_n \rtimes \SSS_n$, where $\KVJ_n = \Ker (\pi_K)$.

The group $\VJ_n$ is not a trickle group, but $\KVJ_n$ is (see Proposition \ref{prop3_7}).
This enables to solve the word problem in $\VJ_n$ but also to prove other results on $\VJ_n$ such as the fact that $\iota_J : J_n \to \VJ_n$ is injective (see Corollary \ref{corl3_9}).
A different proof of this fact can be found in \cite[Corollary 11.4]{IKLPR1}.

\begin{rem}
For $1 \le p < q \le n$ we denote by $t(p,q)$ the permutation which sends $p$ to $q$ and $q$ to $p$, $p+1$ to $q-1$ and $q-1$ to $p+1$, and so on.
More precisely, for all $0 \le i \le q-p$, $t(p,q)$ sends $p+i$ to $q-i$, and $t(p,q)(k) =k$ if $k \not \in [p,q]$.
There exists another natural epimorphism from $\VJ_n$ to $\SSS_n$, denoted by $\pi_P: \VJ_n \to \SSS_n$, defined by
\[
\pi_P (x_{p,q}) = t(p,q) \text{ for } 1 \le p < q \le n \text{ and } \pi_P (\rho_i) = s_i \text{ for } 1 \le i \le n-1\,.
\]
The kernel of $\pi_P$ is the \emph{pure virtual cactus group} $\PVJ_n$ studied in \cite{IKLPR1}, and it is different from the group $\KVJ_n$ that we study.
A presentation for $\PVJ_n$ is given in \cite[Lemma 10.12]{IKLPR1}.
\end{rem}

For $2 \le \ell \le n$ we set $V_\ell = \{ (t_1, \dots, t_\ell) \in \{1, \dots, n\}^\ell \mid t_i \neq t_j \text{ for } i \neq j\}$.
Let $y = (t_1, \dots, t_\ell) \in V_\ell$.
We choose $w \in \SSS_n$ such that $w(i) = t_i$ for all $1 \le i \le \ell$ and we set $\delta_y = \iota_S(w) \,x_{1,\ell}\, \iota_S(w)^{-1}$.
Note that $\delta_y \in \KVJ_n$.

\begin{lem}\label{lem3_2}
Let $2 \le \ell \le n$ and let $y \in V_\ell$.
Then the above definition of $\delta_y$ does not depend on the choice of $w \in \SSS_n$.
\end{lem}

\begin{proof}
We set $y = (t_1, \dots, t_\ell)$ and we take $w, w' \in \SSS_n$ such that $w(i) = w'(i) = t_i$ for all $1 \le i \le \ell$.
We have $w^{-1} w'(i) = i$ for all $i \in \{1, \dots, \ell\}$, hence $w^{-1} w'$ lies in the subgroup of $\SSS_n$ generated by $s_{\ell+1}, \dots, s_{n-1}$, and therefore $\iota_S (w^{-1} w')$ lies in the subgroup of $\VJ_n$ generated by $\rho_{\ell+1}, \dots, \rho_{n-1}$.
Relations (m1) imply that $\rho_i x_{1,\ell} = x_{1,\ell} \rho_i$ for all $i \in \{\ell+1, \dots, n-1\}$, hence  $\iota_S (w^{-1} w') \, x_{1,\ell}\, \iota_S (w'^{-1} w) = x_{1,\ell}$, and therefore $\iota_S(w) \,x_{1,\ell}\, \iota_S (w)^{-1} = \iota_S (w') \,x_{1,\ell} \, \iota_S (w')^{-1}$.
\end{proof}

We set
\[
V = \bigcup_{\ell=2}^n V_\ell\,.
\]
We consider the action of $\SSS_n$ on $V$ defined by, for $w \in \SSS_n$ and $y = (t_1, \dots, t_\ell) \in V$,
\[
w \cdot y = (w (t_1) , \dots, w (t_\ell))\,.
\]

\begin{lem}\label{lem3_3}
\begin{itemize}
\item[(1)]
Let $w \in \SSS_n$ and $y \in V$.
Then $\iota_S (w) \, \delta_y \, \iota_S (w)^{-1} = \delta_{w \cdot y}$.
\item[(2)]
Let $1 \le p < q \le n$.
Then $\delta_{(p,p+1,\dots,q)} = x_{p,q}$.
\end{itemize}
\end{lem}

\begin{proof}
We set $y = (t_1, \dots, t_\ell)$ and we choose $w' \in \SSS_n$ such that $w' \cdot (1, \dots, \ell) = (t_1, \dots, t_\ell)$.
Notice that $w \cdot y = (ww'(1), \dots, ww'(\ell))= ww' \cdot (1, \dots, \ell)$.
Then
\[
\iota_S (w) \, \delta_y \, \iota_S (w)^{-1} = \iota_S (w) \, \iota_S (w') \, x_{1,\ell} \, \iota_S (w')^{-1} \, \iota_S (w)^{-1} = \iota_S (w w') \, x_{1,\ell} \, \iota_S (w w')^{-1} = \delta_{w \cdot y}\,.
\]

Let $u = s_1 s_2 \dots s_{n-1} = (1, 2, \dots, n)$.
We easily see using Relations (m1) and (m2) that, for all $1 \le m < r < n$,
\[
\iota_S (u) \, x_{m,r}\, \iota_S (u)^{-1} = x_{m+1,r+1}\,.
\]
On the other hand, $u \cdot (m, m+1, \dots, r) = (u(m), u(m+1), \dots, u(r)) = (m+1, m+2, \dots, r+1)$. 
So,
\[
\delta_{(p, p+1, \dots, q)} = \iota_S (u)^{p-1} \, x_{1,q-p+1} \, \iota_S(u)^{-p+1} = x_{p,q}\,.
\proved
\]
\end{proof}

\begin{lem}\label{lem3_4}
The set $\{ \delta_y \mid y \in V\}$ generates $\KVJ_n$.
\end{lem}

\begin{proof}
Let $g \in \KVJ_n$.
Since $g$ belongs to $\VJ_n$, there exist $ k \ge 0$, $p_1, \dots, p_k, q_1, \dots, q_k \in \N$, $w_0, w_1, \dots, w_k \in \SSS_n$ such that $1 \le p_i <q_i \le n$ for all $1 \le i \le k$, and
\[
g = \iota_S (w_0) \, x_{p_1,q_1} \, \iota_S (w_1) \dots x_{p_k,q_k} \, \iota_S(w_k)\,.
\]
For $1 \le i \le k+1$ we set $v_i = w_0 w_1 \dots w_{i-1}$.
Then
\[
g = \iota_S (v_1) \, x_{p_1,q_1} \, \iota_S (v_1)^{-1} \, \iota_S (v_2) \, x_{p_2,q_2} \, \iota_S (v_2)^{-1} \dots \iota_S(v_k) \, x_{p_k,q_k} \, \iota_S (v_k)^{-1} \, \iota_S (v_{k+1})\,.
\]
Since $g \in \KVJ_n$, we have $1 = \pi_K (g)= v_{k+1}$, hence
\[
g = \iota_S (v_1) \, x_{p_1,q_1} \, \iota_S (v_1)^{-1} \, \iota_S (v_2) \, x_{p_2,q_2} \, \iota_S (v_2)^{-1} \dots \iota_S(v_k) \, x_{p_k,q_k} \, \iota_S (v_k)^{-1}  = \delta_{y_1} \delta_{y_2} \dots \delta_{y_k}\,,
\]
where $y_i = v_i \cdot (p_i, p_i + 1, \dots, q_i)$ for all $1 \le i \le k$.
\end{proof}

Now we define a quadruple $\Gamma = (\Gamma, \le, \mu, (\varphi_x)_{x \in V(\Gamma)})$ as follows.
The set of vertices of $\Gamma$ is $V = \bigcup_{\ell=2}^n V_\ell$.
We set $\mu (y) = 2$ for all $y \in V$.
For $y = (t_1, \dots, t_\ell)$ and $y' = (t_1', \dots, t_k')$ in $V$, we set $y' \le y$ if there exist $1 \le p < q \le \ell$ such that $y'= (t_p, t_{p+1}, \dots, t_q)$.
Let $y = (t_1, \dots, t_\ell)$ and $y' = (t_1', \dots, t_k')$ in $V$ such that $y \neq y'$.
Then $\{y, y'\} \in E (\Gamma)$ if and only if either $y \le y'$ or $y' \le y$ or $\{t_1, \dots, t_\ell\} \cap \{t_1', \dots, t_k' \} = \emptyset$.
Let $y = (t_1, \dots, t_\ell) \in V$ and $y' \in \starE_y (\Gamma)$.
If $y' \not \le y$, then we set $\varphi_y (y') = y'$.
Suppose $y' \le y$.
We know that there exist $1 \le p < q \le \ell$ such that $y' = (t_p, \dots, t_q)$.
Then we set $\varphi_y (y') = (t_{1 + \ell - q}, \dots, t_{1 + \ell - p})$.

\begin{lem}\label{lem3_5}
The above defined quadruple $\Gamma = (\Gamma, \le, \mu, (\varphi_x)_{x\in V(\Gamma)})$ is a trickle graph.
\end{lem}

\begin{proof}
Conditions (a) and (d) are satisfied by definition, and Conditions (b) and (f) are obviously satisfied. 
It is easily seen that, for all $x \in V$, $\varphi_x^2=1$, hence Condition (e) is also satisfied.

We show that $\Gamma$ satisfies Condition (c). 
Let $x \in V$ and $y, z \in \starE_x (\Gamma)$ be such that $z \le y$. 
If $y || x$ then, by Condition (b), $z || x$, hence $\varphi_x (z) = z \le y = \varphi_x (y)$. 
Suppose $z \le y < x$. 
Set $x = (t_1, \dots, t_\ell)$. 
There exist $1 \le p \le k < m \le q \le \ell$ such that $y = (t_p, \dots, t_q)$ and $z = (t_k, \dots, t_m)$. 
We have $\varphi_x (y) = (t_{1+\ell-q}, \dots, t_{1+\ell-p})$, $\varphi_x (z) = (t_{1+\ell-m}, \dots, t_{1+\ell -k})$ and $1 \le 1+\ell-p \le 1+\ell-m < 1+\ell-k \le 1 + \ell -p \le \ell$, hence $\varphi_x (z) \le \varphi_x (y)$. 
This shows that, if $z \le y$, then $\varphi_x (z) \le \varphi_x (y)$. 
But $\varphi_x^2 = 1$, hence we actually have $z \le y$ if and only if $\varphi_x (z) \le \varphi_x (y)$.

Now, we show that $\Gamma$ satisfies Condition (g). 
Let $x, y, z \in V$ be such that $z \le y \le x$. 
Set $x = (t_1, \dots, t_\ell)$. 
There exist $1 \le p \le k < m \le q \le \ell$ such that $y = (t_p, \dots, t_q)$ and $z = (t_k, \dots, t_m)$. 
Let $y' = \varphi_x (y) = (t_{1+\ell-q}, \dots, t_{1+\ell-p})$. 
Then a direct calculation shows that
\[
(\varphi_x \circ \varphi_y) (z) = (t_{1 + \ell - p - q + k}, \dots, t_{1 + \ell - p - q + m}) = (\varphi_{y'} \circ \varphi_x) (z)\,.
\proved
\]
\end{proof}

\begin{lem}\label{lem3_6}
There is a surjective homomorphism $\Phi: \Tr (\Gamma) \to \KVJ_n$ which sends $x$ to $\delta_x$ for all $x \in V$.
\end{lem}

\begin{proof}
Let $x = (t_1, \dots, t_\ell) \in V$.
Let $w \in \SSS_n$ be such that $w(i) = t_i$ for all $1 \le i \le n$.
Then
\[
\delta_x^2 = \iota_S (w)\, x_{1,\ell}^2\, \iota_S (w)^{-1} = \iota_S (w) \, \iota_S (w)^{-1} = 1\,.
\]
Let $x, y \in V$ be such that $\{x, y\} \in E_{||} (\Gamma)$.
We set $x=(t_1, \dots, t_\ell)$ and $y = (t_1', \dots, t_k')$.
By definition we have $\{t_1, \dots, t_\ell\} \cap \{t_1', \dots, t_k'\} = \emptyset$, hence we can choose $w \in \SSS_n$ such that $ w(i)=t_i$ for all $1 \le i \le \ell$ and $w (\ell + j) = t_j'$ for all $1 \le j\le k$.
Then, by Lemma \ref{lem3_3},
\[
\delta_x \delta_y = \iota_S(w) \, x_{1,\ell} x_{\ell+1,\ell+k}\, \iota_S(w)^{-1} = \iota_S(w) \, x_{\ell+1,\ell+k} x_{1,\ell} \, \iota_S(w)^{-1} = \delta_y \delta_x\,.
\]
Let $x, y \in V$ be such that $y \le x$.
We set $x = (t_1, \dots, t_\ell)$ and $y = (t_p, \dots, t_q)$, where $1 \le p < q \le \ell$.
We choose $w \in \SSS_n$ such that $w (i) = t_i$ for all $1 \le i \le \ell$.
Then, by Lemma \ref{lem3_3},
\[
\delta_x \delta_y = \iota_S(w) \, x_{1,\ell} x_{p,q} \, \iota_S(w)^{-1} = \iota_S(w) \, x_{1+\ell-q,1+\ell-p} x_{1,\ell} \, \iota_S(w)^{-1} = \delta_{\varphi_x(y)} \delta_x\,.
\]
This shows that there is a homomorphism $\Phi: \Tr (\Gamma) \to \KVJ_n$ which sends $x$ to $\delta_x$ for all $x \in V$.
This homomorphism is surjective because, by Lemma \ref{lem3_4}, $\{\delta_x \mid x \in V\}$ generates $\KVJ_n$.
\end{proof}

\begin{prop}\label{prop3_7}
The homomorphism $\Phi: \Tr (\Gamma) \to \KVJ_n$ of Lemma \ref{lem3_6} is an isomorphism.
\end{prop}

\begin{proof}
The action of $\SSS_n$ on $V$ extends to an action of $\SSS_n$ on $\Tr (\Gamma)$, hence we can consider the semi-direct product $G = \Tr (\Gamma ) \rtimes \SSS_n$.
Furthermore, by Lemma \ref{lem3_3}, the map $\hat \Phi: G \to \VJ_n$ defined by $\hat \Phi (g,w) = \Phi(g) \, \iota_S (w)$ is a group homomorphism.

We show that there is a homomorphism $\Psi: \VJ_n \to G$ which sends $x_{p,q}$ to $x_{p,q}' = (p, p+1, \dots, q ) \in V \subseteq \Tr (\Gamma)$ for $1 \le p < q \le n$, and which sends $\rho_i$ to $s_i = (i, i + 1) \in \SSS_n$ for $1 \le i \le n-1$.
The relations (j1) $x_{p,q}'^2 = 1$ for $1 \le p < q \le n$, (j2) $x_{p,q}' x_{m,r}' = x_ {m,r}' x_{p,q}'$ for $[p,q] \cap [m,r] = \emptyset$, and (j3) $x_{p,q}' x_{m,r}' = x_{p+q-r,p+q-m}' x_{p,q}'$ for $[m,r] \subset [p,q]$, follow from the definitions of the trickle graph $\Gamma$ and of the associated presentation of $\Tr (\Gamma)$.
The relations (s1) $s_i^2 = 1$ for $1 \le i \le n-1$, (s2) $s_i s_j s_i = s_j s_i s_j$ for $|i - j| = 1$, and (s3) $s_i s_j = s_j s_i$ for $|i - j| > 1$, are the relations of the standard presentation of the symmetric group $\SSS_n$.
Finally, the relations (m1) $s_i x_{p,q}' = x_{p,q}' s_i$ for $i < p-1$ and for $i \ge q+1$, and (m2) $ x_{p,q}'s_{q} s_{q-1} \dots s_p=s_{q} s_{q-1} \dots s_p x_{p+1,q+1}'$ for $1 \le p < q <n$, follow from the action of $\SSS_n$ on $V$.

For every $w \in \SSS_n$ we have $(\Psi \circ \hat \Phi) (w) = \Psi (\iota_S (w)) = w$.
On the other hand, let $x = (t_1, \dots, t_\ell) \in V$.
We choose $w \in \SSS_n$ such that $w \cdot (1, \dots, \ell) = x = (t_1, \dots, t_\ell)$.
Then
\[
(\Psi \circ \hat \Phi) (x) = \Psi (\delta_x) = \Psi (\iota_S (w)\, x_{1,\ell}\, \iota_S (w)^{-1}) = w \, x_{1,\ell}'\, w^{-1} = w \cdot x_{1,\ell}' = x\,.
\]
This shows that $\Psi \circ \hat \Phi = \id_G$, hence $\hat \Phi$ is injective, and therefore $\Phi$ is injective.
Since we already know that $\Phi$ is surjective, this completes the proof of the proposition.
\end{proof}

Since $\KVJ_n$ is a finite index subgroup of $\VJ_n$, the following follows from Proposition \ref{prop3_7} and Corollary \ref{corl2_5}.

\begin{corl}\label{corl3_8}
Let $n \ge 2$.
Then $\VJ_n$ has a solution to the word problem.
\end{corl}

Another consequence of Proposition \ref{prop3_7} is the following.

\begin{corl}[Ilin--Kamnitzer--Li--Przytycki--Rybnikov \cite{IKLPR1}]\label{corl3_9}
Let $n \ge 2$.
Then the natural homomorphism $\iota_J: J_n \to \VJ_n$ is injective.
\end{corl}

\begin{proof}
Recall that, for $1 \le p < q \le n$, we set $x_{p,q}' = (p,p+1,\dots,q) \in V$.
Set $U = \{x_{p,q}' \mid 1 \le p < q \le n\}$ and denote by $\Gamma_1$ the full subgraph of $\Gamma$ spanned by $U$.
It is easily seen that $\Gamma_1$ is a parabolic subgraph of $\Gamma$.
We denote by $\le_1$ and $\mu_1$ the restrictions of $\le$ and $\mu$ to $V(\Gamma_1)$, respectively, and, for every $x \in V (\Gamma_1)$, we denote by $\varphi_{1,x}$ the restriction of $\varphi_x$ to $\starE_x (\Gamma_1)$.
Then $\Gamma_1 = (\Gamma_1, \le_1 ,\mu_1, (\varphi_{1,x})_{x\in U})$ is a trickle graph and, by Theorem \ref{thm2_10}, the inclusion $U \hookrightarrow V$ induces an injective homomorphism $\iota: \Tr (\Gamma_1) \to \Tr (\Gamma)$.
We easily see that the map $\{x_{p,q} \mid 1 \le p < q \le n\} \to U$, $x_{p,q} \mapsto x_{p,q}'$, induces an isomorphism $\phi: J_n \to \Tr (\Gamma_1)$ and that the composition $\iota \circ \phi: J_n \to \Tr (\Gamma) = \KVJ_n \subseteq \VJ_n$ coincides with $\iota_J: J_n \to \VJ_n$.
So, $\iota_J: J_n \to \VJ_n$ is injective.
\end{proof}


\subsection{Thompson group F}\label{subsec3_3}

There are several equivalent definitions of Thompson group $F$.
In this paper we consider the one as a subgroup of the group of orientation preserving piecewise linear homeomorphisms of the real line (see \cite[Theorem 1.4.1]{Buril1}).

\begin{defin}
Let $f: \R \to \R$ be an orientation preserving piecewise linear homeomorphism of the real line.
We say that $f$ belongs to \emph{Thompson group $F$} if there exist two finite sequences with the same cardinality $u_0, u_1, \dots, u_n$ and $v_0, v_1, \dots, v_n$ in $\Z [\frac{1}{2}]$ such that:
\begin{itemize}
\item
$u_0 < u_1 < \cdots < u_n$ and $v_0 <v_1 < \cdots <v_n$,
\item
$f$ sends linearly the interval $[u_{i-1}, u_i]$ onto $[v_{i-1}, v_i]$ for all $i \in \{1, \dots, n\}$,
\item
$f$ sends the interval $(-\infty, u_0]$ onto $(-\infty, v_0]$ via the translation $t \mapsto t + v_0 - u_0$, and $f$ sends the interval $[u_n, +\infty)$ onto $[v_n, +\infty)$ via the translation $t \mapsto t + v_n - u_n$.
\end{itemize}
\end{defin}

Our goal in this subsection is to provide a trickle graph $\Gamma = (\Gamma, \le, \mu,(\varphi_x)_{x \in V(\Gamma)})$ such that $F \simeq \Tr (\Gamma)$.
We start by defining the vertex set $V = V (\Gamma)$.

For $p \in \N$ we define a subset $V_p \subseteq \Z [\frac{1}{2}]$ and a map $s_p: V_p \to V_p$ which satisfies $x <s_p (x)$ and $[x, s_p(x)] \cap V_p = \{x, s_p(x)\}$ for all $x \in V_p$ by induction on $p$.
We set $V_0 = \Z$ and we set $s_0 (x) = x + 1$ for all $x \in V_0$.
Suppose $V_p$ and $s_p: V_p \to V_p$ are defined.
Let $x \in V_p$.
Let $x' = s_p(x)$ and, for $k \in \N$, let $u_k = u_{p+1,k} (x) = x' - \frac{x'-x}{2^k}$ (see Figure \ref{fig3_3}).
We set $V_{p+1} (x) = \{u_{p+1,k} (x) \mid k \in \N\}$.
Then
\[
V_{p+1} = \bigsqcup_{x \in V_p} V_{p+1} (x)\,.
\]
If $x \in V_p$ and $k \in \N$, then we set $s_{p+1} (u_{p+1, k} (x)) = u_{p+1, k+1} (x)$.
So, in Figure \ref{fig3_3}, $s_{p+1}$ sends $u_k$ to $u_{k+1}$.

\begin{figure}[ht!]
\begin{center}
\includegraphics[width=8cm]{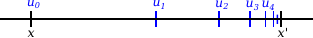}
\caption{Definition of $V_{p+1}$}\label{fig3_3}
\end{center}
\end{figure}

Observe that $V_p\subset V_{p+1}$ for all $p \in \N$ and that $\bigcup_{p \in \N} V_p = \Z [\frac{1}{2}]$.
So, $\{ V_p \mid p \in \N \}$ is a filtration of $\Z [\frac{1}{2}]$.
Now, we set $V = \Z [\frac{1}{2}] \sqcup \{\infty\}$.

We endow $\Z [\frac{1}{2}]$ with the order induced by the one of $\R$, and we set $x < \infty$ for all $x \in \Z [\frac{1}{2}]$.
So, the order $\le$ on $V$ is a total order and $\Gamma$ is defined to be the complete graph on $V$.
We set $\mu (x) = \infty$ for all $x \in V$.
It remains to define $\varphi_x$ for all $x \in V$, but for this we first need to define the generating system $\{h_x \mid x\in V\}$ for $F$ corresponding to the trickle structure.

In order to define the generating system $\{h_x \mid x\in V\}$, in addition to the sets $V_p$ and to the maps $s_p: V_p \to V_p$, we need other maps $t_p: V_p \to V_p$, $p \in \N$, defined by induction on $p$ as follows.
We set $t_0 (x) = x - 1$ for all $x \in V_0 = \Z$.
Let $p \ge 1$ and let $x \in V_p$.
If $x \in V_{p-1}$, then we set $t_p (x) = t_{p-1} (x)$.
If $x \not \in V_{p-1}$, then there exists $y \in V_{p-1}$ and $k \ge 1$ such that $x = u_{p, k} (y)$, and we set $t_p (x) = u_{p, k-1} (y)$.
Note that, for all $x \in V_p$, there exists $N \ge 0$ such that $t_p^n (x) \in V_0 = \Z$ for all $n\ge N$.

The map $h_\infty$ is defined to be the translation $h_\infty: \R \to \R$, $t \mapsto t-1$.
We define $h_x$ for $x \in \Z [\frac{1}{2}]$ as follows.
Let $p$ be the least integer $\ge 0$ such that $x \in V_p$.
We set $x' = t_p (x)$ (see Figure \ref{fig3_4}).
For all $k \ge 0$ we set $v_k = s_{p+1}^k (x')$ and $v_{-k} = t_p^k (x')$.
We have
\begin{gather*}
\cdots < v_{-k} < \cdots < v_{-1} < v_0 = x' < v_1 < \cdots < v_k < \cdots\,,\\
\lim\limits_{k \to -\infty} v_k = -\infty \text{ and } \lim\limits_{k \to +\infty} v_k =x\,.
\end{gather*}
Then $h_x: \R \to \R$  sends linearly the interval $[v_k, v_{k+1}]$ onto $[v_{k-1}, v_k]$ (preserving the orientation) for all $k \in \Z$ and it is the identity on $[x, +\infty)$.

\begin{figure}[ht!]
\begin{center}
\includegraphics[width=8.8cm]{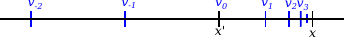}
\caption{Definition of $h_x$}\label{fig3_4}
\end{center}
\end{figure}

The proof of the following is left to the reader.

\begin{lem}\label{lem3_10}
Let $x \in \Z [\frac{1}{2}]$.
Let $(v_k)_{k\in \Z}$ be as defined above.
\begin{itemize}
\item[(1)]
$h_x$  sends linearly the interval $[v_1,x]$ onto $[v_0,x]$.
\item[(2)]
$h_x$ sends the interval $[x,+\infty)$ onto $[x,+\infty)$ via the identity.
\item[(3)]
Let $N$ be the least integer $\ge 0$ such that $v_{-N} \in V_0 = \Z$.
Then $h_x$ sends $(-\infty,v_{-N}]$ onto $(-\infty,v_{-N-1}]=(-\infty,v_{-N}-1]$ via the translation $t \mapsto t-1$.
\end{itemize}
In particular, $h_x\in F$.
\end{lem}

Note that we also have $h_\infty \in F$.
Furthermore, we know that $F$ is generated by $\{h_0,h_\infty\}$ (see \cite[Corollary 2.6]{CaFlPa1}), hence:

\begin{lem}\label{lem3_11}
The set $\{h_x \mid x \in V\}$ generates $F$.
\end{lem}

The next result is an observation.

\begin{lem}\label{lem3_12}
Let $x \in V$.
Then $h_x (\Z [\frac{1}{2}]) = \Z [\frac{1}{2}]$, $h_x$ preserves the order of $\Z [\frac{1}{2}]$, and $h_x(y)=y$ for all $y \in \Z [\frac{1}{2}]$ such that $y \ge x$.
\end{lem}

For $x \in V$, we set $\varphi_x (\infty) = \infty$ and $\varphi_x (y) = h_x(y)$ for all $y \in \Z [\frac{1}{2}]$.

Now we need to show that $\Gamma = (\Gamma, \le, \mu, (\varphi_x)_{x \in V})$ is a trickle graph and that the map $V \to F$, $ x \mapsto h_x$, induces an isomorphism $\Tr (\Gamma) \to F$.
We begin by proving the following lemma which is a key point in the proofs that will follow.

\begin{lem}\label{lem3_13}
Let $x \in V$ and let $y \in \Z [\frac{1}{2}]$ be such that $y < x$.
Let $y' = h_x (y)$.
Then $h_x \circ h_y \circ h_x^{-1} = h_{y'}$.
\end{lem}

\begin{proof}
First assume $x=\infty$.
Then $y'=y-1$ and by construction $h_{y'} = h_{y-1} = h_\infty \circ h_y \circ h_\infty^{-1}$.

Now assume $x \in \Z [\frac{1}{2}]$.
Let $p$ be the least integer $\ge 0$ such that $x \in V_p$.
Set $x' = t_p (x)$.
For all $k \ge 0$ we set $v_k = s_{p+1}^k (x')$ and $v_{-k} = t_p^k (x')$.
Then $h_x: \R \to \R$ sends linearly the interval $[v_k, v_{k+1}]$ onto $[v_{k-1},v_k]$ for all $k \in \Z$ and it is the identity on $[x, +\infty)$.
We denote by $\alpha: (-\infty, x) \to \R$ the orientation preserving piecewise linear homeomorphism which sends linearly $[v_k, v_{k+1}]$ onto $[k, k+1]$ for all $k \in \Z$ (see Figure \ref{fig3_5}).
Since $h_x$, $h_y$ and $h_{y'}$ induce homeomorphisms of $(-\infty, x)$, we can consider the homeomorphisms of the real line $g_x = \alpha \circ h_x|_{ (-\infty, x)} \circ \alpha^{-1}$, $g_y = \alpha \circ h_y|_{(-\infty, x)} \circ \alpha^{-1}$ and $ g_{y'} = \alpha \circ h_{y'}|_{(-\infty, x)}\circ \alpha^{-1}$.
The following properties are easily observed:
\begin{itemize}
\item
$\alpha(\Z [\frac{1}{2}] \cap (-\infty, x)) = \Z[\frac{1}{2}]$,
\item
$g_x = h_\infty$, $g_y = h_{\alpha(y)}$ and $g_{y'} = h_{\alpha(y')}$.
\end{itemize}
We have $\alpha (y') = (\alpha \circ h_x) (y) = g_x (\alpha (y)) = h_\infty (\alpha (y))$, hence, by the case $x =\infty$ already treated,
\[
g_{y'} = h_{\alpha(y')} = h_\infty \circ h_{\alpha(y)} \circ h_\infty^{-1} = g_x \circ g_y \circ g_x^{-1}\,.
\]
Since $h_x|_{[x, +\infty)} = h_y|_{[x, +\infty)} = h_{y'}|_{[x, +\infty)} = \id_{[x ,+\infty)}$, we conclude that $h_{y'} =h_x \circ h_y \circ h_x^{-1}$.
\end{proof}

\begin{figure}[ht!]
\begin{center}
\includegraphics[width=8.8cm]{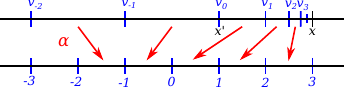}
\caption{Definition of $\alpha$}\label{fig3_5}
\end{center}
\end{figure}

Now we prove that $\Gamma = (\Gamma, \le, \mu, (\varphi_x)_{x\in V(\Gamma)})$ is a trickle graph.

\begin{lem}\label{lem3_14}
The above defined quadruple $\Gamma = (\Gamma, \le, \mu, (\varphi_x)_{x\in V})$ is a trickle graph.
\end{lem}

\begin{proof}
We should prove that $(\Gamma, \le, \mu, (\varphi_x)_{x\in V(\Gamma)})$ satisfies Conditions (a) to (g) of the definition of a trickle graph.
Condition (a) is obvious because $\Gamma$ is a complete graph.
Condition (b) is empty because $E_{||}(\Gamma)=\emptyset$.
Condition (c) holds because $\varphi_x: V \to V$ preserves the order of $V$ for all $x \in V$.
Condition (d) holds by Lemma \ref{lem3_12}.
Conditions (e) and (f) are obvious because $\mu (x) = \infty$ for all $x \in V$.
Condition (g) is a direct consequence of Lemma \ref{lem3_13}.
\end{proof}

Now, we prove that $\Tr (\Gamma)$ is isomorphic to $F$.

\begin{thm}\label{thm3_15}
Let $\Gamma = (\Gamma, \le, \mu, (\varphi_x)_{x\in V})$ be the trickle graph of Lemma \ref{lem3_14}.
Then $\Tr (\Gamma)$ is isomorphic to $F$.
\end{thm}

\begin{proof}
By Lemma \ref{lem3_13} there exists a homomorphism $\Phi: \Tr (\Gamma) \to F$ which sends $x$ to $h_x$ for all $x \in V$.
This homomorphism is surjective because $\{h_x \mid x \in V\}$ generates $F$ (see Lemma \ref{lem3_11}).
So, it remains to prove that $\Phi$ is injective.

Let $g \in \Tr (\Gamma) \setminus \{1\}$.
Consider the rewriting system $\RR = \RR (\Gamma)$ of Theorem \ref{thm2_4}.
Since $\Gamma$ is a complete graph, the $\RR$-irreducible pilings are all of length $1$ or $0$.
Moreover, by Theorems \ref{thm2_3} and \ref{thm2_4}, there exists a unique $\RR$-irreducible piling which represents $g$.
This piling is different from the empty one, because $g \neq 1$, hence it is of the form $(U)$, where $U=\{x_1^{a_1}, \dots, x_p^{a_p}\}$ and $x_1 > x_2 > \cdots >x_p$, where $p \ge 1$.
Then $g = \omega (U) = x_1^{a_1} x_2^{a_2} \dots x_p^{a_p}$, hence
\[
\Phi(g) = h_{x_1}^{a_1} \circ h_{x_2}^{a_2} \circ \cdots \circ h_{x_p}^{a_p}\,.
\]
The homeomorphism $h_{x_1}$ is of infinite order, hence, if $p=1$, then $\Phi (g) = h_{x_1}^{a_1} \neq \id$.
Suppose $p \ge 2$.
Choose $t \in (x_2,x_1)$.
Since $h_{x_i} (t) = t$ for all $2 \le i \le p$ and $h_{x_1}$ is strictly decreasing on $(-\infty, x_1)$, 
\[
\Phi(g)(t) = h_{x_1}^{a_1}(t) \neq t\,,
\]
hence $\Phi(g) \neq \id$.
This shows that $\Ker (\Phi) = \{1\}$, hence $\Phi$ is injective.
\end{proof}

To finish this subsection we show that some natural parabolic subgroups of $\Tr (\Gamma)$ are isomorphic to $F$ itself.
More precisely, let $x\in \Z [\frac{1}{2}]$, let $V_x = \{x\}_\downarrow = \{ y \in V \mid y \le x\}$, and let $\Gamma_x$ be the full subgraph of $\Gamma$ spanned by $V_x$.
As observed in Subsection \ref{subsec2_4}, $\Gamma_x$ is a parabolic subgraph of $\Gamma$.

\begin{prop}\label{prop3_16}
Let $x \in \Z [\frac{1}{2}]$.
There exists a bijection $\beta: V_x \to V$ which induces an isomorphism $\beta: \Tr(\Gamma_x) \to \Tr (\Gamma)$.
In particular, $\Tr (\Gamma_x)$ is isomorphic to $F$.
\end{prop}

\begin{proof}
We consider the map $\alpha$ defined in the proof of Lemma \ref{lem3_13}.
Let $p$ be the least integer $\ge 0$ such that $x \in V_p$.
Set $x' = t_p (x)$.
For all $k \ge 0$ we set $v_k = s_{p+1}^k (x')$ and $v_{-k} = t_p^k (x')$.
Then $\alpha: (-\infty, x) \to \R$ denotes the orientation preserving piecewise linear homeomorphism which sends linearly $[v_k, v_{k+1}]$ onto $[k, k+1]$ for all $k \in \Z$ (see Figure \ref{fig3_5}).
We have $\alpha( (-\infty, x) \cap \Z [\frac{1}{2}]) = \alpha (V_x \setminus \{x\}) = \Z [\frac{1} {2}]$.
So, there is a bijection $\beta: V_x \to V$ which sends $y$ to $\alpha (y)$ for all $y \in V_x \setminus \{x\}$ and sends $x$ to $\infty$.
It is easily checked that, for all $y,y' \in V_x$, we have $y \le y'$ if and only if $\beta(y) \le \beta (y')$, and that, for all $y \in V_x$, $h_{\beta (y)} = \alpha \circ h_y|_{(-\infty,x)} \circ \alpha^{-1}$.
Thus, $\beta$ induces an isomorphism between $\Gamma_x = (\Gamma_x, \le_x, \mu_x, (\varphi_{x,y})_{y \in V_x})$ and $\Gamma = (\Gamma, \le, \mu, (\varphi_y)_{y \in V})$, where $\le_x$ and $\mu_x$ are the restrictions of $\le$ and $\mu$ to $V_x$, respectively, and, for all $y \in V_x$, $\varphi_{x,y}$ is the restriction of $\varphi_y$ to $V_x = \starE_y (\Gamma_x)$.
This $\beta$ obviously defines an isomorphism $\beta : \Tr (\Gamma_x) \to \Tr (\Gamma)$.
\end{proof}

\begin{rem}
By identifying $\Tr (\Gamma)$ with $F$ it is easily seen that $\Tr (\Gamma_x)$ is the subgroup of $F$ consisting of the elements whose support is contained in $(-\infty, x)$.
It is well-known and easy to prove that this subgroup is a copy $F$.
\end{rem}


\subsection{Ordered quandle groups}\label{subsec3_4}

A \emph{quandle} is a non-empty set $Q$ endowed with an operation $*$ which satisfies the following properties.
\begin{itemize}
\item[(1)]
For all $x\in Q$, $x * x = x$.
\item[(2)]
For all $x, y, z \in Q$, $(z * x) * (y * x) = (z * y) * x$.
\item[(3)]
For all $x \in Q$, the map $Q \to Q$, $y \mapsto y * x$, is a bijection.
\end{itemize}
A quandle $Q$ is said to be \emph{ordered} (on the right) if there exists a total order $\le$ on $Q$ such that, for all $x, y, z \in Q$, we have $y \le z$ if and only if $y * x \le z * x$.

\begin{expl1}
A group $G$ is called \emph{bi-orderable} if there exists a total order $\le$ on $G$ which is invariant under left and right multiplication.
Examples of bi-orderable groups are free abelian groups, free groups and, more generally, right-angled Artin groups.
Other important examples are the pure braid groups.
Let $G$ be a bi-orderable group, let $\le$ be a total order on $G$ invariant under left and right multiplication, and let $*$ be the operation on $G$ defined by $y * x = xyx^{-1}$.
Then $(G, *, \le)$ is an ordered quandle.
\end{expl1}

\begin{expl2}
Consider the field $\R$ of real numbers endowed with its natural total order.
Then, for a fixed $\alpha>0$, the binary operation $x*y = \alpha x + (1-\alpha) y$ endows $\R$ with an ordered quandle structure (see \cite{BaElh1}).
\end{expl2}

Let $(Q,*,\le)$ be an ordered quandle.
Let $\Gamma$ be the full graph whose set of vertices is $Q$.
For $x \in Q$ we define $\varphi_x: Q \to Q$ by $\varphi_x (y) = y$ if $x \le y$ and $\varphi_x (y) = y * x$ if $y \le x$.
We set $\mu (x) = \infty$ for all $x \in Q$.

\begin{prop}\label{prop3_17}
The above defined quadruple $\Gamma = (\Gamma, \le, \mu, (\varphi_x)_{x\in Q})$ is a trickle graph.
\end{prop}

\begin{proof}
Since $\Gamma$ is a full graph and $y\mapsto y * x$ is an order-preserving bijection, the map $\varphi_x$ is an order-preserving automorphism of $\Gamma$ for all $x \in Q$.
Since $\Gamma$ is a full graph and $\le$ is a total order, Conditions (a) and (b) are trivially satisfied.
Conditions (e) and (f) hold because $\mu (x) = \infty$ for all $x \in Q$.
Condition (c) follows from the fact that $Q$ is ordered, and Condition (d) follows from the definition of $\varphi_x$.
Finally, Condition (g) is satisfied because $Q$ is a quandle.
\end{proof}


\section{The trickle algorithm}\label{sec4}

In this section we fix a trickle graph $\Gamma = (\Gamma, \le, \mu, (\varphi_x)_{x\in V(\Gamma)})$ and we keep the notations of Subsection \ref{subsec2_2}.
Our aim is to prove Theorem \ref{thm2_4}.

Let $A$ be an alphabet and let $R$ be a rewriting system on $A^*$.
A \emph{critical pair} for $R$ is a quintuple $(u_1, u_2, u_3, v_1, v_2)$ of elements of $A^*$ satisfying one of the following two conditions.
\begin{itemize}
\item[(a)]
$(u_1 \cdot u_2, v_1) \in R$, $(u_2 \cdot u_3, v_2) \in R$ and $u_2 \neq \epsilon$;
\item[(b)]
$(u_1 \cdot u_2 \cdot u_3, v_1) \in R$ and $(u_2, v_2)\in R$.
\end{itemize}
We say that the critical pair $(u_1, u_2, u_3, v_1, v_2)$ is \emph{resolved} if there exists $w \in A^*$ such that
\begin{itemize}
\item
$v_1 \cdot u_3 \to^* w$ and $u_1 \cdot v_2 \to^* w$ in case (a),
\item
$v_1 \to^* w$ and $u_1 \cdot v_2 \cdot u_3 \to^* w$ in case (b).
\end{itemize}
We will say that the critical pair is of \emph{type (a)} or of \emph{type (b)} depending on which of the above conditions (a) or (b) it satisfies.

The following will be used to show that our rewriting system is confluent.

\begin{thm}[Newman \cite{Newma1}]\label{thm4_1}
Let $A$ be an alphabet and let $R$ be a terminating rewriting system on $A^*$.
If all critical pairs of $R$ are resolved, then $R$ is confluent.
\end{thm}

Now, we move on to the proof of Theorem \ref{thm2_4}.
We start by showing that the extraction operation is well-defined.

\begin{lem}\label{lem4_2}
Let $U = \{x_1^{a_1}, x_2^{a_2}, \dots, x_p^{a_p}\}$ be a non-empty stratum and let $x_i^{a_i} \in U$.
Set $U (x_i^{a_i}) = \{ x_j^{a_j} \mid x_j \ge x_i \}$.
Then $\supp (U (x_i^{a_i}))$ is totally ordered and $\gamma (U,x_i^{a_i}) = \gamma (U (x_i^{a_i}), x_i^{a_i})$.
In particular, the definition of $\gamma (U, x_i^{a_i})$ does not depend on the choice of the numbering of the elements of $U$.
\end{lem}

\begin{proof}
As always we assume that, if $x_j > x_k$, then $j <k$.
Recall that 
\[
\gamma (U, x_i^{a_i}) = \big( (\varphi_{x_1}^{a_1} \circ \varphi_{x_2}^{a_2} \circ \cdots \circ \varphi_{x_{i-1 }}^{a_{i-1}}) (x_i) \big)^{a_i}\,.
\]
Let $x_j^{a_j}, x_k^{a_k} \in U (x_i^{a_i})$.
Since $x_i \le x_j$, $x_i \le x_k$, and $\{x_i, x_k\} \in E (\Gamma)$, Condition (b) in the definition of a trickle graph implies that either $x_j \le x_k$, or $x_k \le x_j$.
This shows that $\supp (U (x_i^{a_i}))$ is totally ordered.
Then, since $\supp (U (x_i^{a_i}))$ is totally ordered, the definition of the extraction operation in $U (x_i^{a_i})$ is unique, that is, $\gamma (U (x_i^{a_i}), x_i^{a_i})$ is well-defined.

It remains to show that $\gamma (U, x_i^{a_i}) = \gamma (U (x_i^{a_i}), x_i^{a_i})$.
Notice that this implies that the definition of $\gamma (U, x_i^{a_i})$ does not depend on the choice of the numbering of the elements of $U$.
It suffices to show that, if $x_j^{a_j} \not \in U(x_i^{a_i})$, then $\gamma(U, x_i^{a_i}) =\gamma (L(U, x_j^{a_j}), x_i^{a_i})$.
If $j > i$, then there is nothing to prove.
Suppose $j < i$.
Since $x_i \not \le x_j$, according to Condition (c) in the definition of a trickle graph, 
\begin{gather*}
\varphi_{x_{j+1}}^{a_{j+1}} \circ \cdots \circ \varphi_{x_{i-1}}^{a_{i-1}} (x_i) \not \le \varphi_{x_{j+1}}^{a_{j+1}} \circ \cdots \circ \varphi_{x_{i-1}}^{a_{i-1}} (x_j) = x_j\ \ \Rightarrow\\
\varphi_{x_{j}}^{a_j} \circ \varphi_{x_{j+1}}^{a_{j+1}} \circ \cdots \circ \varphi_{x_{i-1}}^{a_{i-1}} (x_i) = \varphi_{x_{j+1}}^{a_{j+1}} \circ \cdots \circ \varphi_{x_{i-1}}^{a_{i-1}} (x_i)\ \ \Rightarrow\\
\varphi_{x_1}^{a_1} \circ \varphi_{x_2}^{a_2} \circ \cdots \circ \varphi_{x_{i-1}}^{a_{i-1}} (x_i) = \varphi_{x_{1}}^{a_{1}} \circ \cdots \circ \varphi_{x_{j-1}}^{a_{j-1}} \circ \varphi_{x_{j+1}}^{a_{j+1}} \circ \cdots \circ \varphi_{x_{i-1}}^{a_{i-1}} (x_i)\ \ \Rightarrow\\
\gamma (U, x_i^{a_i})  = \gamma (L (U,  x_j^{a_j}), x_i^{a_i})\,.
\end{gather*}
\end{proof}

\begin{corl}\label{corl4_3}
Let $U = \{x_1^{a_1}, x_2^{a_2}, \dots, x_p^{a_p}\}$ be a non-empty stratum and let $x_i^{a_i} \in U$.
Suppose that, if $x_j > x_k$, then $j <k$.
\begin{itemize}
\item[(1)]
Let $j \in \{1, \dots, p\}$ be such that $x_i \not \le x_j$.
Then $\gamma (L(U, x_j^{a_j}), x_i^{a_i}) = \gamma (U, x_i^{a_i})$.
\item[(2)]
Let $j \in \{i-1, i, i+1, \dots, p\}$.
Then $\gamma (U, x_i^{a_i}) = \left( \varphi_{x_1}^{a_1} \circ \varphi_{x_2}^{a_2} \circ \cdots \circ \varphi_{x_j}^{a_j}(x_i) \right)^{a_i}$.
\end{itemize}
\end{corl}

\begin{proof}
Part (1) is a straightforward consequence of Lemma \ref{lem4_2}.
We prove Part (2).
For $j \ge i$, we have $x_i \not \le x_j$, hence $\varphi_{x_j} (x_i) = x_i$.
So, if $j \ge i$, then
\[
\left( \varphi_{x_1}^{a_1} \circ \varphi_{x_2}^{a_2} \circ \cdots \circ \varphi_{x_j}^{a_j}(x_i) \right)^{a_i} = \left( \varphi_{x_1}^{a_1} \circ \varphi_{x_2}^{a_2} \circ \cdots \circ \varphi_{x_{i-1}}^{a_{i-1}}(x_i) \right)^{a_i} = \gamma (U, x_i^{a_i})\,.
\proved
\]
\end{proof}

\begin{lem}\label{lem4_4}
The rewriting system $\RR$ is terminating.
\end{lem}

\begin{proof}
Define the weight of a piling $u = (U_1, \dots, U_r) \in \Omega^*$ as $r + \sum_{i = 1}^r i \cdot \lg_\st (U_i)$.
It is easily seen that each elementary rewriting step strictly reduces the weight.
So, there is no infinite rewriting sequence, and therefore $\RR$ is terminating.
\end{proof}

\begin{prop}\label{prop4_5}
The rewriting system $\RR$ is confluent.
\end{prop}

We know from Lemma \ref{lem4_4} that $\RR$ is terminating, hence, by Theorem \ref{thm4_1}, in order to prove Proposition \ref{prop4_5} it suffices to show that every critical pair is resolved.
Let $(u_1, u_2, u_3, v_1, v_2)$ be a critical pair.
Then we have one of the following three situations.
\begin{itemize}
\item[(C1)]
$u_1 = \epsilon$, $u_2 = (\emptyset)$, $u_3 = (W)$, $v_1 = T (\emptyset, W, z^c)$ and $v_2 = \epsilon$, where $W \in \Omega$ and $z^c \in W$.
This is a critical pair of type (b).
\item[(C2)]
$u_1 = (U)$, $u_2 = (V)$, $u_3 = (W)$, $v_1 = T (U, V, y^b)$ and $v_2 = T (V, W, z^c)$, where $U, V, W \in \Omega$, $y^b \in V$ and $z^c \in W$.
This is a critical pair of type (a).
\item[(C3)]
$u_1 = u_3 = \epsilon$, $u_2 = (U, V)$, $v_1 = T (U, V, y^b)$ and $v_2 = T (U, V, y'^{b'})$, where $U, V \in \Omega$, $y^b, y'^{b'} \in V$ and $y^b \neq y'^{b'}$.
This is a critical pair of type (b).
\end{itemize}
We show in each case that the critical pair is resolved.
This requires a case-by-case proof through several lemmas.

\begin{lem}\label{lem4_6}
Let $x, y, z \in V (\Gamma)$ be such that $z \le y \le x$.
Then:
\begin{itemize}
\item[(1)]
$(\varphi_x \circ \varphi_y) (z) = (\varphi_{\varphi_x (y)} \circ \varphi_x) (z)$,
\item[(2)]
$(\varphi_x^{-1} \circ \varphi_y) (z) = (\varphi_{\varphi_x^{-1} (y)}\circ \varphi_x^{-1}) (z)$,
\item[(3)]
$(\varphi_x \circ \varphi_y^{-1}) (z) = (\varphi_{\varphi_x(y)}^{-1} \circ \varphi_x) (z)$,
\item[(4)]
$(\varphi_x^{-1} \circ \varphi_y^{-1}) (z) = (\varphi_{\varphi_x^{-1}(y)}^{-1} \circ \varphi_x^{-1}) (z)$.
\end{itemize}
\end{lem}

\begin{proof}
Part (1) is Condition (g) in the definition of a trickle graph.
For Part (2) we apply Part (1) to $z_1 = \varphi_x^{-1} (z) \le y_1 = \varphi_x^{-1} (y) \le x$.
For Part (3) we apply Part (1) to $z_1 = \varphi_y^{-1} (z) \le y \le x$.
Finally, for Part (4) we apply Part (2) to $z_1 = \varphi_y^{-1} (z) \le y \le x$.
\end{proof}

\begin{lem}\label{lem4_7}
Let $x, y, z \in V (\Gamma)$ be such that $z \le y \le x$ and let $a, b \in \Z$.
Then
\[
(\varphi_x^a \circ \varphi_y^b) (z) = (\varphi_{\varphi_x^a (y)}^b \circ \varphi_x^a) (z)\,.
\]
\end{lem}

\begin{proof}
If either $a=0$ or $b=0$ then the result is obvious.
So, we can assume that $a \neq 0$ and $b \neq 0$.
From here the proof is divided into four cases, depending on whether $a > 0$ and $b > 0$, or $a < 0$ and $b > 0$, or $a > 0$ and $b < 0$, or $a < 0$ and $b < 0$.
We treat the case $a > 0$ and $b >0$.
The other three cases can be treated in the same way.

So, we take two integers $a > 0$ and $b >0$.
First, we prove by induction on $a$ that
\[
(\varphi_x^a \circ \varphi_y) (z) = (\varphi_{\varphi_x^a (y)} \circ \varphi_x^a) (z)\,.
\]
The case $a = 1$ is Lemma \ref{lem4_6}\,(1).
Assume $a \ge 2$ and that the induction hypothesis holds. 
Let $y_1 = \varphi_x (y)$ and $z_1 = \varphi_x (z)$.
By applying the induction hypothesis to $z_1 \le y_1 \le x$ and Lemma \ref{lem4_6}\,(1) we get:
\begin{gather*}
(\varphi_{\varphi_x^a (y)} \circ \varphi_x^a) (z) = (\varphi_{\varphi_x^{a-1} (y_1)} \circ \varphi_x^{a-1}) (z_1) = (\varphi_x^{a-1} \circ \varphi_{y_1}) (z_1) = (\varphi_x^{a-1} \circ \varphi_{\varphi_x (y)} \circ \varphi_x) (z) =\\
(\varphi_x^{a-1} \circ \varphi_x \circ \varphi_y) (z) = (\varphi_x^a \circ \varphi_y) (z)\,.
\end{gather*}

Now we fix $a \ge 1$ and we prove by induction on $b$ that
\[
(\varphi_x^a \circ \varphi_y^b) (z) = (\varphi_{\varphi_x^a (y)}^b \circ \varphi_x^a) (z)\,.
\]
The case $b = 1$ is proved in the previous paragraph.
Assume $b \ge 2$ and that the induction hypothesis holds.
Let $z_1 = \varphi_y^{b-1} (z)$.
By applying the induction hypothesis and the case $b = 1$ to $z_1 \le y \le x$ we get:
\begin{gather*}
(\varphi_{\varphi_x^a (y)}^b \circ \varphi_x^a) (z) = (\varphi_{\varphi_x^a (y)} \circ \varphi_{\varphi_x^a (y)}^{b-1} \circ \varphi_x^a) (z) = (\varphi_{\varphi_x^a (y)} \circ \varphi_x^a \circ \varphi_y^{b-1}) (z) =\\
(\varphi_{\varphi_x^a (y)} \circ \varphi_x^a) (z_1) = (\varphi_x^a \circ \varphi_y) (z_1) = (\varphi_x^a \circ \varphi_y^b) (z)\,.
\end{gather*}
\end{proof}

\begin{lem}\label{lem4_8}
In Case (C1) every critical pair $(u_1, u_2, u_3, v_1, v_2)$ is resolved.
\end{lem}

\begin{proof}
We have $u_1 = \epsilon$, $u_2 = (\emptyset)$, $u_3 = (W)$, $v_1 = T (\emptyset, W, z^c)$ and $v_2 = \epsilon$, where $W \in \Omega$ and $z^c \in W$.
We must show that there exists $w \in \Omega^*$ such that $v_1 \to^* w$ and $u_1 \cdot v_2 \cdot u_3 \to^* w$.
In fact, we set $w = u_1 \cdot v_2 \cdot u_3 = u_3 = (W)$ and we show that $v_1 \to^* w$.

We set $W = \{z_1^{c_1}, \dots, z_p^{c_p}\}$ so that, if $z_j > z_k$, then $j < k$.
Let $i \in \{1, \dots, p\}$ be such that $z^c = z_i^{c_i}$.
Let $y = (\varphi_{z_1}^{c_1} \circ \cdots \circ \varphi_{z_{i-1}}^{c_{i-1}}) (z_i)$.
We have $\gamma (W, z^c) = y^c$, and $v_1 = T (\emptyset, W, z^c) = (\{y^c\}, L(W, z^c))$.
We define $(V_j, W_j) \in \Omega^*$ for $j \in \{0,1, \dots, p-1\}$ as follows.
We set $V_0 = \{y^c\}$ and $W_0 = L(W, z^c)$.
For $1 \le j< i$ we set
\[
V_j = \{z_1^{c_1}, \dots, z_j^{c_j}, u_j^c \} \text{ and } W_j = \{z_{j+1}^{c_{j+1}}, \dots, z_{i-1}^{c_{i-1}}, z_{i+1}^{c_{i+1}}, \dots, z_p^{c_p}\}\,,
\]
where 
\[
u_j = (\varphi_{z_{j+1}}^{c_{j+1}} \circ \cdots \circ \varphi_{z_{i-1}}^{c_{i-1}}) (z_i)\,.
\]
For $i \le j \le p-1$ we set
\[
V_j = \{z_1^{c_1}, \dots, z_{j+1}^{c_{j+1}}\} \text{ and } W_j = \{z_{j+2}^{c_{j+2}}, \dots, z_p^{c_p}\}\,.
\]
We show that $(V_j, W_j) \to (V_{j+1}, W_{j+1})$ for all $j \in \{0, \dots, p-2\}$.
Since $(V_0, W_0) = v_1$ and $(V_{p-1}, W_{p-1}) = (W, \emptyset) \to (W) =w$, this implies that $v_1 \to^*w$.

For $1 \le j < i$ we have $\gamma (W_{j-1}, z_j^{c_j}) = z_j^{c_j}$, and for $i \le j \le p-2$ we have $\gamma (W_{j-1}, z_{j+1}^{c_{j+1}}) = z_{j+1}^{c_{j+1}}$.
For $1 \le j < i$ we have $L (W_{j-1} ,z_j^{c_j}) = W_j$, and for $i \le j \le p-2$ we have $L (W_{j-1}, z_{j+1}^{c_{j+1}}) = W_j$.
For $1 \le k < j \le p$ we have $\{z_k, z_j\} \in E(\Gamma)$, because $W$ is a stratum, and for $j < i$ we have $\{u_{j-1}, z_j\} \in E(\Gamma)$, because, by the choice of the numbering of the indices,
\[
u_{j-1} = (\varphi_{z_j}^{c_j} \circ \cdots \circ \varphi_{z_{i-1}}^{c_{i-1}}) (z_i) \text{ and }
z_j = (\varphi_{z_j}^{c_j} \circ \cdots \circ \varphi_{z_{i-1}}^{c_{i-1}}) (z_j)\,.
\]
So, for $1 \le j \le i-1$ the syllable $z_j^{c_j}$ can be added to $V_{j-1}$ and $R (V_{j-1}, z_j^{c_j}) = V_j$.
Similarly, for $i \le j \le p-2$ the syllable $z_{j+1}^{c_{j+1}}$ can be added to $V_{j-1}$ and $R (V_{j-1}, z_{j+1}^{c_{j+1}}) = V_j$.
We deduce that, for every $j \in \{1, \dots, i-1\}$, we have $(V_j, W_j) = T(V_{j-1}, W_{j-1}, z_j^{c_j})$, and, for every $j \in \{i, \dots, p-2\}$, we have $(V_j, W_j) = T (V_{j-1}, W_{j-1}, z_{j+1}^{c_{j+1}})$.
So, $(V_j, W_j) \to (V_{j+1}, W_{j+1})$ for all $j \in \{0, \dots, p-2\}$.
\end{proof}

\begin{lem}\label{lem4_9}
Let $U$ be a non-empty stratum and let $x^a \in U$.
Let $y^b \in S (\Gamma)$ be a syllable such that $x \neq y$ and $y^b$ can be added to $U$.
Then
\[
\gamma (R (U, y^b), \varphi_y^{-b} (x)^a) = \gamma (U, x^a)\,.
\]
\end{lem}

\begin{proof}
We set $U = \{ x_1^{a_1}, \dots, x_p^{a_p}\}$ and we assume as always that, if $x_j > x_k$, then $j<k$.
Let $i \in \{1, \dots, p\}$ be such that $x^a =x_i^{a_i}$.
By Lemma \ref{lem4_2} we can assume that $U (x^a) = \{x_1^{a_1}, \dots, x_{i-1}^{a_{i-1}}, x_i^{a_i} \}$ and $x_1 > \cdots > x_{i-1} > x_i$.
From here the proof of the lemma is divided into two cases depending on whether $x \not < y$ or $x < y$.

{\it Case 1: $x \not < y$.}
Let $j \in \{1, \dots, p\}$.
Note that 
\[
\varphi_y^{-b} (x_j) \ge \varphi_y^{-b} (x_i)\ \Leftrightarrow\ x_j \ge x_i \Leftrightarrow\ 1\le j \le i\,.
\]
Suppose $1\le j \le i$.
If $x > y$, then $x_j > y$, hence $\varphi_y^{-b} (x_j) = x_j$.
Suppose $y || x$.
If we had $y > x_j$, then by Condition (b) in the definition of a trickle graph we would have $x_i || x_j$, which would contradict the hypothesis $x_j \ge x_i$.
Thus, $y || x_j$ or $y< x_j$, hence $\varphi_y^{-b} (x_j) = x_j$.
So, $\varphi_y^{-b} (x_j) = x_j$ for all $j \in \{1, \dots, i-1,i\}$, hence $R (U, y^b) (x^a) = U (x^a)$.
It follows that
\[
\gamma (R (U, y^b), \varphi_y^{-b} (x)^a) = \gamma (U (x^a), x^a) = \gamma (U, x^a) \,.
\]

{\it Case 2: $x < y$.}
By Lemma \ref{lem4_2} there exists $k \in \{1, \dots, i-1 \}$ such that $x_{k-1} > y > x_k$ if $y \not \in \supp (U)$ (which simply means $y > x_1$ if $k=1$) and $y = x_k$ if $y \in \supp (U)$.
For each $j \in \{k, \dots, i-1, i \}$ we set $z_j = \varphi_y^{-b} (x_j)$.
In particular, $\varphi_y^{-b} (x) = z_i$.
We have
\begin{gather*}
R(U, y^b) (z_i^a) = \{ x_1^{a_1}, \dots, x_{k-1}^{a_{k-1}}, y^b, z_k^{a_k}, z_{k+1}^{a_{k+1}}, \dots, z_{i-1}^{a_{i-1}}, z_i^{a_i} \} \text{ if } y \not \in \supp (U)\,,\\
R(U, y^b) (z_i^a) = \{ x_1^{a_1}, \dots, x_{k-1}^{a_{k-1}}, x_k^{b+a_k}, z_{k+1}^{a_{k+1}}, \dots, z_{i-1}^{a_{i-1}}, z_i^{a_i} \} \text{ if } y = x_k \in \supp (U) \text{ and } b + a_k \neq 0\,,\\
R(U, y^b) (z_i^a) = \{ x_1^{a_1}, \dots, x_{k-1}^{a_{k-1}}, z_{k+1}^{a_{k+1}}, \dots, z_{i-1}^{a_{i-1}}, z_i^{a_i} \} \text{ if } y= x_k \in \supp (U) \text{ and } b + a_k = 0\,.
\end{gather*}
In the three cases we have 
\[
\gamma (R (U, y^b), \varphi_y^{-b} (x)^a) = \gamma (R (U, y^b) (z_i^{a_i}), z_i^{a_i}) = (\varphi_{x_1}^{a_1} \circ \cdots \circ \varphi_{x_{k-1}}^{a_{k-1}} \circ \varphi_y^b \circ \varphi_{z_k}^{a_k} \circ \cdots \circ \varphi_{z_{i-1}}^{a_{i-1}}) (z_i)^a\,.
\]
By applying Lemma \ref{lem4_7} several times we get
\begin{gather*}
(\varphi_{z_k}^{a_k} \circ \cdots \circ \varphi_{z_{i-1}}^{a_{i-1}}) (z_i) =\\
(\varphi_{z_k}^{a_k} \circ \cdots \circ \varphi_{z_{i-2}}^{a_{i-2}} \circ \varphi_{\varphi_y^{-b} (x_{i-1})}^{a_{i-1}} \circ \varphi_y^{-b}) (x_i) =\\
(\varphi_{z_k}^{a_k} \circ \cdots \circ \varphi_{\varphi_y^{-b} (x_{i-2})}^{a_{i-2}} \circ \varphi_y^{-b} \circ \varphi_{x_{i-1}}^{a_{i-1}}) (x_i) =
\cdots =\\
(\varphi_y^{-b} \circ \varphi_{x_k}^{a_k} \circ \cdots \circ \varphi_{x_{i-1}}^{a_{i-1}}) (x_i)\,,
\end{gather*}
hence
\[
\gamma (R (U, y^b), \varphi_y^{-b} (x)^a) = (\varphi_{x_1}^{a_1} \circ \cdots \circ \varphi_{x_{k-1}}^{a_{k-1}} \circ \varphi_{x_k}^{a_k} \circ \cdots \circ \varphi_{x_{i-1}}^{a_{i-1}}) (x_i)^a =
\gamma (U,x^a)\,.
\proved
\]
\end{proof}

\begin{lem}\label{lem4_10}
In Case (C2) every critical pair $(u_1,u_2,u_3,v_1,v_2)$ is resolved.
\end{lem}

\begin{proof}
There exist $U, V, W \in \Omega$, $y^b \in V$ and $z^c \in W$ such that $u_1 = (U)$, $u_2 = (V)$, $ u_3 = (W)$, $v_1 = T(U, V, y^b)$ and $v_2 = T(V, W, z^c)$.
Let $y'^b = \gamma (V, y^b)$.
Then $y'^b$ can be added to $U$ and $T(U,V,y^b) = (U_1, V_1)$, where $U_1 = R (U, y'^b)$ and $V_1 = L (V, y^b)$.
Let $z'^c = \gamma (W, z^c)$.
Then $z'^c$ can be added to $V$ and $T(V, W, z^c) = (V_2, W_2)$, where $V_2 = R(V, z'^c)$ and $W_2 = L(W, z^c)$.
We must show that there exists $w \in \Omega^*$ such that $(U_1, V_1, W) \to^* w$ and $(U, V_2, W_2) \to^* w$.
We set $V=\{ y_1^{b_1}, \dots, y_p^{b_p} \}$ and we assume as always that, if $y_j > y_k$, then $j < k$.
Let $i \in \{1, \dots, p\}$ be such that $y^b = y_i^{b_i}$.
From here the proof is divided into three cases depending on whether $z' \neq y$, or $z'=y$ and $b+c \neq 0$, or $z'=y$ and $b +c=0$.

{\it Case 1: $z' \neq y$.}
We set
\begin{gather*}
V_3 = \{ \varphi_{z'}^{-c} (y_k)^{b_k} \mid 1 \le k \le p \text{ and } k \neq i \} \cup \{ z'^c \} \text{ if } z' \not \in \supp (V)\,,\\
V_3 = \{ \varphi_{z'}^{-c} (y_k)^{b_k} \mid 1 \le k \le p \text{ and } k \not\in \{i, j\} \} \cup \{ z'^{b_j +c} \} \text{ if } z' =y_j \in \supp (V) \text{ and } b_j + c \neq 0\,,\\
V_3 = \{ \varphi_{z'}^{-c} (y_k)^{b_k} \mid 1 \le k \le p \text{ and } k \not\in \{i, j\} \} \text{ if } z' =y_j \in \supp (V) \text{ and } b_j + c = 0\,,
\end{gather*}
and we set $w= (U_1, V_3, W_2)$.
Observe that $z'^c$ can be added to $V_1$ and $R(V_1,z'^c) = V_3$, hence $T(V_1,W,z^c) = (V_3, W_2)$, and therefore $(U_1, V_1, W) \to (U_1, V_3, W_2) = w$.
By Lemma \ref{lem4_9} we have $\gamma (V_2, \varphi_{z'}^{-c} (y)^b) = \gamma (V, y^b) = y'^b$.
We also know that $y'^b$ can be added to $U$ and $R(U,y'^b) = U_1$.
Moreover, $L (V_2, \varphi_{z'}^{-c} (y)^b) = V_3$, hence $T(U, V_2, \varphi_{z'}^{-c} (y)^b) = (U_1, V_3)$, and therefore $(U, V_2, W_2) \to (U_1, V_3, W_2) = w$.

{\it Case 2: $z' = y$ and $b+c \neq 0$.}
We set
\[
V_4 = \{ y_1^{b_1}, \dots, y_{i-1}^{b_{i-1}}, \varphi_y^{-c} (y_{i+1})^{b_{i+1}}, \dots, \varphi_y^{-c} (y_p)^{b_p} \}\,,
\]
$U_4= R (U, y'^{b+c})$, and $w = (U_4, V_4, W_2)$.
The syllable $z'^c = y^c$ can be added to $V_1$ and $R(V_1, z'^c) = V_3$, where
\[
V_3 = \{ y_1^{b_1}, \dots, y_{i-1}^{b_{i-1}}, y^c, \varphi_y^{-c} (y_{i+1})^{ b_{i+1}}, \dots, \varphi_y^{-c} (y_p)^{b_p} \}\,,
\]
hence $T(V_1,W, z^c) = (V_3,W_2)$, and therefore $(U_1, V_1, W) \to (U_1, V_3, W_2)$.
We have $\gamma (V_3, y^c) = y'^c$, which can be added to $U_1$, $R(U_1, y'^c) = R (U, y'^{b+c}) = U_4$, and $L(V_3, y^c) = V_4$, hence $T (U_1, V_3, y^c) = (U_4, V_4)$, and therefore $(U_1, V_3, W_2) \to (U_4, V_4, W_2) = w$.
So, $(U_1, V_1, W) \to^* w$.
Note that
\[
V_2 = \{ y_1^{b_1}, \dots, y_{i-1}^{b_{i-1}}, y^{b+c}, \varphi_y^{-c} (y_{i+1})^{b_{i+1}}, \dots, \varphi_y^{-c} (y_p)^{b_p} \}\,.
\]
We have $\gamma (V_2, y^{b+c}) = y'^{b+c}$, which can be added to $U$.
Moreover, $R (U, y'^{b+c}) = U_4$ and $L (V_2, y^{b+c}) = V_4$, hence $T (U, V_2, y^{b+c}) = (U_4, V_4)$, and therefore $(U,V_2, W_2) \to (U_4, V_4, W_2)=w$.

{\it Case 3: $z'=y$ and $b+c=0$.}
We set $w = (U, V_2, W_2)$ and we show that $(U_1, V_1, W) \to^* w$.
The syllable $z'^c = y^c$ can be added to $V_1$ and $R(V_1,z'^c) = V_3$, where
\[
V_3 = \{ y_1^{b_1}, \dots, y_{i-1}^{b_{i-1}}, y^c, \varphi_y^{-c} (y_{i+1})^{b_{i+1}}, \dots, \varphi_y^{-c} (y_p)^{b_p} \}\,,
\]
hence $T (V_1, W, z^c) = (V_3, W_2)$, and therefore $(U_1, V_1, W) \to (U_1, V_3, W_2)$.
We have $\gamma (V_3, y^c) = y'^c$, which can be added to $U_1$.
We have $R(U_1, y'^c) = U$, because $c=-b$, and we have $L(V_1,y^c) = V_2$, since
\[
V_2 = \{ y_1^{b_1}, \dots, y_{i-1}^{b_{i-1}}, \varphi_y^{-c} (y_{i+1})^{b_{i+1}}, \dots, \varphi_y^{-c} (y_p)^{b_p} \}\,,
\]
hence $T (U_1, V_3, y^c) = (U, V_2)$, and therefore $(U_1, V_3, W_2) \to (U, V_2, W_2) = w$.
So, $(U_1, V_1, W) \to^* w$.
\end{proof}

\begin{lem}\label{lem4_11}
Let $U$ be a non-empty stratum and let $x^a, y^b \in U$ be such that $x \neq y$.
Let $x'^a = \gamma (U, x^a)$, $y'^b = \gamma (U, y^b)$, $x''^a = \gamma (L (U, y ^b), x^a)$ and $y''^b = \gamma (L(U, x^a), y^b)$.
\begin{itemize}
\item[(1)]
We have $\{x',y'\}, \{x'',y'\} \in E(\Gamma)$.
\item[(2)]
We have $x || y$ if and only if $x' || y'$, and we have $x \le y$ if and only if $x' \le y'$.
\item[(3)]
We have $\varphi_{x'}^{-a} (y') = \varphi_{x''}^{-a} (y') = y''$.
\item[(4)]
Let $V$ be another stratum.
If both $x'^a$ and $y'^b$ can be added to $V$, then $y''^b$ can be added to $R(V,x'^a)$.
\end{itemize}
\end{lem}

\begin{proof}
We set $U = \{ x_1^{a_1}, \dots, x_p^{a_p} \}$ and we suppose that, if $x_k > x_\ell$, then $k < \ell$.
Let $i,j \in \{1, \dots, p\}$ be such that $x=x_i$ and $y = x_j$.
By Lemma \ref{lem4_2} we can assume that $U (y) = \{x_1^{a_1}, \dots, x_{j-1}^{a_{j-1}}, x_j^{a_j} \}$ and $x_1 > \cdots > x_{j-1} > x_j= y$.
From here the proof is divided into two cases depending on whether $y \not < x$, or $y < x$.

{\it Case 1: $y \not < x$.}
Since $U (y) = \{x_1^{a_1}, \dots, x_{j-1}^{a_{j-1}}, x_j^{a_j} \}$, we have $i > j$.
By definition, $x' = (\varphi_{x_1}^{a_1} \circ \cdots \circ \varphi_{x_{i-1}}^{a_{i-1}}) (x_i)$ and, by Corollary \ref{corl4_3}\,(2), $y' = (\varphi_{x_1}^{a_1} \circ \cdots \circ \varphi_{x_{i-1}}^{a_{i-1}}) (x_j)$, hence $\{x',y'\} \in E (\Gamma)$, we have $x|| y$ if and only if $x' || y'$, and we have $x < y$ if and only if $x' < y'$.
Moreover, since $y \not < x$, we have $y' \not < x'$, hence $\varphi_{x'}^{-a} (y') = y'$.

We have $L (U, y^b) = \{ x_1^{a_1}, \dots, x_{j-1}^{a_{j-1}}, x_{j+1}^{a_{j+1}}, \dots, x_p^{a_p} \}$, hence $x'' = (\varphi_{x_1}^{a_1} \circ \cdots \circ \varphi_{x_{j-1}}^{a_{j-1}} \circ \varphi_{x_{j+1}}^{a_{j+1}} \circ \cdots \circ \varphi_{x_{i-1}}^{a_{i-1}}) (x_i)$.
On the other hand, if $k > j$, then $x_k \not > x_j$, hence $\varphi_{x_k}^{a_k} (x_j) = x_j$.
So,
\[
y' = (\varphi_{x_1}^{a_1} \circ \cdots \circ \varphi_{x_{j-1}}^{a_{j-1}}) (x_j) = (\varphi_{x_1}^{a_1} \circ \cdots \circ \varphi_{x_{j-1}}^{a_{j-1}} \circ \varphi_{x_{j+1}}^{a_{j+1}} \circ \cdots \circ \varphi_{x_{i-1}}^{a_{i-1}}) (x_j)\,.
\]
It follows that $\{x'', y' \} \in E (\Gamma)$ and $y' \not < x''$.
In particular, $\varphi_{x''}^{-a} (y') = y'$.

We have $L (U, x^a) = \{ x_1^{a_1}, \dots, x_{i-1}^{a_{i-1}}, x_{i+1}^{a_{i+1}}, \dots, x_p^{a_p} \}$, hence $y'' = (\varphi_{x_1}^{a_1} \circ \cdots \circ \varphi_{x_{j-1}}^{a_{j-1}}) (x_j) = y'$.
So, $\varphi_{x'}^{-a} (y') = \varphi_{x''}^{-a} (y') = y'' = y'$.

Let $V$ be another stratum such that both $x'^a$ and $y'^b$ can be added to $V$.
Set $V = \{z_1^{c_1}, \dots, z_q^{c_q} \}$.
The support of $R (V, x'^a)$ is contained in $\{\varphi_{x'}^{-a} (z_1), \dots, \varphi_{x'}^{-a} (z_q)\} \cup \{x'\}$.
Since $y'^b$ can be added to $V$, $y'' = \varphi_{x'}^{-a}(y')$, and $\{x', y'\} \in E(\Gamma)$, it follows that $y''^b$ can be added to $R (V, x'^a)$.

{\it Case 2: $y < x$.}
Since $U (y) = \{x_1^{a_1}, \dots, x_{j-1}^{a_{j-1}}, x_j^{a_j} \}$, we have $i < j$.
We have $y' = (\varphi_{x_1}^{a_1} \circ \cdots \circ \varphi_{x_{j-1}}^{a_{j-1}}) (x_j)$ and, by Corollary \ref{corl4_3}\,(2), $x' = (\varphi_{x_1}^{a_1} \circ \cdots \circ \varphi_{x_{j-1}}^{a_{j-1}}) (x_i)$, hence $\{x', y' \} \in E (\Gamma)$ and $y' < x'$.
Furthermore, $L(U, y^b) = \{x_1^{a_1}, \dots, x_{j-1}^{a_{j-1}}, x_{j+1}^{a_{j+1}}, \dots, x_p^{a_p}\}$, hence, $x'' = (\varphi_{x_1}^{a_1} \circ \cdots \circ \varphi_{x_{i-1}}^{a_{i-1}}) (x_i) = x'$.
In particular, $\{x'', y'\} \in E (\Gamma)$ and $y' < x''$.

Since $L (U, x^a) = \{x_1^{a_1}, \dots, x_{i-1}^{a_{i-1}}, x_{i+1}^{a_{i+ 1}}, \dots, x_p^{a_p} \}$, we have $y'' = (\varphi_{x_1}^{a_1} \circ \cdots \circ \varphi_{x_{i-1}}^{a_{i-1}} \circ \varphi_{x_{i+1}}^{a_{i+1}} \circ \cdots \circ \varphi_{x_{j-1}}^{a_{j-1}}) (x_j)$.
On the other hand, by applying Lemma \ref{lem4_7} several times we get
\begin{gather*}
\varphi_{x''}^{-a} (y') =
\varphi_{x'}^{-a} (y') =
\big( \varphi_{(\varphi_{x_1}^{a_1} \circ \cdots \circ \varphi_{x_{i-1}}^{a_{i-1}}) (x_i)}^{-a} \circ \varphi_{x_1}^{a_1} \circ \cdots \circ \varphi_{x_{j-1}}^{a_{j-1}} \big) (x_j) =\\
\big( \varphi_{x_1}^{a_1} \circ \varphi_{(\varphi_{x_2}^{a_2} \circ \cdots \circ \varphi_{x_{i-1}}^{a_{i-1}}) (x_i)}^{-a} \circ \varphi_{x_2}^{a_2} \circ \cdots \circ \varphi_{x_{j-1}}^{a_{j-1}} \big) (x_j) =
\cdots =\\
(\varphi_{x_1}^{a_1} \circ \cdots \circ \varphi_{x_{i-1}}^{a_{i-1}} \circ \varphi_{x_i}^{-a} \circ \varphi_{x_i}^a \circ \varphi_{x_{i+1}}^{a_{i+1}} \circ \cdots \circ \varphi_{x_{j-1}}^{a_{j-1}}) (x_j) = \\
(\varphi_{x_1}^{a_1} \circ \cdots \circ \varphi_{x_{i-1}}^{a_{i-1}} \circ \varphi_{x_{i+1}}^{a_{i+1}} \circ \cdots \circ \varphi_{x_{j-1}}^{a_{j-1}}) (x_j) = 
y''\,.
\end{gather*}

Let $V$ be another stratum such that both $x'^a$ and $y'^b$ can be added to $V$.
Then we show that $y''^b$ can be added to $R (V, x'^a)$ in the same way as in Case 1.
\end{proof}

\begin{lem}\label{lem4_12}
In Case (C3) every critical pair $(u_1, u_2, u_3, v_1, v_2)$ is resolved.
\end{lem}

\begin{proof}
There exist $U, V \in \Omega$ and $y^b, z^c \in V$ such that $y \neq z$, $u_1 = u_3 = \epsilon$, $u_2 = (U,V) $, $v_1 = T (U, V, y^b)$ and $v_2 = T (U, V, z^c)$.
Let $y'^b = \gamma (V, y^b)$.
Then $y'^b$ can be added to $U$ and $T (U, V, y^b) = (U_1, V_1)$, where $U_1 = R(U, y'^b)$ and $V_1 = L (V, y^b)$.
Let $z'^c = \gamma (V, z^c)$.
Then $z'^c$ can be added to $U$ and $T (U, V, z^c) = (U_2, V_2)$, where $U_2 = R (U, z'^c)$ and $V_2 = L(V, z^c)$.
We must show that there exists $w \in \Omega^*$ such that $(U_1, V_1) \to^* w$ and $(U_2, V_2) \to^* w$.

Let $z''^c = \gamma (V_1, z^c)$ and $y''^b = \gamma (V_2, y^b)$.
We know from Lemma \ref{lem4_11}\,(4) that $z''^c$ can be added to $U_1$ and $y''^b$ can be added to $U_2$.
Let $(U_3, V_3) = T (U_1, V_1, z^c)$ and $(U_4, V_4) = T (U_2, V_2, y^b)$, where $U_3 = R (U_1, z''^c)$, $V_3 = L (V_1, z^c)$, $U_4 = R (U_2, y''^b)$ and $V_4 = L(V_2, y^b)$.
We will prove that $(U_3, V_3) = (U_4, V_4)$.
Then, by setting $w = (U_3, V_3) = (U_4, V_4)$, we will have $(U_1, V_1) \to w$ and $(U_2, V_2) \to w$.
Clearly, $V_3 = V_4 = V \setminus \{ y^b, z^c\}$, hence we only need to prove that $U_3 = U_4$.

We set $U= \{x_1^{a_1}, \dots, x_p^{a_p} \}$ and we assume as always that, if $x_i > x_j$, then $i < j$.
The rest of the proof is divided into two cases depending on whether $y || z$, or $y < z$.
The case $z < y$ is treated in the same way as the case $y < z$ by interchanging the roles of $y$ and $z$.

{\it Case 1: $y || z$.}
By Lemma \ref{lem4_11} we have $y' || z'$, $y'' = \varphi_{z'}^{-c} (y') = y'$, and $z'' = \varphi_{y'}^{-b} (z') = z'$.
Let
\begin{gather*}
W_1 = \{ x_i^{a_i} \mid 1 \le i \le p \text{ and } x_i \le y'\}\,, \
W_2 = \{ x_i^{a_i} \mid 1 \le i \le p \text{ and } x_i \le z'\}\,, \\
W_3 = U \setminus (W_1 \cup W_2)\,.
\end{gather*}
Condition (b) in the definition of a trickle graph implies that $W_1 \cap W_2 = \emptyset$, hence we have the disjoint union $U = W_1 \sqcup W_2 \sqcup W_3$.
We have $U_1 = R(U, y'^b) = W_1' \sqcup W_2 \sqcup W_3$, where
\begin{gather*}
W_1' = \{ \varphi_{y'}^{-b} (x_i)^{a_i} \mid 1 \le i \le p \text{ and } x_i \le y' \} \cup \{y'^b \} \text{ if } y' \not \in \supp (W_1)\,,\\
W_1' = \{ \varphi_{y'}^{-b} (x_i)^{a_i} \mid 1 \le i \le p\,,\ x_i \le y' \text{ and } i \neq k \} \cup \{y'^{a_k+b} \} \text{ if } y' = x_k \in \supp (W_1)\\
\hskip 5 truecm \text{ and } a_k + b \neq 0\,,\\
W_1' = \{ \varphi_{y'}^{-b} (x_i)^{a_i} \mid 1 \le i \le p\,,\ x_i \le y' \text{ and } i \neq k \} \text{ if } y' = x_k \in \supp (W_1) \text{ and } a_k + b = 0\,.
\end{gather*}
Then $U_3 = R (U_1, z''^c) = R(U_1, z'^c) = W_1' \sqcup W_2' \sqcup W_3$, where
\begin{gather*}
W_2' = \{ \varphi_{z'}^{-c} (x_i)^{a_i} \mid 1 \le i \le p \text{ and } x_i \le z' \} \cup \{z'^c \} \text{ if } z' \not \in \supp (W_2)\,,\\
W_2' = \{ \varphi_{z'}^{-c} (x_i)^{a_i} \mid 1 \le i \le p\,,\ x_i \le z' \text{ and } i \neq \ell \} \cup \{z'^{a_\ell+c} \} \text{ if } z' = x_\ell \in \supp (W_2)\\
\hskip 5 truecm \text{ and } a_\ell + c \neq 0\,,\\
W_2' = \{ \varphi_{z'}^{-c} (x_i)^{a_i} \mid 1 \le i \le p\,,\ x_i \le z' \text{ and } i \neq \ell \} \text{ if } z' = x_\ell \in \supp (W_2) \text{ and } a_\ell + c = 0\,.
\end{gather*}
By using the same argument we prove that $U_2= R(U, z'^c) = W_1 \sqcup W_2' \sqcup W_3$, and then that $U_4 = R(U_2, y''^b) = R (U_2, y'^b) = W_1' \sqcup W_2' \sqcup W_3$, hence $U_3 = U_4$.

{\it Case 2: $y < z$.}
By Lemma \ref{lem4_11} we have $y' < z'$, $z'' = \varphi_{y'}^{-b} (z') = z'$, $y''=\varphi_{z'} ^{-c} (y')$ and $y'' < z'$.
Let 
\begin{gather*}
W_1 = \{ x_i^{a_i} \mid 1 \le i \le p \text{ and } x_i \le y' \}\,,\ 
W_2 = \{ x_i^{a_i} \mid 1 \le i \le p\,,\ x_i \le z' \text{ and } x_i \not \le y' \}\,,\\
W_3 = U \setminus (W_1 \cup W_2)\,.
\end{gather*}
By construction we have the disjoint union $U = W_1 \sqcup W_2 \sqcup W_3$.
We have $U_1 = R (U, y'^b) = W_1^{(1)} \sqcup W_2 \sqcup W_3$, where 
\begin{gather*}
W_1^{(1)} = \{ \varphi_{y'}^{-b} (x_i)^{a_i} \mid 1 \le i \le p \text{ and } x_i \le y' \} \cup \{y'^b \} \text{ if } y' \not \in \supp (W_1)\,,\\
W_1^{(1)} = \{ \varphi_{y'}^{-b} (x_i)^{a_i} \mid 1 \le i \le p\,,\ x_i \le y' \text{ and } i \neq k \} \cup \{y'^{a_k+b} \} \text{ if } y' = x_k \in \supp (W_1)\\
\hskip 4 truecm \text{ and } a_k + b \neq 0\,,\\
W_1^{(1)} = \{ \varphi_{y'}^{-b} (x_i)^{a_i} \mid 1 \le i \le p\,,\ x_i \le y' \text{ and } i \neq k \} \text{ if } y' = x_k \in \supp (W_1) \text{ and } a_k + b = 0\,.
\end{gather*}
Then $U_3 = R (U_1, z''^c) = R(U_1, z'^c) = W_1^{(2)} \sqcup W_2' \sqcup W_3$, where 
\begin{gather*}
W_1^{(2)} = \{(\varphi_{z'}^{-c} \circ \varphi_{y'}^{-b}) (x_i)^{a_i} \mid 1 \le i \le p \text{ and } x_i \le y' \} \cup \{y''^b \} \text{ if } y' \not \in \supp (W_1)\,,\\
W_1^{(2)} = \{ (\varphi_{z'}^{-c} \circ \varphi_{y'}^{-b}) (x_i)^{a_i} \mid 1 \le i \le p\,,\ x_i \le y' \text{ and } i \neq k \} \cup \{y''^{a_k+b} \}\\
\hskip 4 truecm \text{ if } y' = x_k \in \supp (W_1) \text{ and } a_k + b \neq 0\,,\\
W_1^{(2)} = \{ (\varphi_{z'}^{-c} \circ \varphi_{y'}^{-b}) (x_i)^{a_i} \mid 1 \le i \le p\,,\ x_i \le y' \text{ and } i \neq k \} \text{ if } y' = x_k \in \supp (W_1)\\
\hskip 4 truecm \text{ and } a_k + b = 0\,,
\end{gather*}
and
\begin{gather*}
W_2' = \{ \varphi_{z'}^{-c} (x_i)^{a_i} \mid 1 \le i \le p\,,\ x_i \le z' \text{ and } x_i \not \le y' \} \cup \{z'^c \} \text{ if } z' \not \in \supp (W_2)\,,\\
W_2' = \{ \varphi_{z'}^{-c} (x_i)^{a_i} \mid 1 \le i \le p\,,\ x_i \le z'\,,\ x_i \not \le y' \text{ and } i \neq \ell \} \cup \{z'^{a_\ell+c} \}\\
\hskip 4 truecm \text{ if } z' = x_\ell \in \supp (W_2) \text{ and } a_\ell + c \neq 0\,,\\
W_2' = \{ \varphi_{z'}^{-c} (x_i)^{a_i} \mid 1 \le i \le p\,,\ x_i \le z'\,,\ x_i \not \le y' \text{ and } i \neq \ell \} \text{ if } z' = x_\ell \in \supp (W_2)\\
\hskip 4 truecm \text{ and } a_\ell + c = 0\,.
\end{gather*}
On the other hand, $U_2 = R (U, z'^c) = W_1^{(3)} \sqcup W_2' \sqcup W_3$, where 
\[
W_1^{(3)} = \{ \varphi_{z'}^{-c} (x_i)^{a_i} \mid 1 \le i \le p \text{ and } x_i \le y'\}\,.
\]
Then $U_4 = R (U_2, y''^b) = W_1^{(4)} \sqcup W_2' \sqcup W_3$, where 
\begin{gather*}
W_1^{(4)} = \{(\varphi_{y''}^{-b} \circ \varphi_{z'}^{-c}) (x_i)^{a_i} \mid 1 \le i \le p \text{ and } x_i \le y' \} \cup \{y''^b \} \text{ if } y' \not \in \supp (W_1)\,,\\
W_1^{(4)} = \{ (\varphi_{y''}^{-b} \circ \varphi_{z'}^{-c}) (x_i)^{a_i} \mid 1 \le i \le p\,,\ x_i \le y' \text{ and } i \neq k \} \cup \{y''^{a_k+b} \}\\
\hskip 4 truecm \text{ if } y' = x_k \in \supp (W_1) \text{ and } a_k + b \neq 0\,,\\
W_1^{(4)} = \{ (\varphi_{y''}^{-b} \circ \varphi_{z'}^{-c}) (x_i)^{a_i} \mid 1 \le i \le p\,,\ x_i \le y' \text{ and } i \neq k \}\text{ if } y' = x_k \in \supp (W_1)\\
\hskip 4 truecm \text{ and } a_k + b = 0\,.
\end{gather*}
Since $y'' = \varphi_{z'}^{-c} (y')$, by Lemma \ref{lem4_7} we have $(\varphi_{y''}^{-b} \circ \varphi_{z'}^{-c}) (x_i) = (\varphi_{z'}^{-c} \circ \varphi_{y'}^{-b}) (x_i)$ for all $i \in \{1, \dots, p\}$ such that $x_i \le y'$, hence $W_1^{(4)} = W_1^{(2)}$, and therefore $U_3 = U_4$.
\end{proof}

\begin{proof}[Proof of Proposition \ref{prop4_5}]
We know from Lemma \ref{lem4_4} that $\RR$ is terminating, hence, by Theorem \ref{thm4_1}, it suffices to prove that every critical pair is resolved.
We have seen that there are three types of critical pairs.
In Case (1) we know that every critical pair is resolved by Lemma \ref{lem4_8}, in Case (2) we know that every critical pair is resolved by Lemma \ref{lem4_10}, and in Case (3) we know that every critical pair is resolved by Lemma \ref{lem4_12}.
\end{proof}

Now we know that $\RR$ is terminating and confluent (see Lemma \ref{lem4_4} and Proposition \ref{prop4_5}), hence, in order to prove Theorem \ref{thm2_4} it remains to show that $\RR$ is a rewriting system for $\Tr (\Gamma)$, that is, to prove that $\Tr (\Gamma) \simeq \langle \Omega \mid u=v \text{ for } (u,v) \in \RR \rangle^+$.

We set $M (\Gamma) = \langle \Omega \mid u=v \text{ for } (u,v) \in \RR \rangle^+$.
We will construct homomorphisms $\Phi: M (\Gamma) \to \Tr (\Gamma)$ and $\Psi: \Tr (\Gamma) \to M (\Gamma)$, and then we will prove that $\Psi \circ \Phi = \id_{M (\Gamma)}$ and $\Phi \circ \Psi = \id_{\Tr (\Gamma)}$.

Let $U$ be a non-empty stratum.
We write $U = \{ x_1^{a_1}, \dots, x_p^{a_p} \}$ so that, if $x_i > x_j$, then $i < j$, and we set
\[
\omega (U) = x_1^{a_1} x_2^{a_2} \dots x_p^{a_p} \in \Tr (\Gamma)\,.
\]
If $U = \emptyset$, then we set $\omega (U) = 1$.

\begin{lem}\label{lem4_13}
Let $U \in \Omega$.
Then the definition of $\omega (U)$ does not depend on the choice of the numbering of the elements of $U$.
\end{lem}

\begin{proof}
Suppose we have $U = \{ x_1^{a_1}, \dots, x_p^{a_p} \} = \{ {x_1'}^{a_1'}, \dots, {x_p'}^{a_p'} \} $ so that, if $x_i > x_j$, then $i < j$, and if $x_i' > x_j'$, then $i < j$.
We prove that $x_1^{a_1} \dots x_p^{a_p} = {x_1'}^{a_1'} \dots {x_p'}^{a_p'}$ by induction on $p = \lg_\st (U )$.
If $p=0$, then $U = \emptyset$, and there is nothing to prove.
We assume that $p \ge 1$ and that the induction hypothesis holds.
Let $i \in \{1, \dots, p \}$ be such that $x_1^{a_1} = {x_i'}^{a_i'}$.
Let $j \in \{1, \dots, p\}$.
If $j < i$, then either $x_j' > x_i'$ or $x_j' || x_i'$.
But, since ${x_i'}^{a_i'} = x_1^{a_1}$, there is no ${x_j'}^{a_j'} \in U$ such that $x_j' > x_i'$, hence, if $j < i$, then $x_j' || x_i'$.
It follows that
\begin{gather*}
{x_1'}^{a_1'} \dots {x_{i-1}'}^{a_{i-1}'} {x_i'}^{a_i'} {x_{i+1}'}^{a_{i+1}'} \dots {x_p'}^{a_p'} = {x_i'}^{a_i'} {x_1'}^{a_1'} \dots {x_{i-1}'}^{a_{i-1}'} {x_{i+1}'}^{a_{i+1}'} \dots {x_p'}^{a_p'} =\\
x_1^{a_1} {x_1'}^{a_1'} \dots {x_{i-1}'}^{a_{i-1}'} {x_{i+1}'}^{a_{i+1}'} \dots {x_p'}^{a_p'}\,.
\end{gather*}
By the induction hypothesis applied to $U \setminus \{ x_1^{a_1} \}$,  
\[
x_2^{a_2} \dots x_p^{a_p} = {x_1'}^{a_1'} \dots {x_{i-1}'}^{a_{i-1}'} {x_{i+1}'}^{a_{i+1}'} \dots {x_p'}^{a_p'}\,,
\]
hence $x_1^{a_1} \dots x_p^{a_p} = {x_1'}^{a_1'} \dots {x_p'}^{a_p'}$. 
\end{proof}

Let $\Phi^* : \Omega^* \to \Tr (\Gamma)$ be the homomorphism which sends $U$ to $\omega (U)$ for all $U \in \Omega$.
It is easily verified that, for all $(u,v) \in \RR$, $\Phi^* (u) = \Phi^* (v)$, hence $\Phi^*$ induces a homomorphism $\Phi : M (\Gamma) \to \Tr (\Gamma)$ which sends $U$ to $\omega (U)$ for all $U \in \Omega$.

For the construction of $\Psi : \Tr (\Gamma) \to M (\Gamma)$ we use the following monoid presentation for $\Tr (\Gamma)$:
\[
\Tr (\Gamma) = \left\langle S (\Gamma) \left|
\begin{array}{ll}
x^a \cdot x^b = x^{a+b} &\text{for } x \in V(\Gamma),\ a, b \in \Z_{\mu (x)} \setminus \{0\} \text{ and } a+b \neq 0\,,\\
x^a \cdot x^b = 1  &\text{for } x \in V(\Gamma),\ a, b \in \Z_{\mu (x)} \setminus \{0\} \text{ and } a+b = 0\,,\\
\varphi_x^a(y)^b \cdot x^a = \varphi_y^b (x)^a \cdot y^b &\text{for } \{x, y\} \in E (\Gamma),\ a \in \Z_{\mu (x)} \setminus \{ 0 \},\
 b \in \Z_{\mu (y)} \setminus \{ 0 \}
\end{array}
\right. \right\rangle^+\,.
\]

\begin{lem}\label{lem4_14}
The homomorphism $\Psi^*: S (\Gamma)^* \to M (\Gamma)$ which sends $x^a$ to $(\{x^a\})$ induces a homomorphism $ \Psi: \Tr (\Gamma) \to M (\Gamma)$.
\end{lem}

\begin{proof}
Let $x \in V (\Gamma)$ and $a,b \in \Z_{\mu (x)} \setminus \{ 0 \}$ be such that $a + b \neq 0$.
We have the following sequence of rewriting transformations
\[
(\{x^a\}, \{x^b\}) \stackrel{\RR}{\to} (\{x^{a+b} \}, \emptyset) \stackrel{\RR}{ \to} (\{ x^{a+b} \})\,,
\]
hence $\Psi^* (x^a \cdot x^b) = \Psi^* (x^{a+b})$.
Let $x \in V (\Gamma)$ and $a,b \in \Z_{\mu (x)} \setminus \{ 0 \}$ be such that $a + b = 0$.
We have the following sequence of rewriting transformations 
\[
(\{x^a\}, \{x^b\}) \stackrel{\RR}{\to} (\emptyset, \emptyset) \stackrel{\RR}{\to} (\emptyset) \stackrel {\RR}{\to} \epsilon\,,
\]
hence $\Psi^* (x^a \cdot x^b) = \Psi^* (1)$.
Let $x^a, y^b \in S (\Gamma)$ be such that $\{x, y\} \in E (\Gamma)$ and $x || y$.
We have the following sequences of rewriting transformations
\begin{gather*}
( \{\varphi_x^a(y)^b\},  \{x^a\} ) = ( \{y^b\},  \{x^a\} ) \stackrel{\RR}{\to} ( \{x^a, y^b\}, \emptyset) \stackrel{\RR}{\to} ( \{x^a, y^b\})\,, \\
( \{ \varphi_y^b (x)^a\},  \{y^b\} ) = ( \{ x^a\},  \{y^b\} ) \stackrel{\RR}{\to} (\{ x^a, y^b\}, \emptyset) \stackrel{\RR}{\to} ( \{x^a, y^b\})\,,
\end{gather*}
hence $\Psi^* (\varphi_x^a(y)^b \cdot x^a) = \Psi^*(\varphi_y^b (x)^a \cdot y^b)$. 
Let $x^a, y^b \in S (\Gamma)$ be such that $\{x, y\} \in E (\Gamma)$ and $x  < y$. 
We have the following sequences of rewriting transformations
\begin{gather*}
( \{\varphi_x^a(y)^b\},  \{x^a\} ) = ( \{y^b\},  \{x^a\} ) \stackrel{\RR}{\to} ( \{ y^b, x^a\}, \emptyset ) \stackrel{\RR}{\to} ( \{ y^b, x^a\})\,, \\
( \{ \varphi_y^b (x)^a\},  \{y^b\} ) \stackrel{\RR}{\to} ( \{ y^b, (\varphi_y^{-b} \circ \varphi_y^b)(x)^a \}, \emptyset ) = ( \{ y^b, x^a\}, \emptyset ) \stackrel{\RR}{\to}  ( \{ y^b, x^a\})\,,
\end{gather*}
hence $\Psi^* (\varphi_x^a(y)^b \cdot x^a) = \Psi^*(\varphi_y^b (x)^a \cdot y^b)$. 
\end{proof}

\begin{lem}\label{lem4_15}
We have $\Psi \circ \Phi = \id_{M (\Gamma)}$ and $\Phi \circ \Psi = \id_{\Tr (\Gamma)}$.
\end{lem}

\begin{proof}
Let $U \in \Omega$.
We write $U = \{ x_1^{a_1}, x_2^{a_2}, \dots, x_p^{a_p} \}$ so that, if $x_i > x_j$, then $i < j$.
We have the following sequence of rewriting transformations in $\Omega^*$:
\begin{gather*}
( \{ x_1^{a_1}\}, \{x_2^{a_2}\}, \dots, \{x_p^{a_p} \}) \to ( \{ x_1^{a_1}, x_2^{a_2}\}, \emptyset, \{x_3^{a_3}\}, \dots, \{x_p^{a_p} \}) \to ( \{ x_1^{a_1}, x_2^{a_2}\}, \{x_3^{a_3}\}, \dots, \{x_p^{a_p} \})\\
\to \cdots \to (\{x_1^{a_1}, \dots, x_{p-1}^{a_{p-1}} \}, \{x_p^{a_p}\} ) \to (\{x_1^{a_1}, \dots, x_{p-1}^{a_{p-1}}, x_p^{a_p} \}, \emptyset ) \to (\{x_1^{a_1}, \dots, x_{p-1}^{a_{p-1}}, x_p^{a_p} \})  = (U)\,.
\end{gather*}
So,
\[
(\Psi \circ \Phi) ((U)) =  \Psi ( x_1^{a_1} x_2 ^{a_2} \dots x_p^{a_p} ) = ( \{ x_1^{a_1}\}, \{x_2^{a_2}\}, \dots, \{x_p^{a_p} \} ) = (U)\,.
\]
Since $M (\Gamma)$ is generated by $\{ (U) \mid U \in \Omega \}$, this shows that $\Psi \circ \Phi = \id_{M (\Gamma)}$.

For $x^a \in S (\Gamma)$ we have $(\Phi \circ \Psi) (x^a) = \Phi (\{x^a\}) = x^a$.
Since $\Tr (\Gamma)$ is generated by $S (\Gamma)$, we conclude that $\Phi \circ \Psi = \id_{\Tr (\Gamma)}$.
\end{proof}

The following is a straightforward consequence of Lemma \ref{lem4_15}, and it ends the proof of Theorem \ref{thm2_4}.

\begin{prop}\label{prop4_16}
The pair $(\Omega,\RR)$ is a rewriting system for $\Tr (\Gamma)$.
\end{prop}


\section{The Tits-style algorithm}\label{sec5}

In this section we fix a trickle graph $\Gamma = (\Gamma, \le, \mu, (\varphi_x)_{x \in V(\Gamma)})$ and we keep the notations of Subsection \ref{subsec2_3}.
Our goal is to prove Theorem \ref{thm2_8}.

We fix a total order $\preceq$ on $V(\Gamma)$ which extends the partial order $\le$ in the sense that, if $x \le y$, then $x \preceq y$.
If $U$ is a non-empty stratum written $U = \{x_1^{a_1}, x_2^{a_2}, \dots, x_p^{a_p}\}$ with $x_1 \succ x_2 \succ \cdots \succ x_p$, then we set $\hat \omega^S (U) = (x_1^{a_1}, x_2^{a_2}, \dots, x_p^{a_p}) \in S (\Gamma)^*$.
On the other hand we set $\hat \omega^S (U) = \epsilon$ if $U = \emptyset$.
Notice that $\hat \omega^S (U)$ is a representative of $\omega (U)$.

The following three lemmas are preliminaries to the proof of Theorem \ref{thm2_8}.

\begin{lem}\label{lem5_1}
Let $U = \{x_1^{a_1}, \dots, x_p^{a_p} \}$ be a stratum.
We suppose that the numbering of the elements of $U$ is chosen so that, if $x_i > x_j$, then $i < j$.
Let $v = (x_1^{a_1}, \dots, x_p^{a_p}) \in S (\Gamma)^*$.
Then $v \stackrel{II\ *}{\to} \hat \omega^S (U)$.
\end{lem}

\begin{proof}
We argue by induction on $p = \lg_\st (U)$.
If $p = 0$, then $U = \emptyset$ and $v = \hat \omega^S (U) = \epsilon$.
So, we can assume that $p \ge 1$ and that the induction hypothesis holds. 
We write $U=\{ {x_1'}^{a_1'}, \dots, {x_p'}^{a_p'} \}$ with $x_1' \succ \cdots \succ x_p'$.
Let $i \in \{1, \dots, p \}$ be such that ${x_1'}^{a_1'} = x_i^{a_i}$.
Let $j \in \{1, \dots, p\}$.
If $j < i$, then either $x_j > x_i$ or $x_j || x_i$.
But, since $x_i^{a_i} = {x_1'}^{a_1'}$, there is no $x_j^{a_j} \in U$ such that $x_j > x_i=x_1'$, hence, if $j < i$, then $ x_j || x_i$.
It follows that
\begin{gather*}
v= (x_1^{a_1}, \dots, x_{i-1}^{a_{i-1}}, x_i^{a_i}, x_{i+1}^{a_{i+1}}, \dots, x_p^{a_p}) \stackrel{II}{\to}
(x_1^{a_1}, \dots, x_{i-2}^{a_{i-2}}, x_i^{a_i}, x_{i-1}^{a_{i-1}},  x_{i+1}^{a_{i+1}}, \dots, x_p^{a_p}) \stackrel{II}{\to} \cdots \stackrel{II}{\to}\\
(x_i^{a_i}, x_1^{a_1}, \dots, x_{i-1}^{a_{i-1}}, x_{i+1}^{a_{i+1}}, \dots, x_p^{a_p}) = ({x_1'}^{a_1'}, x_1^{a_1}, \dots, x_{i-1}^{a_{i-1}}, x_{i+1}^{a_{i+1}}, \dots, x_p^{a_p})\,.
\end{gather*}
By the induction hypothesis applied to $U \setminus \{ x_i^{a_i} \}$ we have  
\[
(x_1^{a_1}, \dots, x_{i-1}^{a_{i-1}}, x_{i+1}^{a_{i+1}}, \dots, x_p^{a_p}) \stackrel{II\ *}{\to} ({x_2'}^{a_2'}, \dots, {x_p'}^{a_p'})\,.
\]
So,
\[
v = (x_1^{a_1}, \dots, x_p^{a_p}) \stackrel{II\ *}{\to} ({x_1'}^{a_1'}, \dots, {x_p'}^{a_p'}) = \hat \omega^S (U)\,.
\proved
\]
\end{proof}

\begin{lem}\label{lem5_2}
Let $U$ be a non-empty stratum and let $x^a \in U$.
Let $y^a = \gamma (x^a, U)$.
Then $\hat \omega^S (U) \stackrel{II\ *}{\to} (y^a) \cdot \hat \omega^S (L (U, x^a))$.
\end{lem}

\begin{proof}
We write $U = \{ x_1^{a_1}, \dots, x_p^{a_p} \}$ with $x_1 \succ x_2 \succ \cdots \succ x_p$.
Let $i \in \{1, \dots, p \}$ be such that $x^a = x_i^{a_i}$.
Note that, if $j < i$, then $x_j || x_i$ or $x_j > x_i$.
This implies that, if $j < i$, then $x_j || (\varphi_{x_{j+1}}^{a_{j+1}} \circ \cdots \circ \varphi_{x_{i-1}}^{a_{i-1}}) (x_i)$ or $x_j > (\varphi_{x_{j+1}}^{a_{j+1}} \circ \cdots \circ \varphi_{x_{i-1}}^{a_{i-1}}) (x_i)$.
Then we have the following sequence of rewriting transformations
\begin{gather*}
\hat \omega^S (U)  =
(x_1^{a_1}, \dots, x_{i-1}^{a_{i-1}}, x_i^{a_i}, x_{i+1}^{a_{i+1}}, \dots, x_p^{a_p}) \stackrel{II}{\to}\\
(x_1^{a_1}, \dots, x_{i-2}^{a_{i-2}}, \varphi_{x_{i-1}}^{a_{i-1}} (x_i)^{a_i}, x_{i-1}^{a_{i-1}}, x_{i+1}^{a_{i+1}}, \dots, x_p^{a_p}) \stackrel{II}{\to} \cdots \stackrel{II}{\to}\\
(x_1^{a_1}, (\varphi_{x_2}^{a_2} \circ \cdots \circ \varphi_{x_{i-1}}^{a_{i-1}}) (x_i)^{a_i}, x_2^{a_2}, \dots, x_{i-1}^{a_{i-1}}, x_{i+1}^{a_{i+1}}, \dots, x_p^{a_p})  \stackrel{II}{\to}\\
( (\varphi_{x_1}^{a_1} \circ \cdots \circ \varphi_{x_{i-1}}^{a_{i-1}}) (x_i)^{a_i}, x_1^{a_1}, \dots, x_{i-1}^{a_{i-1}}, x_{i+1}^{a_{i+1}}, \dots, x_p^{a_p}) =
(y^a) \cdot \hat \omega^S (L (U, x^a))\,.
\end{gather*}
\end{proof}

\begin{lem}\label{lem5_3}
Let $U \in \Omega$ be a stratum and let $y^b \in S (\Gamma)$ be a syllable which can be added to $U$.
Then
\[
\hat \omega^S (U) \cdot (y^b) \stackrel{M\ *}{\to} \hat \omega^S (R (U, y^b))\,.
\]
\end{lem}

\begin{proof}
We write $U = \{x_1^{a_1}, \dots, x_p^{a_p}\}$ with $x_1 \succ \cdots \succ x_p$.
Let $i \in \{0, 1, \dots, p\}$ be such that $x_i \succeq y \succ x_{i+1}$.
If $i=0$ this means that $y \not \in \supp (U)$ and $y \succ x_1$, and if $i=p$ this means that $x_p \succeq y$.
If $y \in \supp (U)$, then $y=x_i$, otherwise $x_i \succ y \succ x_{i+1}$.
Let $j \in \{i+1, \dots, p\}$.
Then either $y > x_j$ or $y || x_j$.
In both cases we have
\[
((x_j^{a_j}, y^b) , (y^b,\varphi_y^{-b} (x_j)^{a_j})) \in \RR_{II}\,.
\]
It follows that 
\[
\hat \omega^S (U) \cdot (y^a) = (x_1^{a_1}, \dots, x_p^{a_p}, y^a) \stackrel{II \ *}{\to} (x_1^{a_1}, \dots, x_i^{a_i}, y^b, \varphi_y^{-b} (x_{i+1})^{a_{i+1}}, \dots, \varphi_y^{-b} (x_p)^{a_p})\,.
\]
If $y \not \in \supp (U)$, then 
\[
R(U, y^b) = \{x_1^{a_1}, \dots, x_i^{a_i}, y^b, \varphi_y^{-b} (x_{i+1})^{a_{i+1}}, \dots, \varphi_y^{-b} (x_p)^{a_p}\}\,,
\]
hence, by Lemma \ref{lem5_1},
\[
(x_1^{a_1}, \dots, x_i^{a_i}, y^b, \varphi_y^{-b} (x_{i+1})^{a_{i+1}}, \dots, \varphi_y^{-b} (x_p)^{a_p}) \stackrel{II \ *}{\to} \hat \omega^S (R (U, y^b))\,.
\]
If $y = x_i \in \supp (U)$ and $b+a_i \neq 0$, then  
\[
R(U, y^b) = \{x_1^{a_1}, \dots, x_{i-1}^{a_{i-1}}, y^{a_i+b}, \varphi_y^{-b} (x_{i+1})^{a_{i+1}}, \dots, \varphi_y^{-b} (x_p)^{a_p}\}\,,
\]
hence, by Lemma \ref{lem5_1},
\begin{gather*}
(x_1^{a_1}, \dots, x_{i-1}^{a_{i-1}}, x_i^{a_i}, y^b, \varphi_y^{-b} (x_{i+1})^{a_{i+1}}, \dots, \varphi_y^{-b} (x_p)^{a_p}) \stackrel{I }{\to}\\
(x_1^{a_1}, \dots, x_{i-1}^{a_{i-1}}, y^{a_i+b}, \varphi_y^{-b} (x_{i+1})^{a_{i+1}}, \dots, \varphi_y^{-b} (x_p)^{a_p}) \stackrel{II \ *}{\to} \hat \omega^S (R (U, y^b))\,.
\end{gather*}
If $y = x_i \in \supp (U)$ and $b+a_i = 0$, then  
\[
R(U, y^b) = \{x_1^{a_1}, \dots, x_{i-1}^{a_{i-1}}, \varphi_y^{-b} (x_{i+1})^{a_{i+1}}, \dots, \varphi_y^{-b} (x_p)^{a_p}\}\,,
\]
hence, by Lemma \ref{lem5_1},
\begin{gather*}
(x_1^{a_1}, \dots, x_{i-1}^{a_{i-1}}, x_i^{a_i}, y^b, \varphi_y^{-b} (x_{i+1})^{a_{i+1}}, \dots, \varphi_y^{-b} (x_p)^{a_p}) \stackrel{I}{\to}\\
(x_1^{a_1}, \dots, x_{i-1}^{a_{i-1}}, \varphi_y^{-b} (x_{i+1})^{a_{i+1}}, \dots, \varphi_y^{-b} (x_p)^{a_p}) \stackrel{II \ *}{\to} \hat \omega^S (R (U, y^b))\,.
\end{gather*}
\end{proof}

The following definition is a slight modification of the definition of normal form given in Subsection \ref{subsec2_2}.
Let $g \in \Tr (\Gamma)$.
By Theorems \ref{thm2_3} and \ref{thm2_4} there exists a unique piling $w = (U_1,U_2, \dots, U_p) \in \Omega^*$ such that $w$ is $\RR$-irreducible and $w = \Psi (g)$.
Then we define the \emph{$S$-normal form} of $g$ to be 
\[
\nf^S (g) = \hat \omega^S (U_1) \cdot \hat \omega^S (U_2) \cdots \hat \omega^S (U_p) \in S(\Gamma)^*\,.
\]
Theorem \ref{thm2_8} will be a consequence of the following.

\begin{prop}\label{prop5_4}
Let $v \in S (\Gamma)^*$ and let $g = \bar v \in \Tr (\Gamma)$.
Then $v \stackrel{M\ *}{\to} \nf^S (g)$.
\end{prop}

\begin{proof}
For a piling $w = (U_1, \dots, U_p) \in \Omega^*$ we set
\[
\hat \omega^S (w) = \hat \omega^S (U_1) \cdot \hat \omega^S (U_2) \cdots \hat \omega^S (U_p) \in S(\Gamma)^*\,.
\]
Let $w, w' \in \Omega^*$.
It follows from Lemmas \ref{lem5_2} and \ref{lem5_3} that, if $w \stackrel{\RR\ *}{\to} w'$, then $\hat \omega^S (w) \stackrel{M\ *}{\to} \hat \omega^S (w')$.
Let $v = (x_1^{a_1}, \dots, x_p^{a_p}) \in S (\Gamma)^*$ and let $g = \bar v \in \Tr (\Gamma)$.
Set $w = ( \{ x_1^{a_1} \}, \{ x_2^{a_2} \}, \dots, \{x_p^{a_p} \}) \in \Omega^*$ and  denote by $w'$ the unique piling in $\Omega^*$ such that $w'$ is $\RR$-irreducible and $\overline{w'} = g$. 
Then, by Theorems \ref{thm2_3} and \ref{thm2_4}, we have $w \stackrel{\RR\ *}{\to} w'$, hence, by the above,
\[
v = \hat \omega^S (w) \stackrel{M \ *}{\to} \hat \omega^S (w') = \nf^S (g)\,.
\proved
\]
\end{proof}

\begin{proof}[Proof of Theorem \ref{thm2_8}]
To prove Theorem \ref{thm2_8} it suffices to show that, if $w, w' \in S (\Gamma)^*$ are syllabically $M$-reduced and $\overline{w} = \overline{w'}$, then $w \stackrel{II\ *}{\to} w'$.
Let $w, w' \in S (\Gamma)^*$ be two syllabically $M$-reduced words and let $g \in \Tr (\Gamma)$ be such that $\overline{w} = \overline{w'} = g$.
By Proposition \ref{prop5_4} we have $w \stackrel{M\ *}{\to} \nf^S (g)$ and $w' \stackrel{M\ *}{\to} \nf^S (g)$.
Since $w$ and $w'$ are both syllabically $M$-reduced, we actually must have $w \stackrel{II\ *}{\to} \nf^S (g)$ and $w' \stackrel {II\ *}{\to} \nf^S (g)$.
Finally, since the operation $\stackrel{II}{\to}$ is reversible, we conclude $w \stackrel{II\ *}{\to} \nf^S (g) \stackrel{II\ *}{\to} w'$.
\end{proof}


\section{Parabolic subgroups}\label{sec6}

\begin{proof}[Proof of Theorem \ref{thm2_10}]
Note that $\Omega (\Gamma_1) \subseteq \Omega (\Gamma)$, hence $\Omega (\Gamma_1)^* \subseteq \Omega (\Gamma)^*$.
It is easily seen that, if $U \in \Omega (\Gamma_1)$ and $x^a \in U$, then $\gamma (U, x^a) \in S (\Gamma_1)$.
It is also easily seen that, if $U \in \Omega (\Gamma_1)$ and $x^a \in S (\Gamma_1)$ can be added to $U$, then $R (U, x^a) \in \Omega (\Gamma_1)$.
Let $w \in \Omega (\Gamma_1)^*$ and $w' \in \Omega (\Gamma)^*$.
Then from the above observations it follows that, if $w \stackrel{\RR (\Gamma)\ *}{\longrightarrow} w'$, then $w' \in \Omega (\Gamma_1)^*$ and $w \stackrel{\RR (\Gamma_1)\ *}{\longrightarrow} w'$.

Let $g \in \Tr (\Gamma_1)$.
Write $g = x_1^{a_1} \dots x_p^{a_p}$ with $x_1^{a_1}, \dots, x_p^{a_p} \in S (\Gamma_1)$, and set
\[
v = (\{x_1^{a_1}\}, \{x_2^{a_2}\}, \dots, \{x_p^{a_p}\}) \in \Omega (\Gamma_1)^*\,.
\]
Let $w \in \Omega (\Gamma)^*$ be the unique $\RR (\Gamma)$-irreducible piling such that $v \stackrel{\RR (\Gamma)\ *}{\longrightarrow} w$.
From the above we get $w \in \Omega (\Gamma_1)^*$ and $v \stackrel{\RR (\Gamma_1)\ *}{\longrightarrow} w$.
Moreover, $w$ is $\RR (\Gamma_1)$-irreducible because it is $\RR (\Gamma)$-irreducible.
Write $w = (U_1, \dots, U_q)$ with $U_1, \dots, U_q \in \Omega (\Gamma_1)$.
Then
\[
\nf_\Gamma (\iota_1 (g)) =\nf_{\Gamma_1} (g) = \hat \omega (U_1) \cdot \hat \omega (U_2) \cdots \hat \omega (U_q)\,.
\]

Let $g \in \Ker (\iota_1)$.
From the above we get $\nf_{\Gamma_1} (g) = \nf_\Gamma (\iota_1 (g)) = \epsilon$, hence $g=1$.
This shows that $\iota_1$ is injective.

Let $g \in \Tr (\Gamma)$.
We already know that, if $g \in \Tr (\Gamma_1)$, then $\nf_\Gamma (g) = \nf_{\Gamma_1} (g)$, hence $\nf_\Gamma (g) \in (V (\Gamma_1) \sqcup V (\Gamma_1)^{-1})^*$.
Conversely, it is obvious that, if $\nf_\Gamma (g) \in (V (\Gamma_1) \sqcup V (\Gamma_1)^{-1})^*$, then $g = \overline{\nf_\Gamma (g)} \in \Tr (\Gamma_1)$.
\end{proof}


\section{PreGarside trickle groups}\label{sec7}

In this section $\Gamma = (\Gamma, \le, \mu, (\varphi_x)_{x \in V (\Gamma)})$ denotes a fixed preGarside trickle graph.
We also fix a total order $\preceq$ on $V (\Gamma)$ which extends the partial order $\le$ in the sense that, for all $x, y \in V (\Gamma)$, if $x \le y$, then $x \preceq y$.

The proof of Theorem \ref{thm2_14} is based on the existence of a complemented presentation for $\Tr^+ (\Gamma)$ that satisfies the sharp $\theta$-cube condition.
So, we start by recalling these notions and explaining their link with Theorem \ref{thm2_14}.

\begin{defin}
Let $\Lambda$ be a simplicial graph.
As always we denote by $V (\Lambda)$ its vertex set and by $E (\Lambda)$ its edge set.
Let $\hat E (\Lambda) = \{ (x,y) \in V (\Lambda) \times V (\Lambda) \mid \{ x, y \} \in E (\Lambda) \}$.
A \emph{partial complement} on $V (\Lambda)$ based on $\Lambda$ is a map $f : \hat E (\Lambda) \to V (\Lambda)^*$.
A presentation $M = \langle V \mid R \rangle^+$ for a monoid $M$ is called \emph{complemented} if there exist two partial complements $f$ on a graph $\Lambda$ and $g$ on a graph $\Lambda'$ such that $V (\Lambda) = V (\Lambda') = V$ and
\[
R = \{ x\, f(x,y) = y \, f(y,x) \mid \{x,y\} \in E(\Lambda) \} = \{ g(y,x)\,x = g(x,y)\,y \mid \{x,y\} \in E(\Lambda') \}\,.
\]
Such a presentation is called \emph{short} if $f(x,y) \in V$ for all $(x,y) \in \hat E (\Lambda)$, that is, if all relations in $R$ are of the form $xy' = yx'$ with $x,y,x',y' \in V$.
\end{defin}

\begin{rem}
The definition of short presentation given here is more restrictive than the one given in \cite{DDGKM1}, but it allows to consider only homogeneous monoid presentations. 
In particular, this implies that the considered monoids are all atomic (see the remark below).
\end{rem}

\begin{expl}
Let $\Gamma = (\Gamma, \le, \mu, (\varphi_x)_{x \in V (\Gamma)})$ be our preGarside trickle graph.
For all $(x, y) \in \hat E (\Gamma)$ we set $f (x, y) = \varphi_x^{-1} (y)$ and $g (x, y) = \varphi_y (x)$.
Then, by definition and (the proof of) Proposition \ref{prop2_2}, $\Tr^+ (\Gamma)$ has the presentation $\Tr^+ (\Gamma) = \langle V \mid R \rangle$, where $V = V(\Gamma)$ and
\[
R = \{ x\, f (x, y) = y \, f (y, x) \mid \{ x, y \} \in E (\Gamma) \} = \{ g(y, x) \, x = g (x, y) \, y \mid \{ x, y \} \in E (\Gamma) \} \,.
\]
In particular, this presentation is a complemented short presentation.
\end{expl}

\begin{rem}
Let $M$ be a monoid with a complemented short presentation $M = \langle V \mid R \rangle$.
Define a map $\nu : M \to \N$ as follows.
If $a \in M$ is written $a=x_1 \dots x_p$ with $x_1, \dots, x_p  \in V$, then set $\nu (a) =p$.
The map $\nu$ is well-defined because the relations in $R$ are homogeneous, hence the length $p$ does not depend on the choice of the expression of $a$.
It is clear that $\nu (a) = 0$ if and only if $a = 1$ and $\nu (ab) = \nu (a) + \nu (b)$ for all $a, b \in M$, hence $\nu$ is a norm and $M$ is atomic.
In particular, $\Tr^+ (\Gamma)$ is atomic.
\end{rem}

The sharp $\theta$-cube condition introduced in \cite[Definition II.4.14]{DDGKM1} is technical, but in our case where the presentations are short, its definition is simpler as, in particular, there is no need of extending the partial map $\star$ to $V^*$ as in the general case.

\begin{defin}
Let $\langle V \mid R \rangle^+$ be a complemented short presentation of a monoid $M$.
Recall that there exist a graph $\Lambda$ with $V = V (\Lambda)$ and a partial complement $g : \hat E (\Lambda) \to V$ such that
\[
R = \{ g(y,x)\, x = g(x, y)\, y \mid \{ x, y \} \in E (\Lambda) \}\,.
\]
We define a partial map $(x ,y) \mapsto x \star y$ from $(V \cup\{ \epsilon\}) \times (V \cup\{ \epsilon\})$ to $V \cup \{ \epsilon \}$ as follows.
Let $x, y \in V$ be such that $x \neq y$. 
If $\{x, y\} \in E (\Lambda)$, then we set $x \star y = g (x, y)$.
If $\{x, y\} \not \in E (\Lambda)$, then neither $x \star y$ nor $y \star x$ is defined.
On the other hand, for $x \in V \cup \{ \epsilon \}$ we set $x \star x = \epsilon \star x = \epsilon$ and $x \star \epsilon = x$.
We say that the presentation $\langle V \mid R \rangle^+$ satisfies the \emph{left sharp $\theta$-cube condition} if, for all $x, y, z \in V$ pairwise distinct, we have the following alternative:
\begin{itemize}
\item
either $(z \star x) \star (y \star x)$ and $(z \star y) \star (x \star y)$ are both defined and are equal,
\item
or neither expression is defined. 
\end{itemize}
We define similarly the \emph{right sharp $\theta$-cube condition}.
We say that the presentation satisfies the \emph{sharp $\theta$-cube condition} if it satisfies both, the left and the right sharp $\theta$-cube conditions.
\end{defin}

The main tool in the proof of Theorem \ref{thm2_14} comes from \cite[Proposition II.4.16]{DDGKM1}.

\begin{prop}[Dehornoy--Digne--Godelle--Krammer--Michel 
\cite{DDGKM1}]\label{prop7_1}
If a monoid $M$ admits a short complemented presentation that satisfies the sharp $\theta$-cube condition, then $M$ is a preGarside monoid.
\end{prop}

\begin{proof}[Proof of Theorem \ref{thm2_14}]
We already know that $\Tr^+ (\Gamma)$ has a short complemented presentation.
So, in order to prove Theorem \ref{thm2_14}, it suffices to show that the standard presentation of $\Tr^+ (\Gamma)$ satisfies the sharp $\theta$-cube condition.
By symmetry and thanks to Proposition \ref{prop2_2}, it suffices to show that this presentation satisfies the left sharp $\theta$-cube condition.

Let $x, y \in V (\Gamma)$ be such that $x \neq y$.
If $\{x, y\} \in E (\Gamma)$, then we set $x \star y = \varphi_{y} (x)$, and, if $\{x, y\} \not \in E(\Gamma)$, then $x \star y$ is not defined.
On the other hand, for $x \in V (\Gamma) \cup \{\epsilon\}$ we set $\epsilon \star x = x \star x = \epsilon$ and $x \star \epsilon = x$.
The fact that the standard presentation of $\Tr^+ (\Gamma)$ satisfies the left sharp $\theta$-cube condition is a straightforward consequence of the following two claims.

{\it Claim 1.}
For all $x, y, z \in V$ pairwise distinct we have the following alternative:
\begin{itemize}
\item
either $( z \star x) \star ( y \star x)$ and $(z \star y) \star (x \star y)$ are both defined,
\item
or neither expression is defined.
\end{itemize}
Furthermore, $(z \star x) \star (y \star x)$ and $(z \star y) \star (x \star y)$ are both defined if and only if $\{x, y\}$, $\{x, z\}$, and $\{y, z\}$ belong to $E (\Gamma)$.

{\it Proof of Claim 1.}
By symmetry between $x$ and $y$, it suffices to show that $(z \star x) \star (y \star x)$ is defined if and only if $\{x, y\}$, $\{x ,z\}$ and $\{y, z\}$ belong to $E (\Gamma)$.
Note that, if $y \star x$ is defined, then $y \star x = \varphi_x (y)$ lies in $V (\Gamma)$ (and it is different from $\epsilon$), and, if $y \star x = \varphi_x (y)$ and $z \star x = \varphi_x (z)$ are both defined, then they are different.
So, saying that $(z \star x) \star (y \star x)$ is defined is equivalent of saying that $\{x, y\}$, $\{x, z\}$ and $\{(y \star x), (z \star x) \} = \{\varphi_x (y), \varphi_x (z) \}$ belong to $E (\Gamma)$.
But $\{\varphi_x (y), \varphi_x (z)\}$ belongs to $E (\Gamma)$ if and only if $\{y, z\}$ belongs to $E (\Gamma)$, hence $(z \star x) \star (y \star x)$ is defined if and only if $\{x, y\}$, $\{x, z\}$ and $\{ y, z \}$ belong to $E (\Gamma)$.
This completes the proof of Claim 1.

{\it Claim 2.}
Let $x, y, z \in V(\Gamma)$ pairwise distinct be such that $\{x, y\}, \{x, z\}, \{y, z\} \in E(\Gamma)$.
Then 
\[
(z \star x) \star (y \star x) = (z \star y) \star (x \star y)\,.
\]

{\it Proof of Claim 2.}
The proof of Claim 2 is divided into two cases depending on whether $x || y$ or $y < x$.
By symmetry, the case $x < y$ is treated in the same way as the case $y < x$.

{\it Case 1:} $x || y$.
By Condition (b) in the definition of a trickle graph, we cannot have together $z < x$ and $z < y$.
Thus, we have two subcases: $z<x$ and $z \not < y$ for the first, and $z \not < x$ and $z \not < y$ for the second.
The subcase $z \not < x$ and $z < y$ is treated in the same way as the subcase $z<x$ and $z \not < y$.
Suppose $z < x$ and $z \not < y$.
Then $\varphi_x (z) \not < \varphi_x (y) = y$, hence
\[
(z \star x) \star (y \star x) = \varphi_x (z) \star y = \varphi_x (z) = z \star x = (z \star y) \star (x \star y)\,.
\]
Suppose $z \not < x$ and $z \not < y$.
Then
\[
(z \star x) \star (y \star x) = z \star y = z = z \star x = (z \star y) \star (x \star y)\,.
\]

{\it Case 2:} $y < x$.
Here we have three subcases: $z \not < x$ (which implies $z \not < y$) for the first, $z < x$ and $z \not < y$ for the second, and $z < y$ (which implies $z < y < x$) for the third.
Suppose $z \not < x$.
Then $z = \varphi_x (z) \not < \varphi_x (y)$, hence
\[
(z \star x) \star (y \star x) = z \star \varphi_x (y) = z = z \star x = (z \star y) \star (x \star y)\,.
\]
Suppose $z < x$ and $z \not < y$.
Then $\varphi_x (z) \not < \varphi_x (y)$, hence
\[
(z \star x) \star (y \star x) = \varphi_x (z) \star \varphi_x (y) = \varphi_x (z) = z \star x = (z \star y) \star (x \star y)\,.
\]
Suppose $z < y$.
Then, by Condition (g) in the definition of a trickle graph,
\[
(z \star x) \star (y \star x) = \varphi_x (z) \star \varphi_x (y) = (\varphi_{\varphi_x (y)} \circ \varphi_x) (z) = (\varphi_x \circ \varphi_y) (z) = \varphi_y (z) \star x = (z \star y) \star (x \star y)\,.
\]
This completes the proof of Claim 2 and therefore of Theorem \ref{thm2_14}.
\end{proof}

For the remaining proofs in this section we need to understand how the rewriting system $\RR$ of Section \ref{sec4} can be restricted to positive words.
Then Theorem \ref{thm2_17} will be a straightforward consequence of this analysis, and therefore it will be proved before Theorem \ref{thm2_15} and Theorem \ref{thm2_16}.

Set $S^+ (\Gamma) = \{ x^a \mid x \in V (\Gamma) \text{ and } a \in \N_{\ge 1} \}$.
A \emph{positive stratum} of $\Gamma$ is a finite subset $U = \{x_1^{a_1}, x_2^{a_2}, \dots, x_p^{a_p}\} \subseteq S^+ (\Gamma)$ such that $x_i \neq x_j$ and $\{x_i, x_j\} \in E (\Gamma)$ for all $i, j \in \{1, \dots, p\}$, $i \neq j$.
The set of positive strata is denoted by $\Omega^+ = \Omega^+ (\Gamma)$.
Note that $\Omega^+ \subset \Omega$, hence $(\Omega^+)^* \subset \Omega^*$.
It is easily seen that, if $U \in \Omega^+$ and $x^a \in U$, then $\gamma (U, x^a) \in S^+ (\Gamma)$.
It is also easily seen that, if $U \in \Omega^+$ and $x^a \in S^+ (\Gamma)$ can be added to $U$, then $R (U, x^a) \in \Omega^+$.
Let $\RR^+ = \RR \cap ((\Omega^+)^* \times (\Omega^+)^*)$.
Then the following is a direct consequence of the above observations combined with Theorem \ref{thm2_4}.

\begin{lem}\label{lem7_2}
\begin{itemize}
\item[(1)]
$\RR^+$ is a terminating and confluent rewriting system.
\item[(2)]
Let $w \in (\Omega^+)^*$ and $w' \in \Omega^*$.
If $w \stackrel{\RR\ *}{\longrightarrow} w'$, then $w' \in (\Omega^+)^*$ and $w \stackrel{\RR^+\ *}{\longrightarrow} w'$.
\end{itemize}
\end{lem}

Set $M^+ (\Gamma) = \langle \Omega^+ \mid u = v \text{ for } (u, v) \in \RR^+ \rangle^+$.
Let $U\in \Omega^+$, $U \neq \emptyset$.
Write $U = \{ x_1^{a_1}, \dots, x_p^{a_p} \}$ with $x_1 \succ x_2 \succ \cdots \succ x_p$, and set $\omega^+ (U) = x_1^{a_1} \dots x_p^{a_p} \in \Tr^+ (\Gamma)$.
If $U = \emptyset$, then set $\omega^+ (U) = 1$.
The same proofs as the ones of Lemmas \ref{lem4_14} and \ref{lem4_15} applied to $M^+ (\Gamma)$ and $\Tr^+ (\Gamma)$ prove the following.

\begin{lem}\label{lem7_3}
We have an isomorphism $\Phi^+ : M^+ (\Gamma) \to \Tr^+ (\Gamma)$ which sends $U$ to $\omega^+ (U)$ for all $U \in \Omega^+$.
\end{lem}

If $U$ is a non-empty positive stratum that we write $U = \{ x_1^{a_1}, x_2^{a_2}, \dots, x_p^{a_p} \}$ with $x_1 \succ x_2 \succ \cdots \succ x_p$, then we set $\hat \omega^+ (U) = x_1^{a_1} \cdot x_2^{a_2} \cdots x_p^{a_p} \in V (\Gamma)^*$.
Note that $\hat \omega^+ (U)$ is a representative of $\omega^+(U)$.
Let $g \in \Tr^+ (\Gamma)$.
By the above there exists a unique piling $w = (U_1, U_2, \dots, U_p) \in (\Omega^+)^*$ such that $w$ is $\RR^+$-irreducible and $w$ represents $g$.
Then we set
\[
\nf^+ (g) = \hat \omega^+ (U_1) \cdot \hat \omega^+ (U_2) \cdots \hat \omega^+ (U_p) \in V (\Gamma)^*\,.
\]
The following is a direct consequence of the above observations.

\begin{lem}\label{lem7_4}
\begin{itemize}
\item[(1)] 
Let $g \in \Tr^+ (\Gamma)$.
Then $\nf^+ (g) = \nf (\iota_{\Tr^+ (\Gamma)} (g))$.
\item[(2)] 
Let $g \in \Tr (\Gamma)$.
We have $g \in \iota_{\Tr^+ (\Gamma)} (\Tr^+ (\Gamma))$ if and only if $\nf (g) \in V(\Gamma)^*$.
\end{itemize}
\end{lem}

\begin{proof}[Proof of Theorem \ref{thm2_17}]
Let $g, h \in \Tr^+ (\Gamma)$.
If $\iota_{\Tr^+ (\Gamma)} (g) = \iota_{\Tr^+ (\Gamma)} (h)$, then we have $\nf ( \iota_{\Tr^+ (\Gamma)} (g)) = \nf ( \iota_{\Tr^+ (\Gamma)} (h))$, hence, by Lemma \ref{lem7_4}, $\nf^+ (g) = \nf^+ (h)$, and therefore $g=h$.
\end{proof}

Now, we turn to the proof of Theorem \ref{thm2_15}.
The following two lemmas are preliminaries of it.

\begin{lem}\label{lem7_5}
Let $g \in \Tr^+ (\Gamma)$.
Let $w = (U_1, \dots, U_p)$ be the unique piling in $(\Omega^+)^*$ such that $\Phi^+ (\bar w) = g$ and $w$ is $\RR^+$-irreducible.
Write $U_1=\{ x_1^{a_1}, \dots, x_q^{a_q} \}$ with $x_1 \succ x_2 \succ \cdots \succ x_q$.
Let $\psi = \varphi_{x_1}^{a_1} \circ \varphi_{x_2}^{a_2} \circ \cdots \circ \varphi_{x_q}^{a_q}$.
Then $\Div_L (g) \cap V (\Gamma) = \psi (\{x_1, x_2, \dots, x_q\})$.
\end{lem}

\begin{proof}
Let $y \in \psi (\{x_1, x_2, \dots, x_q\})$.
There exists $i \in \{1, \dots, q\}$ such that $y = \psi (x_i)$.
Since $x_j \not > x_i$ for $j \ge i$, we have $y = \psi(x_i) =(\varphi_{x_1}^{a_1} \circ \cdots \circ \varphi_{x_{i-1}}^{a_{i-1}}) (x_i)$.
Then 
\begin{gather*}
g = x_1^{a_1} \dots x_{i-1}^{a_{i-1}} x_i^{a_i} x_{i+1}^{a_{i+1}} \dots x_q^{a_q} \, \omega (U_2) \dots \omega (U_p) =\\
x_1^{a_1} \dots x_{i-2}^{a_{i-2}} \varphi_{x_{i-1}}^{a_{i-1}} (x_i)^{a_i} x_{i-1}^{a_{i-1}} x_{i+1}^{a_{i+1}} \dots x_q^{a_q} \, \omega (U_2) \dots \omega (U_p) = \cdots =\\
(\varphi_{x_1}^{a_1} \circ \cdots \circ \varphi_{x_{i-1}}^{a_{i-1}}) (x_i)^{a_i}\, x_1^{a_1} \dots x_{i-1}^{a_{i-1}} x_{i+1}^{a_{i+1}} \dots x_q^{a_q} \, \omega (U_2) \dots \omega (U_p) =\\
y\,y^{a_i-1} x_1^{a_1} \dots x_{i-1}^{a_{i-1}} x_{i+1}^{a_{i+1}} \dots x_q^{a_q} \, \omega (U_2) \dots \omega (U_p)\,,
\end{gather*}
hence $y \in \Div_L (g) \cap V (\Gamma)$.

Let $y \in \Div_L (g) \cap V (\Gamma)$.
Let $g' \in \Tr^+ (\Gamma)$ be such that $g = yg'$.
Let $w'$ be the unique piling in $(\Omega^+)^*$ such that $\Phi^+ (\overline{w'}) = g'$ and $w'$ is $\RR^+$-irreducible.
Let $u = \{y\} \cdot w' \in (\Omega^+)^*$.
We have $\Phi^+ (\bar u) = g$, hence $u \stackrel{\RR^+\ *}{\longrightarrow} w$.
In other words, there exists a finite sequence $u_0=u, u_1, \dots, u_\ell=w$ in $(\Omega^+)^*$ such that $u_{j-1} \stackrel{\RR^+}{\to} u_j$ for all $j \in \{1, \dots, \ell \}$.
For each $j \in \{0,1, \dots, \ell \}$ we denote by $V_j$ the first stratum in the piling $u_j$ and we set $V_j = \{z_{1,j}^{c_{1,j}}, \dots, z_{r_j,j}^{c_{r_j,j}} \}$ with $z_{1,j} \succ \cdots \succ z_{r_j,j}$.

{\it Claim.} 
Let $j \in \{ 0, 1, \dots, \ell\}$.
Then there exists $i \in \{1, \dots, r_j \}$ such that $y = (\varphi_{z_{1,j}}^{c_{1,j}} \circ \cdots \circ \varphi_{z_{i-1,j}}^{c_{i-1,j}}) (z_{i,j})$.

{\it Proof of the claim.}
We argue by induction on $j$.
Suppose $j = 0$.
Then $V_0 = \{y\}$ and the claim is trivial.
Suppose $j > 0$ and that the induction hypothesis holds.
By the induction hypothesis there exists $i \in \{1, \dots, r_{j-1} \}$ such that $y = (\varphi_{z_{1,j-1}}^{c_{1,j-1}} \circ \cdots \circ \varphi_{z_{i-1,j-1}}^{c_{i-1,j-1}}) (z_{i,j-1})$.
If $V_j = V_{j-1}$, then there is nothing to prove.
So, we can assume that $V_j \neq V_{j-1}$.
Then there exists a syllable $z^c$ that can be added to $V_{j-1}$ such that $V_j = R(V_{j-1},z^c)$.
Let $k \in \{0, \dots, r_{j-1} \}$ be such that $z_{k,j-1} \succeq z \succ z_{k+1,j-1}$.
If $k=0$, then this means that $z \succ z_{1,j-1}$.
If $k=r_{j-1}$, then this means that $z_{r_{j-1},j-1} \succeq z$.
Also, note that, if $z = z_{k,j-1}$, then $c + c_{k,j-1}>0$, hence $c + c_{k,j-1} \neq 0$.
If $k \ge i$, then
\[
(\varphi_{z_{1,j}}^{c_{1,j}} \circ \cdots \circ \varphi_{z_{i-1,j}}^{c_{i-1,j}}) (z_{i,j}) = (\varphi_{z_{1,j-1}}^{c_{1,j-1}} \circ \cdots \circ \varphi_{z_{i-1,j-1}}^{c_{i-1,j-1}}) (z_{i,j-1}) = y\,.
\] 
If $k < i$ and $z \neq z_{k,j-1}$, then, by applying Lemma \ref{lem4_7} several times, 
\begin{gather*}
(\varphi_{z_{1,j}}^{c_{1,j}} \circ \cdots \circ \varphi_{z_{i,j}}^{c_{i,j}})  (z_{i+1,j}) =\\
(\varphi_{z_{1,j-1}}^{c_{1,j-1}} \circ \cdots \circ \varphi_{z_{k,j-1}}^{c_{k,j-1}} \circ \varphi_z^c \circ  \varphi_{\varphi_z^{-c} (z_{k+1,j-1})}^{c_{k+1,j-1}} \circ \cdots \circ \varphi_{\varphi_z^{-c} (z_{i-1,j-1})}^{c_{i-1,j-1}}) (\varphi_z^{-c} (z_{i,j-1})) =\\ 
(\varphi_{z_{1,j-1}}^{c_{1,j-1}} \circ \cdots \circ \varphi_{z_{k,j-1}}^{c_{k,j-1}} \circ \varphi_z^c \circ \varphi_z^{-c} \circ \varphi_{z_{k+1,j-1}}^{c_{k+1,j-1}} \circ \cdots \circ \varphi_{z_{i-1,j-1}}^{c_{i-1,j-1}}) (z_{i,j-1}) =\\
(\varphi_{z_{1,j-1}}^{c_{1,j-1}} \circ \cdots \circ \varphi_{z_{i-1,j-1}}^{c_{i-1,j-1}}) (z_{i,j-1}) = y\,.
\end{gather*}
If $k < i$ and $z= z_{k,j-1}$, then we prove in the same way that
\[
(\varphi_{z_{1,j}}^{c_{1,j}} \circ \cdots \circ \varphi_{z_{i-1,j}}^{c_{i-1,j}}) (z_{i,j}) = y\,.
\]
This completes the proof of the claim.

Applying the claim to $j=\ell$ we see that there exists $i \in \{1, \dots, q \}$ such that $y = (\varphi_{x_1}^{a_1} \circ \cdots \circ \varphi_{x_{i-1}}^{a_{i-1}}) (x_i) = \psi (x_i)$.
\end{proof}

An element $g \in \Tr^+ (\Gamma)$ is called \emph{square-free} if it can be written in the form $g = x_1 x_2 \dots x_p$ with $x_1 \succ x_2 \succ \cdots \succ x_p$.
We denote by $\SF (\Gamma)$ the set of square-free elements.

\begin{lem}\label{lem7_6}
Suppose $V (\Gamma)$ is finite and $\Gamma$ is complete.
Write $V (\Gamma) = \{x_1, \dots, x_n\}$ with $x_1 \succ x_2 \succ \cdots \succ x_n$, and set $\Delta = x_1 x_2 \dots x_n$.
Then $\Delta = \vee_L V (\Gamma) = \vee_R V (\Gamma)$ and $\Div_L (\Delta) = \Div_R (\Delta) = \SF (\Gamma)$.
In particular, $\Delta$ is balanced and $\Div (\Delta)$ generates $\Tr^+ (\Gamma)$, that is, $\Delta$ is a Garside element.
\end{lem}

\begin{proof}
Let $\nu : \Tr^+ (\Gamma) \to \N$ be the length homomorphism which sends $x_i$ to $1$ for all $i \in \{1, \dots, n\}$.
For $p \in \{0,1, \dots, n\}$ we denote by $\SF_p (\Gamma)$ the set of square-free elements of length $p$.
Let $g \in \SF_p (\Gamma)$.
We write $g = x_{i_1} x_{i_2} \dots x_{i_p}$ with $i_1 < i_2 < \cdots < i_p$, we set $\psi_g = \varphi_{x_{i_1}} \circ \cdots \circ \varphi_{x_{i_p}}$ and $X_g = \psi_g (\{x_{i_1}, \dots, x_{i_p} \})$.
We know by Lemma \ref{lem7_5} that $\Div_L (g) \cap V (\Gamma) = X_g$.
Furthermore, again by Lemma \ref{lem7_5}, if $h \in \Tr^+ (\Gamma)$ satisfies $\Div_L (h) \cap V (\Gamma) \supseteq X_g$, then $\nu (h) \ge p$.
This implies that $\vee_L X_g$ exists and $\vee_L X_g = g$.
Denote by $\PP_p (V (\Gamma))$ the set of subsets of $V (\Gamma)$ with cardinality $p$.
We have an injective map $\SF_p (\Gamma) \to \PP_p (V (\Gamma))$, $g \mapsto X_g$, and $|\SF_p (\Gamma)| = |\PP_p (V (\Gamma))| = \binom{n}{p}$, hence this map is a bijection.
In particular, since $\SF_n (\Gamma) = \{\Delta\}$ and $\PP_n (V (\Gamma)) = \{ V(\Gamma)\}$, we have $\vee_L V(\Gamma) = \Delta$.

Let $g \in \Tr^+ (\Gamma)$.
If $g \le_L \Delta$, then $g$ is square-free, otherwise by applying the trickle algorithm we see that $\Delta$ would not be square-free.
Suppose that $g$ is square-free.
Then, since $X_g \subseteq V (\Gamma)$, we have $g = \vee_L X_g \le_L \vee_L V (\Gamma) = \Delta$.
This shows that $\Div_L (\Delta) = \SF (\Gamma)$.

For $p \in \{0, 1, \dots, n\}$ we denote by $\widetilde{\SF}_p (\Gamma)$ the set of elements of $\Tr^+ (\Gamma)$ that can be written in the form $x_{i_p} \dots x_{i_2} x_{i_1}$ with $i_1 < i_2 < \cdots < i_p$.
We show by induction on $p$ that $\widetilde{\SF}_p (\Gamma) = \SF_p (\Gamma)$.
The cases $p=0$ and $p=1$ are obvious, hence we can assume that $p \ge 2$ and that the induction hypothesis holds.
Let $g \in \SF_p (\Gamma)$.
Write $g$ in the form $g = x_{i_1} x_{i_2} \dots x_{i_p}$ with $i_1 < i_2 < \cdots < i_p$.
Since, for all $j \in \{2, \dots, p\}$, we have $x_{i_1} > x_{i_j}$ or $x_{i_1} || x_{i_j}$, 
\[
g= \varphi_{x_{i_1}}(x_{i_2}) \dots \varphi_{x_{i_1}}(x_{i_p}) \, x_{i_1}\,.
\]
By Lemma \ref{lem4_13}, $\varphi_{x_{i_1}}(x_{i_2}) \dots \varphi_{x_{i_1}}(x_{i_p}) \in \SF_{p-1} (\Gamma)$.
Let $V_1 = \{ x \in V (\Gamma) \mid x \not \ge x_{i_1} \}$ and let $\Gamma_1$ be the full subgraph of $\Gamma$ spanned by $V_1$.
Observe that $\Gamma_1$ is a parabolic subgraph of $\Gamma$ and $\varphi_{x_{i_1}}(x_{i_2}) \dots \varphi_{x_{i_1}}(x_{i_p}) \in \SF_{p-1} (\Gamma_1)$. 
By the induction hypothesis, $\SF_{p-1} (\Gamma_1) = \widetilde{\SF}_{p-1} (\Gamma_1)$, hence $\varphi_{x_{i_1}}(x_{i_2}) \dots \varphi_{x_{i_1}}(x_{i_p})$ can be written in the form $x_{k_p} \dots x_{k_3} x_{k_2}$ with $k_2 < k_3 < \cdots < k_p$ and $x_{k_j} \in V_1$ for all $j \in \{2, \dots, p\}$.
Note that, for $j \in \{2, \dots, p\}$, as $x_{k_j} \in V_1$, either $x_{i_1} > x_{k_j}$ or $x_{i_1} || x_{k_j}$.
Applying Lemma \ref{lem4_13} to the dual presentation of $\Tr^+ (\Gamma)$ we deduce that $g = x_{k_p} \dots x_{k_2} x_{i_1} \in \widetilde{\SF}_p (\Gamma)$.
This shows that $\SF_p (\Gamma) \subseteq \widetilde{\SF}_p (\Gamma)$.
We show in the same way that $\widetilde{\SF}_p (\Gamma) \subseteq \SF_p (\Gamma)$, hence $\widetilde{\SF}_p (\Gamma) = \SF_p (\Gamma)$.

Let $\widetilde{\SF} (\Gamma) = \bigcup_{p=0}^n \widetilde{\SF}_p (\Gamma)$.
From the equality $\SF_n (\Gamma) = \widetilde{\SF}_n (\Gamma)$ it follows that $\Delta = x_n \dots x_2 x_1$.
Moreover, by applying the previous reasoning to the dual presentation of $\Tr^+ (\Gamma)$ we get $\Div_R (\Delta) = \widetilde{\SF} (\Gamma) = \SF (\Gamma)$.
\end{proof}

\begin{proof}[Proof of Theorem \ref{thm2_15}]
Suppose $\Tr^+ (\Gamma)$ is a Garside monoid.
Let $\Delta$ be a Garside element of $\Tr^+ (\Gamma)$.
By definition, $\Div (\Delta) = \Div_L (\Delta)$ generates $\Tr^+ (\Gamma)$, hence $V (\Gamma) \subseteq \Div_L (\Delta)$, that is, $\Div_L (\Delta) \cap V (\Gamma) = V (\Gamma)$.
Now, Lemma \ref{lem7_5} implies that $\Div_L (\Delta) \cap V(\Gamma) = V(\Gamma)$ is a finite stratum, hence $V (\Gamma)$ is finite and $\Gamma$ is complete.

Suppose $V (\Gamma)$ is finite and $\Gamma$ is complete.
We write $V (\Gamma) = \{ x_1, \dots, x_n \}$ with $x_1 \succ x_2 \succ \cdots \succ x_n$, and we set $\Delta = x_1 x_2 \dots x_n$.
Then, by Lemma \ref{lem7_6}, $\Delta$ is a Garside element, hence $\Tr^+ (\Gamma)$ is a Garside monoid.
\end{proof}

\begin{proof}[Proof of Theorem \ref{thm2_16}]
We argue by induction on $|V (\Gamma)|$.
If $|V (\Gamma)| = 1$, then $\Tr (\Gamma) \simeq \Z$, hence $\Tr (\Gamma)$ is torsion-free.
So, we can assume that $|V (\Gamma)| \ge 2$ and that the induction hypothesis holds.
Let $Y$ be the set of maximal elements of $V (\Gamma)$ for the (partial) order $\le$.

First, suppose that there exist $y_1, y_2 \in Y$ such that $y_1 \neq y_2$ and $\{ y_1, y_2 \} \not \in E (\Gamma)$.
Let $X_1 = V (\Gamma) \setminus \{ y_1 \}$, $X_2 = V (\Gamma) \setminus \{ y_2 \}$ and $X_{12} = V (\Gamma) \setminus \{ y_1, y_2 \}$.
We denote by $\Gamma_1$ the full subgraph of $\Gamma$ spanned by $X_1$, by $\Gamma_2$ the full subgraph of $\Gamma$ spanned by $X_2$, and by $\Gamma_{12}$ the full subgraph of $\Gamma$ spanned by $X_{12}$.
It is easily seen that $\Gamma_1$, $\Gamma_2$, and $\Gamma_{12}$ are parabolic subgraphs of $\Gamma$.
Moreover, by Theorem \ref{thm2_10}, the embedding of $\Gamma_{12}$ into $\Gamma_k$ induces an injective homomorphism $\Tr (\Gamma_{12}) \hookrightarrow \Tr (\Gamma_k)$ and, by the induction hypothesis, $\Tr (\Gamma_k)$ is torsion-free, for $k \in \{1, 2\}$.
From the standard presentation of $\Tr (\Gamma)$ we see that $\Tr (\Gamma) = \Tr (\Gamma_1) *_{\Tr (\Gamma_{12})} \Tr (\Gamma_2)$, hence, by \cite[Chapter I, Corollary 1]{Serre1}, $\Tr (\Gamma)$ is torsion-free.

Suppose $\{y_1, y_2 \} \in E (\Gamma)$ for all $y_1, y_2 \in Y$, $y_1 \neq y_2$.
Note that, since all the elements of $Y$ are maximal for the order $\le$, we have $y_1 || y_2$ for all $y_1, y_2 \in Y$, $y_1 \neq y_2$. 
Let $X = V(\Gamma) \setminus Y$.
Denote by $\Gamma_Y$ the full subgraph of $\Gamma$ spanned by $Y$ and by $\Gamma_X$ the full subgraph of $\Gamma$ spanned by $X$.
It is easily observed that $\Gamma_X$ and $\Gamma_Y$ are both parabolic subgraphs of $\Gamma$ and that $\Tr (\Gamma_Y) \simeq \Z^{|Y|}$.
Let $y \in Y$.
Let $x \in X$ be such that $x \not < y$.
By definition, there exists $y' \in Y$ such that $x < y'$.
Since $y || y'$, Condition (b) in the definition of a trickle graph implies that $x|| y$.
Thus, $X \subseteq \starE_y (\Gamma)$ and $\varphi_y (X) = X$.
Let $y,y' \in Y$, $y \neq y'$.
Let $x \in X$.
If $x \not < y$ and $x \not < y'$, then
\[
(\varphi_y \circ \varphi_{y'}) (x) = \varphi_y (x) = x = \varphi_{y'} (x) = (\varphi_{y'} \circ \varphi_y) (x)\,.
\]
If $x < y$, then, by Condition (b) in the definition of a trickle graph, $x || y'$.
Since $\varphi_y$ is an automorphism of $\starE_y (\Gamma)$, we also have $\varphi_y(x) || y'$.
So, 
\[
(\varphi_y \circ \varphi_{y'}) (x) = \varphi_y (x) = \varphi_{y'}(\varphi_y (x)) = (\varphi_{y'} \circ \varphi_y) (x).
\]
If $x < y'$, then we prove in the same way that
\[
(\varphi_y \circ \varphi_{y'}) (x) = (\varphi_{y'} \circ \varphi_y) (x).
\]
This shows that we have an action of $\Tr (\Gamma_Y) \simeq \Z^{|Y|}$ on $\Tr (\Gamma_X)$ defined by $y \cdot x = \varphi_y (x)$ for all $y \in Y$ and $x \in X$, and that $\Tr (\Gamma) = \Tr (\Gamma_X) \rtimes \Tr (\Gamma_Y)$. 
Since $\Tr (\Gamma_Y) \simeq \Z^{|Y|}$ is torsion-free and, by the induction hypothesis, $\Tr (\Gamma_X)$ is torsion-free, it follows that $\Tr (\Gamma)$ is also torsion-free.
\end{proof}

Now, we turn to the study of parabolic submonoids and subgroups of preGarside trickle monoids and groups.
We first observe that, by repeating mutatis mutandis the proof of Theorem \ref{thm2_10}, we get the following.

\begin{lem}\label{lem7_7}
Let $\Gamma_1$ be a parabolic subgraph of $\Gamma$.
\begin{itemize}
\item[(1)]
The canonical homomorphism $\iota_{\Tr^+ (\Gamma_1)}: \Tr^+ (\Gamma_1) \to \Tr^+ (\Gamma)$ is injective.
So, we can identify $\Tr^+ (\Gamma_1)$ with $\iota_{\Tr^+ (\Gamma_1)} (\Tr^+ (\Gamma_1))$.
\item[(2)]
Let $g \in \Tr^+ (\Gamma)$.
We have $g \in \Tr^+ (\Gamma_1)$ if and only if $\nf^+ (g) \in V (\Gamma_1)^*$.
Moreover, in this case, $g$ has the same normal form in $\Tr^+ (\Gamma)$ as in $\Tr^+ (\Gamma_1)$, that is, $\nf_\Gamma^+ (g) = \nf_{\Gamma_1}^+ (g)$.
\end{itemize}
\end{lem}

As mentioned in Section \ref{sec2}, we have two different definitions of parabolicity, the one coming from the trickle groups and the one coming from the preGarside monoids.
Now, we show that these two definitions coincide.

\begin{prop}\label{prop7_8}
Let $N$ be a submonoid of $\Tr^+ (\Gamma)$.
Then $N$ is a parabolic submonoid of $\Tr^+ (\Gamma)$ if and only if there exists a parabolic subgraph $\Gamma_1$ of $\Gamma$ such that $N = \Tr^+ (\Gamma_1)$.
\end{prop}

\begin{proof}
Suppose $N$ is a parabolic submonoid of $\Tr^+ (\Gamma)$.
Set $X = V (\Gamma) \cap N$ and denote by $\Gamma_X$ the full subgraph of $\Gamma$ spanned by $X$.
Let $g \in N$.
Since $V (\Gamma)$ generates $\Tr^+ (\Gamma)$, the element $g$ can be written in the form $g = x_1 x_2 \dots x_p$ with $x_1, \dots, x_p \in V (\Gamma)$.
Since $N$ is a parabolic submonoid, $N$ is a special submonoid, hence $x_1, \dots, x_p \in N$, that is, $x_1, \dots, x_p \in X$.
This shows that $N$ is the submonoid of $\Tr^+ (\Gamma)$ generated by $X$.
Let $x, y \in X$ be such that $y < x$.
Since $\varphi_x(y) \, x = x y \in N$ and $N$ is special, we have $\varphi_x(y) \in N \cap V (\Gamma) = X$.
This implies that, for all $x \in X$, $\varphi_x$ restricts to an automorphism of $\starE_x (\Gamma_X)$.
So, $\Gamma_X$ is a parabolic subgraph of $\Gamma$ and $N = \Tr^+ (\Gamma_X)$.

Now, we take a parabolic subgraph $\Gamma_1$ of $\Gamma$ and we show that $\Tr^+ (\Gamma_1)$ is a parabolic submonoid of $\Tr^+ (\Gamma)$.

{\it Claim 1.}
$\Tr^+ (\Gamma_1)$ is special.

{\it Proof of Claim 1.}
We prove that, if $x \in V (\Gamma)$ and $h \in \Tr^+ (\Gamma)$ are such that $g= xh \in \Tr^+ (\Gamma_1)$, then $x \in V (\Gamma_1)$ and $h \in \Tr^+ (\Gamma_1)$.
Let $w = (U_1, \dots, U_p)$ be the unique piling in $(\Omega^+)^*$ such that $\Phi^+ (\bar w) = g$ and $w$ is $\RR^+$-irreducible.
Write $U_1=\{ x_1^{a_1}, \dots, x_q^{a_q} \}$ with $x_1 \succ x_2 \succ \cdots \succ x_q$.
Let $\psi = \varphi_{x_1}^{a_1} \circ \varphi_{x_2}^{a_2} \circ \cdots \circ \varphi_{x_q}^{a_q}$.
We know from Lemma \ref{lem7_5} that $x \in \psi (\{x_1, x_2, \dots, x_q\})$, hence there exists $i \in \{1, \dots, q \}$ such that $x = (\varphi_{x_1}^{a_1} \circ \cdots \circ \varphi_{x_{i-1}}^{a_{i-1}}) (x_i)$.
We have 
\[ 
g = x x^{a_i-1} x_1^{a_1} \dots x_{i-1}^{a_{i-1}} x_{i+1}^{a_{i+1}} \dots x_q^{a_q} \omega (U_2) \dots \omega (U_p)\,.
\] 
By Lemma \ref{lem7_7}, $\nf^+ (g) \in V (\Gamma_1)^*$, hence $x_1, \dots, x_q \in V (\Gamma_1)$, and therefore $x \in V (\Gamma_1)$.
Moreover, the inclusion $\nf^+ (g) \in V (\Gamma_1)^*$ also implies that $\omega (U_2), \dots, \omega (U_p) \in \Tr^+ (\Gamma_1)$, hence
\[
h = x^{a_i-1} x_1^{a_1} \dots x_{i-1}^{a_{i-1}} x_{i+1}^{a_{i+1}} \dots x_q^{a_q} \omega (U_2) \dots \omega (U_p) \in \Tr^+ (\Gamma_1)\,.
\]
This completes the proof of Claim 1.

{\it Claim 2.}
Let $g_1, g_2 \in \Tr^+ (\Gamma_1)$.
Suppose there exists $g \in \Tr^+ (\Gamma)$ such that $g_1, g_2 \le_L g$.
Then there exists $g' \in \Tr^+ (\Gamma_1)$ such that $g_1, g_2 \le_L g' \le_L g$.

{\it Proof of Claim 2.}
We argue by induction on the word length $\nu (g)$ of $g$.
If $\nu (g_1) = 0$, then $g_1 = 1$ and $g' = g_2$ satisfies $g' \in \Tr^+ (\Gamma_1)$ and $g_1, g_2 \le_L g' \le_L g$.
Similarly, if $\nu (g_2) = 0$, then $g_2 = 1$ and $g' = g_1$ satisfies $g' \in \Tr^+ (\Gamma_1)$ and $g_1, g_2 \le_L g' \le_L g$.
The case $\nu (g_1)= 0$ contains the case $\nu (g) = 0$, hence we can assume that $\nu (g_1) \ge 1$, $\nu (g_2) \ge 1$, and that the induction hypothesis holds.
Let $y_1, y_2 \in V (\Gamma_1)$ be such that $y_1 \le_L g_1$ and $y_2 \le_L g_2$, and let $h_1, h_2 \in \Tr^+ (\Gamma_1)$ be such that $g_1 = y_1 h_1$ and $g_2 = y_2 h_2$.

First, assume that $y_1 = y_2$.
Let $h \in \Tr^+ (\Gamma)$ be such that $y_1 h = g$.
Since $\Tr^+ (\Gamma)$ is cancellative, $h_1, h_2 \le_L h$, hence, by the induction hypothesis, there exists $h' \in \Tr^+ (\Gamma_1)$ such that $h_1, h_2 \le_L h' \le_L h$.
Let $g' = y_1 h'$.
Then $g' \in \Tr^+ (\Gamma_1)$ and $g_1, g_2 \le_L g' \le_L g$.

Now, assume $y_1 \neq y_2$.
Let $w = (U_1, \dots, U_p)$ be the unique piling in $(\Omega^+)^*$ such that $\Phi^+ (\bar w) = g$ and $w$ is $\RR^+$-irreducible.
Write $U_1=\{ x_1^{a_1}, \dots, x_q^{a_q} \}$ with $x_1 \succ x_2 \succ \cdots \succ x_q$.
Let $\psi = \varphi_{x_1}^{a_1} \circ \cdots \circ \varphi_{x_q}^{a_q}$.
Since $y_1, y_2 \le_L g$, by Lemma \ref{lem7_5} we have $y_1, y_2 \in \psi (\{x_1, x_2, \dots, x_q\})$.
Let $i, j \in \{1, \dots, q\}$ be such that $y_1 = \psi(x_i)$ and $y_2 = \psi(x_j)$.
We can assume without loss of generality that $i < j$.
Then, since $y_1 = \psi (x_i)$ and $y_2 = \psi (x_j)$, we have $\{y_1, y_2\} \in E (\Gamma)$ and $y_1 \not < y_2$.
Moreover,
\begin{gather*}
g = x_1^{a_1} \dots x_q^{a_q} \omega(U_2) \dots \omega (U_p) = y_2^{a_j} x_1^{a_1} \dots x_{j-1}^{a_{j-1}} x_{j+1}^{a_{j+1}} \dots x_q^{a_q} \omega(U_2) \dots \omega (U_p) =\\
y_2^{a_j} y_1^{a_i} x_1^{a_1} \dots x_{i-1}^{a_{i-1}} x_{i+1}^{a_{i+1}} \dots x_{j-1}^{a_{j-1}} x_{j+1}^{a_{j+1}} \dots x_q^{a_q} \omega(U_2) \dots \omega (U_p) =\\
y_2 y_1 (\varphi_{y_1}^{-1}(y_2))^{a_j-1} y_1^{a_i-1} x_1^{a_1} \dots x_{i-1}^{a_{i-1}} x_{i+1}^{a_{i+1}} \dots x_{j-1}^{a_{j-1}} x_{j+1}^{a_{j+1}} \dots x_q^{a_q} \omega(U_2) \dots \omega (U_p)\,,
\end{gather*}
hence $y_2 y_1 = y_1\, \varphi_{y_1}^{-1} (y_2) \le_L g$.
Let $h \in \Tr^+ (\Gamma)$ be such that $y_2 y_1 h = y_1\, \varphi_{y_1}^{-1} (y_2)\, h = g$.
We have $h_2, y_1 \le_L y_1 h$, hence, by the induction hypothesis, there exists $k_2 \in \Tr^+ (\Gamma_1)$ such that $h_2, y_1 \le_L k_2 \le_L y_1 h$.
Similarly, there exists $k_1 \in \Tr^+ (\Gamma_1)$ such that $h_1, \varphi_{y_1}^{-1} (y_2) \le_L k_1 \le_L \varphi_{y_1}^{-1}(y_2)\, h$.
Let $\ell_1, \ell_2 \in \Tr^+ (\Gamma_1)$ be such that $k_2 = y_1 \ell_2$ and $k_1 = \varphi_{y_1}^{-1} (y_2)\, \ell_1$.
We have $\ell_1, \ell_2 \le_L h$, hence, by the induction hypothesis, there exists $h' \in \Tr^+ (\Gamma_1)$ such that $\ell_1, \ell_2 \le_L h' \le_L h$.
Let $g' = y_2 y_1 h' = y_1\, \varphi_{y_1}^{-1} (y_2)\, h'$.
Then $g' \in \Tr^+ (\Gamma_1)$ and $g_1, g_2 \le_L g' \le_L g$.
This completes the proof of Claim 2.

Recall that $\Gamma_1$ is a parabolic subgraph of $\Gamma$ and that we want to prove that $\Tr^+ (\Gamma_1)$ is a parabolic submonoid of $\Tr^+ (\Gamma)$.
We know from Claim 1 that $\Tr^+ (\Gamma_1)$ is special.
Let $g_1, g_2 \in \Tr^+ (\Gamma_1)$ be such that $g_1 \vee_L g_2$ exists in $\Tr^+ (\Gamma)$.
By Claim 2 there exists $g' \in \Tr^+ (\Gamma_1)$ such that $g_1, g_2 \le_L g' \le_L g_1 \vee_L g_2$.
Then, by definition, $g' = g_1 \vee_L g_2 \in \Tr^+ (\Gamma_1)$.
Let $g_1, g_2 \in \Tr^+ (\Gamma_1)$ be such that $g_1 \vee_R g_2$ exists in $\Tr^+ (\Gamma)$.
By applying the above argument to the dual presentation of $\Tr^+ (\Gamma)$ we get that $g_1 \vee_R g_2 \in \Tr^+ (\Gamma_1)$.
\end{proof}

\begin{proof}[Proof of Theorem \ref{thm2_18}]
Let $N_1$ be a parabolic submonoid of $\Tr^+ (\Gamma)$.
By Proposition \ref{prop7_8} there exists a parabolic subgraph $\Gamma_1$ of $\Gamma$ such that $N_1 = \Tr^+ (\Gamma_1)$.
The homomorphism $G(N_1) \to \Tr (\Gamma)$ induced by the embedding $N_1 \hookrightarrow \Tr^+ (\Gamma)$ is the homomorphism $\iota_1 : \Tr (\Gamma_1) \to \Tr (\Gamma)$ induced by the embedding of $\Gamma_1$ into $\Gamma$, and we know from Theorem \ref{thm2_10} that this homomorphism is injective.

Let $g \in G(N_1) \cap \Tr^+ (\Gamma) = \Tr (\Gamma_1) \cap \Tr^+ (\Gamma)$.
By Theorem \ref{thm2_10} we have $\nf (g) \in (V (\Gamma_1) \cup V (\Gamma_1)^{-1})^*$ and by Lemma \ref{lem7_4} we have $\nf (g) \in V(\Gamma)^*$, hence $\nf (g) \in V (\Gamma_1)^*$, and therefore $g \in \Tr^+ (\Gamma_1)$.
This shows that $N_1 = \Tr^+ (\Gamma_1) = G (N_1) \cap \Tr^+ (\Gamma)$.

Let $N_2$ be another parabolic submonoid of $\Tr^+ (\Gamma)$.
As for $N_1$, there exists a parabolic subgraph $\Gamma_2$ of $\Gamma$ such that $N_2 = \Tr^+ (\Gamma_2)$.
Let $g \in N_1 \cap N_2 = \Tr^+ (\Gamma_1) \cap \Tr^+ (\Gamma_2)$.
By Lemma \ref{lem7_7} we have $\nf^+ (g) \in V(\Gamma_1)^* \cap V (\Gamma_2)^* = V (\Gamma_1 \cap \Gamma_2) ^*$, hence $g \in \Tr^+ (\Gamma_1 \cap \Gamma_2)$.
This shows that $N_1 \cap N_2 \subseteq \Tr^+ (\Gamma_1 \cap \Gamma_2)$.
Since the inclusion $\Tr^+ (\Gamma_1 \cap \Gamma_2) \subseteq N_1 \cap N_2$ is obvious, it follows that $N_1 \cap N_2 = \Tr^+ (\Gamma_1 \cap \Gamma_2)$.
In particular, by Proposition \ref{prop7_8}, $N_1 \cap N_2$ is a parabolic submonoid of $\Tr^+ (\Gamma)$. 
Finally, applying Corollary \ref{corl2_12} we get
\[
G (N_1) \cap G (N_2) = \Tr (\Gamma_1) \cap \Tr (\Gamma_2) = \Tr (\Gamma_1 \cap \Gamma_2) = G (N_1 \cap N_2)\,. \proved
\]
\end{proof}




\frenchspacing
\begin{thebibliography}{DDH$\phantom{}^+$07}

\bibitem[BE22]{BaElh1}
{\bf I Ba, M Elhamdadi,}
{\it Circular orderability and quandles,}
Preprint,  	arXiv:2204.09458, 2022.

\bibitem[BPS22]{BaPaSi1}
{\bf V\,G Bardakov, I\,B\,S Passi, M Singh,}
{\it Zero-divisors and idempotents in quandle rings,}
Osaka J. Math. 59 (2022), no. 3, 611--637.

\bibitem[BCP16]{BeCiPa1}
{\bf P Bellingeri, B\,A Cisneros de la Cruz, L Paris,}
{\it A simple solution to the word problem for virtual braid groups,}
Pacific J. Math. 283 (2016), no. 2, 271--287.

\bibitem[BP20]{BelPar1}
{\bf P Bellingeri, L Paris,}
{\it Virtual braids and permutations,}
Ann. Inst. Fourier (Grenoble) 70 (2020), no. 3, 1341--1362.

\bibitem[BPT23]{BePaTh1}
{\bf P Bellingeri, L Paris, A-L Thiel,}
{\it Virtual Artin groups,}
Proc. London Math. Soc. (3) 126 (2023), no. 1, 192--215.

\bibitem[Bes03]{Bessi1}
{\bf D Bessis,}
{\it The dual braid monoid,}
Ann. Sci. École Norm. Sup. (4) 36 (2003), no. 5, 647--683. 

\bibitem[BKL98]{BiKoLe1}
{\bf J Birman, K\,H Ko, S\,J Lee,}
{\it A new approach to the word and conjugacy problems in the braid groups,}
Adv. Math. 139 (1998), no. 2, 322--353. 

\bibitem[Bon16]{Bonna1}
{\bf C Bonnaf\'e,}
{\it Cells and cacti,}
Int. Math. Res. Not. IMRN (2016), no. 19, 5775--5800.

\bibitem[Bou68]{Bourb1}
{\bf N Bourbaki,}
{\it \'El\'ements de math\'ematique. Fasc. XXXIV. Groupes et alg\`ebres de Lie. Chapitre IV: Groupes de Coxeter et syst\`emes de Tits. Chapitre V: Groupes engendr\'es par des r\'eflexions. Chapitre VI: syst\`emes de racines,} 
Actualit\'es Scientifiques et Industrielles No 1337. Hermann, Paris, 1968.

\bibitem[BW08]{BraWat1}
{\bf T Brady, C Watt,}
{\it Non-crossing partition lattices in finite real reflection groups,}
Trans. Amer. Math. Soc. 360 (2008), no. 4, 1983--2005.

\bibitem[BS72]{BriSai1}
{\bf E Brieskorn, K Saito,}
{\it Artin-Gruppen und Coxeter-Gruppen,}
Invent. Math. 17 (1972), 245--271.

\bibitem[Bur24]{Buril1}
{\bf J. Burillo,}
{\it Introduction to Thompson's group $F$,}
Book currently being written, available at
https://web.mat.upc.edu/pep.burillo.

\bibitem[CFP96]{CaFlPa1}
{\bf J\,W Cannon, W\,J Floyd, W\,R Parry,}
{\it Introductory notes on Richard Thompson's groups,}
Enseign. Math. (2) 42 (1996), no. 3--4, 215--256.

\bibitem[CGP20]{ChGlPy1}
{\bf M Chmutov, M Glick, P Pylyavskyy,}
{\it The Berenstein-Kirillov group and cactus groups,}
J. Comb. Algebra 4 (2020), no. 2, 111--140.

\bibitem[CGW09]{CrGoWi1}
{\bf J Crisp, E Godelle, B Wiest,}
{\it The conjugacy problem in subgroups of right-angled Artin groups,}
J. Topol. 2 (2009), no. 3, 442--460.

\bibitem[DDH$\phantom{}^+$07]{DDHPV1}
{\bf M\,A Dabkowska, M\,K Dabkowski, V\,S Harizanov, J\,H Przytycki, M\,A Veve,}
{\it Compactness of the space of left orders,}
J. Knot Theory Ramifications 16 (2007), no. 3, 257--266.

\bibitem[DJS03]{DaJaSc1}
{\bf M Davis, T Januszkiewicz, R Scott,}
{\it Fundamental groups of blow-ups,}
Adv. Math. 177 (2003), no. 1, 115--179.

\bibitem[Deh02]{Dehor1}
{\bf P Dehornoy,}
{\it Groupes de Garside,}
Ann. Sci. \'Ecole Norm. Sup. (4) 35 (2002), no. 2, 267--306.

\bibitem[DDG$\phantom{}^+$15]{DDGKM1}
{\bf P Dehornoy, F Digne, E Godelle, D Krammer, J Michel,}
{\it Foundations of Garside theory,}
EMS Tracts Math., 22, European Mathematical Society (EMS), Zürich, 2015. 

\bibitem[DP99]{DehPar1}
{\bf P Dehornoy, L Paris,}
{\it Gaussian groups and Garside groups, two generalisations of Artin groups,}
Proc. London Math. Soc. (3) 79 (1999), no. 3, 569--604.

\bibitem[Dev99]{Devad1}
{\bf S\,L Devadoss,}
{\it Tessellations of moduli spaces and the mosaic operad,}
Contemp. Math., 239, American Mathematical Society, Providence, RI, 1999, 91--114.

\bibitem[DK92]{DucKro1}
{\bf G Duchamp, D Krob,}
{\it The lower central series of the free partially commutative group,}
Semigroup Forum 45 (1992), no. 3, 385--394.

\bibitem[EHK$\phantom{}^+$10]{EHKR1}
{\bf P Etingof, A Henriques, J Kamnitzer, E\,M Rains,}
{\it The cohomology ring of the real locus of the moduli space of stable curves of genus 0 with marked points,}
Ann. of Math. (2) 171 (2010), no. 2, 731--777.

\bibitem[Gen22]{Genev1}
{\bf A Genevois,}
{\it Cactus groups from the viewpoint of geometric group theory,}
Topology Proceedings, to appear, arXiv: 2212.03494, 2022.

\bibitem[Gob24]{Gobet1}
{\bf T Gobet,}
{\it Toric reflection groups,}
J. Aust. Math. Soc. 116 (2024), no. 2, 171--199.

\bibitem[God07]{Godel1}
{\bf E Godelle,}
{\it Parabolic subgroups of Garside groups,}
J. Algebra 317 (2007), no. 1, 1--16.

\bibitem[God10]{Godel2}
{\bf E Godelle,}
{\it Parabolic subgroups of Garside groups II: ribbons,}
J. Pure Appl. Algebra 214 (2010), no. 11, 2044--2062. 

\bibitem[GP12a]{GodPar1}
{\bf E Godelle, L Paris,}
{\it $K(\pi,1)$ and word problems for infinite type Artin-Tits groups, and applications to virtual braid groups,}
Math. Z. 272 (2012), no. 3--4, 1339--1364.

\bibitem[GP12b]{GodPar3}
{\bf E Godelle, L Paris,}
{\it Basic questions on Artin-Tits groups,}
Configuration spaces, 299--311.
CRM Series, 14, Edizioni della Normale, Pisa, 2012.

\bibitem[GP13]{GodPar2}
{\bf E Godelle, L Paris,}
{\it PreGarside monoids and groups, parabolicity, amalgamation, and FC property,}
Internat. J. Algebra Comput. 23 (2013), no. 6, 1431--1467.

\bibitem[Gre90]{Green1}
{\bf E\,R Green,}
{\it Graph products of groups,}
Ph. D. Thesis, University of Leeds, 1990.

\bibitem[HK06]{HenKam1}
{\bf A Henriques, J Kamnitzer,}
{\it Crystals and coboundary categories,}
Duke Math. J. 132 (2006), no. 2, 191--216.

\bibitem[HM95]{HerMei1}
{\bf S Hermiller, J Meier,}
{\it Algorithms and geometry for graph products of groups,}
J. Algebra 171 (1995), no. 1, 230--257.

\bibitem[${\rm IKL}^+$23]{IKLPR1}
{\bf A Ilin, J Kamnitzer, Y Li, P Przytycki, L Rybnikov,}
{\it The moduli space of cactus flower curves and the virtual cactus group,}
Preprint,  	arXiv:2308.06880, 2023. 

\bibitem[Joy99]{Joyce1}
{\bf D Joyce,}
{\it A classifying invariant of knots, the knot quandle,}
J. Pure Appl. Algebra 23 (1982), no. 1, 37--65.

\bibitem[Kau99]{Kauff1}
{\bf L\,H Kauffman,}
{\it Virtual knot theory,}
European J. Combin. 20 (1999), no. 7, 663--690. 

\bibitem[Kau00]{Kauff2}
{\bf L\,H Kauffman,}
{\it A survey of virtual knot theory,}
Knots in Hellas '98 (Delphi), 143--202, Ser. Knots Everything, 24, World Sci. Publ., River Edge, NJ, 2000. 

\bibitem[KW19]{KhWi1}
{\bf A Khoroshkin, T Willwacher,}
{\it Real moduli space of stable rational curves revisted,}
preprint, arXiv:1905.04499, 2019,
to appear in J. Eur. Math. Soc..

\bibitem[KTW04]{KnTaWo1}
{\bf A Knutson, T Tao, C Woodward,}
{\it A positive proof of the Littlewood-Richardson rule using the octahedron recurrence,}
Electron. J. Combin. 11 (2004), no. 1, Research Paper 61, 18 pp.

\bibitem[Los19]{Losev1}
{\bf I Losev,}
{\it Cacti and cells,}
J. Eur. Math. Soc. (JEMS) 21 (2019), no. 6, 1729--1750.

\bibitem[MS17]{McCSul1}
{\bf J McCammond, R Sulway,}
{\it Artin groups of Euclidean type,}
Invent. Math. 210 (2017), no. 1, 231--282.

\bibitem[Mat82]{Matve1}
{\bf S\,V Matveev,}
{\it Distributive groupoids in knot theory,}
Mat. Sb. (N.S.) 119 (161) (1982), no. 1, 78--88, 160.

\bibitem[Mos19]{Mosto1}
{\bf J Mostovoy,}
{\it The pure cactus group is residually nilpotent,}
Arch. Math. (Basel) 113 (2019), no. 3, 229--235.

\bibitem[New42]{Newma1}
{\bf M\,H\,A Newman,}
{\it On theories with a combinatorial definition of ``equivalence'',}
Ann. of Math. (2) 43 (1942), 223--243.

\bibitem[PS21]{PaoSal1}
{\bf G Paolini, M Salvetti,}
{\it Proof of the $K (\pi, 1)$ conjecture for affine Artin groups,}
Invent. Math. 224 (2021), no. 2, 487--572.

\bibitem[Par02]{Paris1}
{\bf L Paris,}
{\it Artin monoids inject in their groups,}
Comment. Math. Helv. 77 (2002), no. 3, 609--637.

\bibitem[PS23]{ParSoe1}
{\bf L Paris, M Soergel,}
{\it Word problem and parabolic subgroups in Dyer groups,}
Bull. Lond. Math. Soc. 55 (2023), no. 6, 2928--2947. 

\bibitem[RW24]{RouWhi1}
{\bf R Rouquier, N White,}
{\it Cactus groups and Lusztig's asymptotic algebra,}
preprint, arXiv:2408.16922, 2024.

\bibitem[Ser77]{Serre1}
{\bf J\,P Serre,}
{\it Arbres, amalgames, SL2,}
Astérisque, No. 46,
Société Mathématique de France, Paris, 1977.

\bibitem[Tit69]{Tits1}
{\bf J Tits,}
{\it Le probl\`eme des mots dans les groupes de Coxeter,}
1969 Symposia Mathema\-ti\-ca (INDAM, Rome, 1967/68), Vol. 1, pp. 175--185, Academic Press, London.

\bibitem[Wyk94]{Wyk1}
{\bf L Van Wyk,}
{\it Graph groups are biautomatic,}
J. Pure Appl. Algebra 94 (1994), no. 3, 341--352.

\bibitem[Yu23]{Yu1}
{\bf R Yu,}
{\it Linearity of generalized cactus groups,}
J. Algebra 635 (2023), 256--270.

\end{thebibliography}
\end{document}